\documentclass[12pt,a4paper]{amsart}
\usepackage{amsfonts,amsmath,amssymb}
\usepackage{hyperref}
\usepackage{latexsym,fullpage,amsfonts,amssymb,amsmath,amscd,graphics,epic}
\usepackage[all]{xy}
\usepackage{amssymb,amsthm,amsxtra}
\usepackage[usenames]{color}
\usepackage{amscd}
\usepackage{amsthm}
\usepackage{amsfonts}
\usepackage{amssymb}
\usepackage{mathrsfs}
\usepackage{url}
\usepackage{bbm}
\usepackage{wasysym}
\usepackage{enumitem}
\usepackage[vcentermath]{youngtab}%
\theoremstyle{plain}
\newtheorem*{theorem*}{Theorem}
\newtheorem*{remark*}{Remark}
\newtheorem*{example*}{Example}
\newtheorem{lemma}{Lemma}[subsection]
\newtheorem{proposition}[lemma]{Proposition}
\newtheorem{corollary}[lemma]{Corollary}
\newtheorem{theorem}[lemma]{Theorem}

\newtheorem*{conjecture*}{Conjecture}

\newtheorem{sublemma}[lemma]{Sublemma}

\newtheorem{introtheorem}{Theorem}

\theoremstyle{definition}
\newtheorem{definition}[lemma]{Definition}

\newtheorem{example}[lemma]{Example}

\theoremstyle{remark}
\newtheorem{remark}[lemma]{Remark}

\newtheorem{obsr}[lemma]{Observation}

\newtheorem{notation}[lemma]{Notation}

 \newcommand{\idealI}{\mathfrak{I}}
\newcommand{\Hom}{\operatorname{Hom}}

\newcommand{\triv}{{\mathbbm 1}}

\newcommand{\id}{\operatorname{Id}}

\renewcommand{\Im}{\operatorname{Im}}
\newcommand{\Ker}{\operatorname{Ker}}
\newcommand{\Ext}{\operatorname{Ext}}
\newcommand{\co}{{\it O}}

\newcommand{\Aut}{{\operatorname{Aut}}}

\newcommand{\End}{\operatorname{End}}

\newcommand{\bC}{{\mathbb C}}
\newcommand{\bZ}{{\mathbb Z}}

\newcommand{\Del}{{\Delta}}

\newcommand{\fh}{{\mathfrak{h}}}

\newcommand{\T}{\tau}

\newcommand{\gl}{{\mathfrak{gl}}}

\newcommand{\abs}[1]{\left|{#1}\right|}

\newcommand{\Dab}{\underline{Rep}^{ab}(S_{\nu})}

\newcommand{\InnaA}[1]{{{#1}}}
\newcommand{\InnaB}[1]{{{#1}}}
\def\quotient#1#2{%
    \raise1ex\hbox{$#1$}\Big/\lower1ex\hbox{$#2$}%
}

\begin{document}

\date{\today}
\title{Schur-Weyl duality for Deligne categories}
 \author{Inna Entova Aizenbud}
\address{Inna Entova Aizenbud,
Massachusetts Institute of Technology,
Department of Mathematics,
Cambridge, MA 02139 USA.}
\email{inna.entova@gmail.com}

\begin{abstract}
 This paper gives an analogue to the classical Schur-Weyl duality in the setting of Deligne categories.
 Given a finite-dimensional unital vector space $V$ (i.e. a vector space $V$ with a \InnaB{distinguished non-zero vector $\InnaB{\triv}$}) we give a definition of a complex tensor power of $V$. This is an $Ind$-object of the Deligne category $\underline{Rep}(S_{t})$ equipped with a natural action of $\gl(V)$.


 This construction allows us to describe a duality between the abelian envelope of the category $\underline{Rep}(S_{t})$ and \InnaB{a localization of} the category $\co^{\mathfrak{p}}_{V}$ (the parabolic category $\co$ for $\gl(V)$ associated with the \InnaB{pair $(V, \InnaB{ \triv})$}).

 In particular, we obtain an exact contravariant functor $\widehat{SW}_{t}$ from the category $\underline{Rep}^{ab}(S_{t})$ (the abelian envelope of the category $\underline{Rep}(S_{t})$) to a certain quotient of the category $\co^{\mathfrak{p}}_{V}$. This quotient, denoted by $\widehat{\co}^{\mathfrak{p}}_{t, V}$, is obtained by taking the full subcategory of $\co^{\mathfrak{p}}_{V}$ consisting of modules of \InnaA{degree} $t$, and \InnaB{localizing} by the subcategory of finite dimensional modules.
 It turns out that the contravariant functor $\widehat{SW}_{t}$ makes $\widehat{\co}^{\mathfrak{p}}_{t, V}$ a Serre quotient of the category $\underline{Rep}^{ab}(S_{t})^{op}$, and the kernel of $\widehat{SW}_{t}$ can be explicitly described.
 
\end{abstract}

\maketitle
\setcounter{tocdepth}{3}
\section{Introduction}

\subsection{}
The study of representations in complex rank involves defining and studying families of abelian categories depending on a parameter $t$ which are polynomial interpolations of the categories of representations of objects such as finite groups, Lie groups, Lie algebras and more. This was done by P. Deligne in \cite{D} for finite dimensional representations of the general linear group $GL_n$, the orthogonal and symplectic groups $O_n, Sp_{2n}$ and the symmetric group $S_n$. Deligne defined Karoubian tensor categories $Rep(GL_t), Rep(OSp_t), \underline{Rep}(S_t)$, $t \in \bC$, which at points $n=t \in \bZ_+$ allow an essentially surjective additive functor onto the standard categories $Rep(GL_n), Rep(OSp_n), Rep(S_n)$. The category $\underline{Rep}(S_t)$ was subsequently studied by himself and others (e.g. by V. Ostrik, J. Comes in \cite{CO}, \cite{CO2}). 

This paper gives an analogue to the classical Schur-Weyl duality in the setting of Deligne categories. In order to do this, we define the ``complex tensor power'' of a finite-dimensional split unital \InnaA{complex} vector space (i.e. a vector space $V$ with a distinguished non-zero vector $\InnaB{\triv}$ and a splitting $V \cong \InnaB{\bC \triv} \oplus U$). This ``complex tensor power'' of $V$ is an $Ind$-object in the category $\underline{Rep}(S_t)$, and comes with an action of $\gl(V)$ on it. \InnaB{Furthermore, it can be shown that this object does not depend on the choice of splitting, but only on the pair $(V, \triv)$.}

The ``$t$-th tensor power'' of $V$ is defined for any $t \in \bC$; for $n=t \in \bZ_+$, the functor $\underline{Rep}(S_{t=n}) \rightarrow Rep(S_n)$ takes this $Ind$-object of $\underline{Rep}(S_{t=n})$ to the usual tensor power $V^{\otimes n}$ in $Rep(S_n)$. Moreover, the action of $\gl(V)$ on the former object corresponds to the action of $\gl(V)$ on $V^{\otimes n}$.

This allows us to define an additive contravariant functor, called the Schur-Weyl functor,
$$SW_{t, V}: \underline{Rep}^{ab}(S_{t}) \rightarrow \co^{\mathfrak{p}}_{V}$$
Here $\underline{Rep}^{ab}(S_{t})$ is the abelian envelope of the category $\underline{Rep}(S_{t})$ (this envelope was described in \cite[Chapter 8]{D}, \cite{CO2}) and the category $\co^{\mathfrak{p}}_{V}$ is the parabolic category $\co$ for $\gl(V)$ associated with the \InnaB{pair $(V, \InnaB{ \triv})$}.

It turns out that $SW_{t, V}$ induces an anti-equivalence of abelian categories between a Serre quotient of $\underline{Rep}^{ab}(S_{t})$ and a \InnaB{localization of} $\co^{\mathfrak{p}}_{V}$. The latter quotient is obtained by taking the full subcategory of $\co^{\mathfrak{p}}_{V}$ consisting of \InnaB{``polynomial''} modules of \InnaA{degree} $t$ (i.e. modules on which $\id_V \in \End(V)$ acts by the scalar $t$, \InnaB{and on which the group $GL(V /\bC \triv)$ acts by polynomial maps), and localizing} by the Serre subcategory of finite dimensional modules. This quotient is denoted by $\widehat{\co}^{\mathfrak{p}}_{t, V}$.

Thus for any unital finite-dimensional space $(V, \triv)$ and for any $t \in \bC$, the category $\widehat{\co}^{\mathfrak{p}}_{t, V}$ is a Serre quotient of $\underline{Rep}^{ab}(S_{t})^{op}$. 

In the upcoming paper \cite{EA}, we will consider the categories $\widehat{\co}^{\mathfrak{p}_N}_{t, \bC^N}$ for $N \in \bZ_+$, and the corresponding Schur-Weyl functors. Defining appropriate restriction functors $$\widehat{Res}_{t, N}:\widehat{\co}^{\mathfrak{p}_N}_{t, \bC^N} \rightarrow \widehat{\co}^{\mathfrak{p}_{N-1}}_{t, \bC^{N-1}}$$ we can consider the inverse limit of the system $((\widehat{\co}^{\mathfrak{p}_N}_{t, \bC^N})_{N \geq 0}, (\widehat{Res}_{t, N})_{N \geq 1})$ and a contravariant functor $$ \underline{Rep}^{ab}(S_{t}) \rightarrow \varprojlim_{N \in \bZ_+} \widehat{\co}^{\mathfrak{p}_N}_{t, \bC^N}$$ induced by the Schur-Weyl functors $SW_{t, \bC^N}$.

We then define a a full subcategory of $ \varprojlim_{N \in \bZ_+} \widehat{\co}^{\mathfrak{p}_N}_{t, \bC^N}$ called ``the filtered inverse limit'' of the system $((\widehat{\co}^{\mathfrak{p}_N}_{t, \bC^N})_{N \geq 0}, (\widehat{Res}_{t, N})_{N \geq 1})$. Intuitively, one can describe the ``the filtered inverse limit'' as follows:

By definition, the objects in $\varprojlim_{N \in \bZ_+} \widehat{\co}^{\mathfrak{p}_N}_{t, \bC^N}$ are sequences $(M_N)_{N\in \bZ_+}$ such that $M_N \in \widehat{\co}^{\mathfrak{p}_N}_{t, \bC^N}$, together with isomorphisms $\widehat{Res}_{t, N}(M_N) \rightarrow M_{N-1}$. The objects in the filtered inverse limit are those sequences $(M_N)_{N\in \bZ_+}$ for which the integer sequence $\{\ell(M_N)\})_{N\in \bZ_+}$ stabilizes ($\ell(M_N)$ is the length of the \InnaA{$\widehat{\co}^{\mathfrak{p}_N}_{t, \bC^N}$-object} $M_N$).

In this setting, we have an anti-equivalence of abelian categories 
$$ \underline{Rep}^{ab}(S_{t}) \longrightarrow \varprojlim_{N \in \bZ_+, \text{  filtered}} \widehat{\co}^{\mathfrak{p}_N}_{t, \bC^N}$$
induced by the Schur-Weyl functors $SW_{t, \bC^N}$. \InnaB{This anti-equivalence gives us an unexpected structure of a rigid symmetric monoidal category on $\varprojlim_{N \in \bZ_+, \text{  filtered}} \widehat{\co}^{\mathfrak{p}_N}_{t, \bC^N}$. In particular, the duality coming from the tensor structure will be the same as the one arising from the usual duality in category $\co$.}

\InnaB{The Schur-Weyl functor described above can also be used to extend other classical dualities to complex rank. Namely, one can consider categories which are constructed ``on the basis of $\underline{Rep}(S_t)$''. A method for constructing such categories was suggested in \cite{E1}, and was used in \cite{EA1}, \cite{Mathew} to study representations of degenerate affine Hecke algebras and of rational Cherednik algebras in complex rank. One can then try to generalize the classical Schur-Weyl dualities for these new categories: for example, one can use the notion of a complex tensor power to construct a Schur-Weyl functor between the category of representations of the degenerate affine Hecke algebra of type $A$ of complex rank, and the category of parabolic-type representations of the Yangian $Y(\gl_N)$ for $N \in \bZ_+$. We plan to study these dualities in detail in the future.
}
\subsection{Summary of results}

Recall that the classical Schur-Weyl duality describes the relation between the actions of $\gl(V)$, $S_d$ on $V^{\otimes d}$ (here $V$ is a finite-dimensional \InnaA{complex} vector space, $d$ is a non-negative integer, and $S_d$ is the symmetric group). 

In particular, it says that the actions of $\gl(V)$, $S_d$ on $V^{\otimes d}$ commute with each other, and we have a decomposition of \InnaA{$\bC[S_d] \otimes_{\bC} \mathcal{U}(\gl(V))$-modules}
$$ V^{\otimes d} \cong \bigoplus_{\substack{\lambda \text{ is a Young diagram }\\ \abs{\lambda} =d}} \lambda \otimes S^{\lambda} V $$

We would like to extend this duality to the Deligne category $\underline{Rep}(S_t)$, by constructing an object $V^{\underline{\otimes}  t}$ in $\underline{Rep}(S_t)$, together with the action of $\gl(V)$ on it, which is an analogue (a polynomial interpolation) of the module $V^{\otimes d}$ for \InnaA{$\bC[S_d] \otimes_{\bC} \mathcal{U}(\gl(V))$}.

It turns out that this can be done in the following setting: 
\begin{itemize}
 \item The space $V$ is required to be unital, that is, we fix a \InnaB{distinguished non-zero vector $\triv$ in $V$. We then choose any splitting $V \cong \InnaB{\bC \triv} \oplus U$. It can be shown that $V^{\underline{\otimes}  t}$ does not depend on the choice of the splitting, but only on the choice of the distinguished vector $\triv$.
 
 For $t \notin \bZ_+$, one can actually give a definition without choosing a splitting, as it is done by P. Etingof in \cite{E1} (c.f. Appendix \ref{app:complex_ten_power_generic_nu}).}
 \item The object $V^{\underline{\otimes}  t}$ is not finite-dimensional (unlike $V^{\otimes d}$), but is an $Ind$-object (a countable direct sum) of objects from $\underline{Rep}(S_t)$.
\end{itemize}
 
The intuition for working in the above setting is as follows (proposed by P. Etingof in \cite{E1}): let $t \in \bC$, and let $x$ be a formal variable. The expression $ x^t$ is not polynomial in $t$, and has no algebraic meaning, but if we present $x$ as $x := 1+y$, we can write:
$$x^t= (1+y)^t = \sum_{k \in \bZ_+} \binom{t}{k} y^k$$ The function $\binom{t}{k}$ is polynomial in $t$, so the expression $\sum_{k \in \bZ_+} \binom{t}{k} y^k$ is just a formal power series with polynomial coefficients. This explains why it is convenient to work with unital vector spaces. 

Notice that for $t \notin \bZ_+$, the above sum is infinite, which also explains why the object $V^{\underline{\otimes}  t}$ can be expected to be an $Ind$-object of $\underline{Rep}(S_t)$ ($\underline{Rep}(S_t)$ has only finite direct sums).

To define the object $V^{\underline{\otimes}  t}$ for a \InnaB{unital vector space $(V, \triv)$}, we use the following notation:

\begin{notation}\label{notn:par_subalg}
\mbox{}
\begin{itemize}
 \item \InnaB{We denote by $\mathfrak{p}_{(V, \InnaB{\bC \triv})}\subset \mathfrak{gl}(V)$ the parabolic Lie subalgebra which consists of} all the endomorphisms $\phi: V \rightarrow V$ for which $\InnaB{\phi(\triv) \in \InnaB{\bC \triv} }$. We will write $\mathfrak{p}:=\mathfrak{p}_{(V, \InnaB{\bC \triv})}$ for short.
 \item \InnaB{$\bar{\mathfrak{P}}_{\triv}$ denotes the mirabolic subgroup corresponding to $\triv$, i.e. the group of automorphisms $\Phi: V \rightarrow V$ such that $\Phi(\triv) = \triv$, and $\bar{\mathfrak{p}}_{\InnaB{\bC \triv}} \subset \mathfrak{p}$ denotes the algebra of endomorphisms $\phi: V \rightarrow V$ for which $\phi(\triv) =0$ (thus $\bar{\mathfrak{p}}_{\InnaB{\bC \triv}} = Lie(\bar{\mathfrak{P}}_{\triv})$).}
 \item \InnaB{$\mathfrak{U}_{\triv}$ denotes the subgroup of $\bar{\mathfrak{P}}_{\triv}$ of automorphisms $\Phi: V \rightarrow V$ for which $Im(\Phi - \id_V) \subset \bC \triv$, and $\mathfrak{u}_{\mathfrak{p}}^{+} \subset \bar{\mathfrak{p}}_{\InnaB{\bC \triv}}$ denotes the algebra of endomorphisms $\phi: V \rightarrow V$ for which $\Im \phi \subset \InnaB{\bC \triv} \subset \Ker \phi$ (thus $\mathfrak{u}_{\mathfrak{p}}^{+} = Lie(\mathfrak{U}_{\triv})$). }
\end{itemize}
\end{notation}

Fix a splitting $V= \InnaB{\bC \triv} \oplus U $.

Recall that we have a splitting $\gl(V) \cong \mathfrak{p} \oplus \mathfrak{u}_{\mathfrak{p}}^{-}$, where $\mathfrak{u}_{\mathfrak{p}}^{-} \cong U$. This gives us an analogue of triangular decomposition:
$$\gl(V) \cong \bC   \id_V \oplus \mathfrak{u}_{\mathfrak{p}}^{-} \oplus \mathfrak{u}_{\mathfrak{p}}^{+}  \oplus \gl(U)$$
with $\mathfrak{u}_{\mathfrak{p}}^{+} \cong U^*$.

The definition of $V^{\underline{\otimes}  t}$ is essentially an analogue of the isomorphism of 

\InnaA{$\bC[S_d] \otimes_{\bC} \mathcal{U}(\gl(V))$}-modules
$$V^{\otimes d} \cong \bigoplus_{k=0,...,d} (U^{\otimes k} \otimes \bC Inj(\{1,...,k\} , \{1,...,d\}))^{S_k} $$
Here the action of $\gl(V)$ on the right hand side (viewed as a $\bZ_+$-graded space) \InnaA{is given as follows}:
$\id_V$ acts by the scalar $t$, $\gl(U)$ acts on each summand \InnaA{though its action on the spaces $U^{\otimes k}$}, and $\mathfrak{u}_{\mathfrak{p}}^{-}, \mathfrak{u}_{\mathfrak{p}}^{+}$ act by operators of degrees $1, -1$ respectively.

\InnaB{The group $S_d$ acts on each summand through its action on the set $Inj(\{1,...,k\} , \{1,...,d\})$ of injective maps from $\{1,...,k\} $ to $ \{1,...,d\}$.}

In the Deligne category $\underline{Rep}(S_{t})$, we have objects $\Del_k$ which are analogues of the $S_k \times S_d$ representation $\bC Inj(\{1,...,k\} , \{1,...,d\})$. The objects $\Del_k$ carry an action of $S_k$, therefore we can define a $\bZ_+$-graded $Ind$-object of $\underline{Rep}(S_{t})$:
$$ V^{\underline{\otimes}  t} := \bigoplus_{k \geq 0} (U^{\otimes k} \otimes \Del_k)^{S_k}$$

Next, one can define the action of $\gl(V)$ on $ V^{\underline{\otimes}  t}$ so that $\id_V$ acts by scalar $t$, $\gl(U)$ acts naturally on each summand $(U^{\otimes k} \otimes \Del_k)^{S_k}$, and $\mathfrak{u}_{\mathfrak{p}}^{-}, \mathfrak{u}_{\mathfrak{p}}^{+}$ act by operators of degrees $1, -1$ respectively.

\InnaB{In fact, it can be shown that the object $V^{\underline{\otimes}  t}$ does not depend on the choice of the splitting, but only on the choice of the distinguished vector $\triv$.}

\InnaA{We also show that for $t=n \in \bZ_+$, the functor $\underline{Rep}(S_{t=n}) \rightarrow Rep(S_n)$ takes $V^{\underline{\otimes}  t=n}$ to the usual tensor power $V^{\otimes n}$ in $Rep(S_n)$, and the action of $\gl(V)$ on $V^{\underline{\otimes}  t=n}$ corresponds to the action of $\gl(V)$ on $V^{\otimes n}$.}

\begin{remark}
 \InnaA{The Hilbert series of $V^{\underline{\otimes}  t}$ corresponding to the grading $gr_0(V):= \InnaB{\bC \triv}$, $gr_1(V):= U$ would be $(1+y)^t$.}
\end{remark}

\begin{remark}
 \InnaA{Given any symmetric monoidal category $\mathcal{C}$ with unit object $\mathbf{1}$ and a fixed object $X \in \mathcal{C}$, one can similarly define the object $(\mathbf{1} \oplus X)^{\underline{\otimes}  t}$ of $Ind-(\underline{Rep}(S_t) \boxtimes \mathcal{C})$.}
\end{remark}

%
%
$$ $$

We now proceed to the second part of the Schur-Weyl duality. Recall that in the classical Schur-Weyl duality for $\gl(V)$, $S_d$, the module $V^{\otimes d}$ over \InnaA{$\bC[S_d] \otimes_{\bC} \mathcal{U}(\gl(V))$} defines a contravariant functor 
\begin{align*}
 &\mathtt{SW}_{d, V}: Rep(S_d) \longrightarrow Mod_{\mathcal{U}(\gl(V)), poly} \\
 &\mathtt{SW}_{d, V} := \Hom_{S_d}(\cdot, V^{\otimes d})
\end{align*}
 Here 
 \begin{itemize}
  \item The category $Rep(S_d)$ is the semisimple abelian category of finite-dimensional representations of $S_d$. 
  \item The category $Mod_{\mathcal{U}(\gl(V)), poly}$ is the semisimple abelian category of polynomial representations of $\gl(V)$ (``polynomial'' meaning that these are \InnaA{direct summands of finite direct sums of tensor powers of $V$; \InnaB{alternatively, one can define these as finite-dimensional representations $GL(V) \rightarrow \Aut(W)$ which can be extended to an algebraic map $\End(V) \rightarrow \End(W)$})}.
 \end{itemize}

This functor takes the simple representation of $S_d$ corresponding to the Young diagram $\lambda$ either to zero, or to the simple representation $S^{\lambda} V$ of $\gl(V)$. 
Notice that the image of functor $ \mathtt{SW}_{d, V}$ lies in the full additive subcategory $Mod_{\mathcal{U}(\gl(V)), poly, d}$ of $Mod_{\mathcal{U}(\gl(V)), poly}$ whose objects are $\gl(V)$-modules on which $\id_V$ acts by the scalar $d$.

It is then easy to see that the contravariant functor $ \mathtt{SW}_{d, V}: Rep(S_d) \rightarrow Mod_{\mathcal{U}(\gl(V)), poly, d}$ is full and essentially surjective.

Considering these dualities for a fixed finite-dimensional vector space $V$ and every $d \in \bZ_+$, we can construct a full, essentially surjective, additive contravariant functor 
$$ \mathtt{SW}_V: \bigoplus_{d \in \bZ_+} Rep(S_d) \rightarrow Mod_{\mathcal{U}(\gl(V)), poly} \cong \bigoplus_{d \in \bZ_+} Mod_{\mathcal{U}(\gl(V)), poly, d}$$
between semisimple abelian categories.

The simple objects in $\bigoplus_{d \in \bZ_+} Rep(S_d)$ which $\mathtt{SW}_V$ sends to zero are (up to isomorphism) exactly those parametrized by Young diagrams $\lambda$ such that $\lambda$ has more than $\dim V $ rows. 

Thus the contravariant functor $ \mathtt{SW}_{V}$ induces an anti-equivalence of abelian categories between a Serre quotient of the semisimple abelian category $\bigoplus_{d \in \bZ_+}  Rep(S_d)$, and the semisimple abelian category $Mod_{\mathcal{U}(\gl(V)), poly}$.
$$ $$

In our case, we would like to consider the Deligne category $\underline{Rep}(S_t)$ and a category of representations of $\gl(V)$ related to the unital structure of $V$. 

Unfortunately, the Deligne category $\underline{Rep}(S_t)$ is Karoubian but not necessarily abelian, which would make it difficult to obtain an anti-equivalence of abelian categories. However, it turns out that the Karoubian tensor category $\underline{Rep}(S_t)$ is abelian semisimple whenever $t \notin \bZ_+$. For $t=d \in \bZ_+$, this is not the case, but then $\underline{Rep}(S_{t=d})$ can be embedded (as a Karoubian \InnaB{tensor} category) into a larger abelian tensor category, denoted by $\underline{Rep}^{ab}(S_{t=d})$. The construction of this abelian envelope is discussed in detail in \cite[Section 8]{D} and in \cite{CO2}. We will denote by $\underline{Rep}^{ab}(S_{t})$ the abelian envelope of $\underline{Rep}(S_t)$ for any $t \in \bC$, with $\underline{Rep}^{ab}(S_{t})$ being just $\underline{Rep}(S_{t})$ whenever $t \notin \bZ_+$.

The structure of $\underline{Rep}^{ab}(S_{t})$ as an abelian category is known, and described in \cite{CO2} and in Section \ref{ssec:S_nu_abelian_env}. In particular, it is a highest weight category \InnaA{(with infinitely many weights)}, with simple objects parametrized by all Young diagrams.
$$ $$
It turns out that the correct categories to consider for the Schur-Weyl duality in complex rank are the abelian category $\underline{Rep}^{ab}(S_t)$ and the parabolic category $\co$ for $\gl(V)$ corresponding to the \InnaB{pair $(V, \InnaB{ \triv})$}. 
\InnaB{
$$ $$
Consider the short exact sequence of groups $$ 1 \rightarrow \mathfrak{U}_{\triv} \longrightarrow \bar{\mathfrak{P}}_{\triv} \longrightarrow GL \left( \quotient{V}{\bC \triv} \right) \rightarrow 1$$
For any irreducible finite-dimensional algebraic representation $\rho: \bar{\mathfrak{P}}_{\triv} \rightarrow \Aut(E)$ of the mirabolic subgroup, $\mathfrak{U}_{\triv}$ acts trivially on $E$, and thus $\rho$ factors through $GL \left( V / \bC \triv \right)$. 

This allows us to say that $\rho$ is a {\it $GL \left( V / \bC \triv \right)$-polynomial representation of $\bar{\mathfrak{P}}_{\triv}$} if 

$\rho: GL \left( V / \bC \triv \right) \rightarrow \Aut(E)$ is a polynomial representation (i.e. if $\rho$ extends to an algebraic map $\End \left( V / \bC \triv \right) \rightarrow \End(E)$).
 
Now, for any finite-dimensional algebraic representation $E$ of $\bar{\mathfrak{P}}_{\triv}$, we say that $E$ is {\it $GL \left( V / \bC \triv \right)$-polynomial} if the Jordan-Holder components of $E$ are $GL \left( V / \bC \triv \right)$-polynomial representations of $\bar{\mathfrak{P}}_{\triv}$.
 
This allows us to give the following definition:
}
\begin{definition}
 The category $\co^{\mathfrak{p}}_{t, V}$ is defined to be the full subcategory of $Mod_{\mathcal{U}(\gl(V))}$ whose objects $M$ satisfy the following conditions:
 \begin{itemize}
 \item \InnaB{$M$ is a Harish-Chandra module for the pair $(\gl(V), \bar{\mathfrak{P}}_{\triv})$, i.e. the action of the Lie subalgebra ${\mathfrak{p}}_{\InnaB{\bC \triv}}$ on $M$ integrates to the action of the group $\bar{\mathfrak{P}}_{\triv}$.
 
 Furthermore, we require that as a representation of $\bar{\mathfrak{P}}_{\triv}$, $M$ be a filtered colimit of $GL \left( V / \bC \triv \right)$-polynomial representations, i.e. $$M \rvert_{\bar{\mathfrak{P}}_{\triv}} \in Ind-Rep(\bar{\mathfrak{P}}_{\triv})_{GL \left( V / \bC \triv \right)-poly}$$
 }
  \item $M$ is a finitely generated $\mathcal{U}(\gl(V))$-module.
  \item $\id_V \in \gl(V)$ acts by $t \id_M$ on $M$.
 \end{itemize}
\end{definition}
\InnaB{
\begin{remark}
 For any fixed splitting $V = \bC \triv \oplus U$, the first condition can be replaced by the requirement that $M \rvert_{\gl(U)}$ be a direct sum of polynomial simple $\mathcal{U}(\gl(U))$-modules, and that $\mathfrak{u}_{\mathfrak{p}}^{+}$ act locally finitely on $M$.
\end{remark}
}

The category $\co^{\mathfrak{p}}_{t, V}$ is an Artinian abelian category, and is a Serre subcategory of the usual category $\co$ for $\gl(V)$.

The $\gl(V)$-action on the object $V^{\underline{\otimes}  t}$ is a ``$\co^{\mathfrak{p}}_{t}$-type'' action, which allows us to define a contravariant functor from $\underline{Rep}^{ab}(S_t)$ to $\co^{\mathfrak{p}}_{t, V}$:
$$SW_{t, V} :=  \Hom_{\underline{Rep}^{ab}(S_t)}(\cdot, V^{\underline{\otimes}  t})$$

This contravariant functor is linear and additive, yet \InnaA{only left} exact. To fix this problem, we compose this functor with the quotient functor \InnaB{$\hat{\pi}$} from $\co^{\mathfrak{p}}_{t, V}$ to the category $\widehat{\co}^{\mathfrak{p}}_{t, V}$: the localization of $\co^{\mathfrak{p}}_{t, V}$ by the Serre subcategory of finite-dimensional modules. \InnaB{We denote the newly obtained functor by $\widehat{SW}_{t, V}$.}

One of the main results of this paper is the following theorem (c.f. Theorem \ref{thrm:SW_almost_equiv}):
\begin{introtheorem}\label{introthrm:SW_almost_equiv}
The contravariant functor $\widehat{SW}_{t, V}:\underline{Rep}^{ab}(S_{t}) \rightarrow {\widehat{\co}^{\mathfrak{p}}_{t, V}}$ is exact and essentially surjective.

Moreover, the induced contravariant functor $$ \quotient{\underline{Rep}^{ab}(S_{t})}{Ker(\widehat{SW}_{t, V})} \rightarrow \widehat{\co}^{\mathfrak{p}}_{t, V}$$ is an anti-equivalence of abelian categories, thus making $\widehat{\co}^{\mathfrak{p}}_{t, V}$ a Serre quotient of $\underline{Rep}^{ab}(S_{t})^{op}$.
\end{introtheorem}

\InnaB{We also show that the duality in $\underline{Rep}^{ab}(S_{t})$ (given by the tensor structure) corresponds to the duality in the category $\widehat{\co}^{\mathfrak{p}}_{t, V}$ (c.f. \cite[Chapter 3.2]{H}), i.e. that there is an isomorphism of (covariant) functors $$ \widehat{SW}_{t, V}(( \cdot )^*) \longrightarrow \hat{\pi}(SW_{t, V}( \cdot )^{\vee})$$ 

This gives a new interpretation to the notion of duality in the category $\widehat{\co}^{\mathfrak{p}}_{t, V}$.
}

%
%

\subsection{Structure of the paper}
In Section \ref{sec:class_Schur_Weyl}, we recall the classical Schur-Weyl duality for $\gl(V)$, $S_d$.

In Section \ref{sec:Del_cat_S_nu}, we recall the necessary information about the Deligne category $\underline{Rep}(S_{\nu})$ (thoughout the paper, we use the parameter $\nu$ instead of $t$), as well as the abelian structure of the abelian envelope $\underline{Rep}^{ab}(S_{\nu})$ of $\underline{Rep}(S_{\nu})$.

In Section \ref{sec:par_cat_o}, we recall the necessary information about the parabolic category $\co$ for $\gl(V)$.

In Section \ref{sec:comp_tens_power}, we define the complex tensor power $V^{\underline{\otimes}  \nu}$ of a split unital vector space $V$.
Subsection \ref{ssec:def_cat_Del_dash_par_O} gives the formal framework needed to define the object $V^{\underline{\otimes}  \nu}$; Subsection \ref{ssec:def_comp_ten_power_split} contains the definition of $V^{\underline{\otimes}  \nu}$ and lists some properties of this object (part of the proof that $V^{\underline{\otimes}  \nu}$ is well defined is contained in Appendix \ref{app:complex_ten_power_action_well_def}); Subsection \ref{ssec:comp_tens_power_compatible_classical} shows that for $\nu =n \in \bZ_+$, our definition is compatible (via the functor $\underline{Rep}(S_{\nu=n}) \rightarrow Rep(S_n) $) with the usual notion of a tensor power of $V$; \InnaB{Subsection \ref{ssec:comp_tens_power_indep_splitting} shows that the object $V^{\underline{\otimes}  \nu}$ does not actually depend on the splitting of $V$, but rather only on the pair $(V, \triv)$}.

In Section \ref{sec:SW_functor}, we define the Schur-Weyl contravariant functors
$SW_{\nu, V} : \underline{Rep}^{ab}(S_{\nu}) \rightarrow \co^{\mathfrak{p}}_{\nu, V}$ and $\widehat{SW}_{\nu, V} : \underline{Rep}^{ab}(S_{\nu}) \rightarrow \widehat{\co}^{\mathfrak{p}}_{\nu, V}$, and prove Theorem \ref{introthrm:SW_almost_equiv}. The case of a semisimple block of the category $\Dab$ is discussed in Subsection \ref{ssec:SW_duality_ss_block} and the case of a non-semisimple block is discussed in Subsection \ref{ssec:SW_duality_non_ss_block}. \InnaB{Subsection \ref{ssec:dualities_SW_commute} shows how the functor $SW_{\nu, V}$ relates the duality in $\underline{Rep}^{ab}(S_{\nu})$ to the duality in the category $\widehat{\co}^{\mathfrak{p}}_{\nu, V}$.}

Appendix \ref{app:complex_ten_power_generic_nu} gives an alternative definition (due to P. Etingof) of the object $V^{\underline{\otimes}  \nu}$ when $\nu \notin \bZ_+$. This definition requires $V$ to be unital, but not split. We show that this definition coincides with our definition once we fix a splitting.

\subsection{Acknowledgements}
I would like to thank my advisor, Pavel Etingof, for guiding me thoughout this project.

\section{Notation and definitions}\label{sec:notation}
The base field throughout the paper will be $\bC$.
\subsection{Tensor categories}
\InnaA{The following standard notation will be used thoughout the paper: }
\begin{notation}
Let $\mathcal{C}$ be a rigid symmetric monoidal category. We denote by $\mathbf{1}$ the unit object. Also, for any object $M$, we denote by $M^*$ the dual of $M$.
\end{notation} 
\subsection{Karoubian categories}
\begin{definition}[Karoubian category]
We will call a category $\mathcal{A}$ Karoubian\footnote{Deligne calls such categories "pseudo-abelian" (c.f. \cite[1.9]{D}). } if it is an additive category, and every idempotent morphism is a projection onto a direct factor.
\end{definition}

\begin{definition}[Block of a Karoubian category]
 A block in an Karoubian category is a full subcategory generated by an equivalence class of indecomposable objects, defined by the minimal equivalence relation such that any two indecomposable objects with a non-zero morphism between them are equivalent.
\end{definition}

\subsection{Serre subcategories and quotients}
\begin{definition}[Serre subcategory]
A (nonempty) full subcategory $\mathcal{C}$ of an abelian category $\mathcal{A}$ is called a {\it Serre subcategory} if for any exact sequence
$$ 0 \rightarrow M' \rightarrow M \rightarrow M'' \rightarrow 0$$
$M$ is in $\mathcal{C}$ iff $M'$ and $M''$ are in $\mathcal{C}$.
\end{definition}

\begin{definition}[Serre quotient]
Let $\mathcal{A}$ be an abelian category, and $\mathcal{C}$ be a Serre subcategory.

We define the category $\mathcal{A} /{\mathcal{C}}$, called {\it the Serre quotient} of $\mathcal{A}$ by $\mathcal{C}$, whose objects are the objects of $\mathcal{A}$ and where the morphisms are defined by
$$ \Hom_{\mathcal{A} /{\mathcal{C}}}(X,Y):= \varinjlim_{\substack{X' \subset X, Y' \subset Y \\ X/X', Y' \in \mathcal{C}}} \Hom_{\mathcal{A}}(X',Y/Y') $$

The category $\mathcal{A} /{\mathcal{C}}$ comes with a quotient functor, $\pi: \mathcal{A} \rightarrow \mathcal{A} /{\mathcal{C}}$, which takes $X \in \mathcal{A}$ to $X \in \mathcal{A} /{\mathcal{C}}$, and $f: X \rightarrow Y$ in $\mathcal{A}$ to its image in $ \varinjlim_{\substack{X' \subset X, Y' \subset Y \\ X/X', Y' \in \mathcal{C}}} \Hom_{\mathcal{A}}(X',Y/Y') $.
\end{definition}

\begin{remark}
 It is easy to see that the category $\mathcal{A} /{\mathcal{C}}$ is abelian, and the functor $\pi:  \mathcal{A} \rightarrow \mathcal{A} /{\mathcal{C}}$ is exact.
\end{remark}

Let $\mathcal{A}, \mathcal{B}$ be abelian categories, and $\mathcal{F}:\mathcal{A} \rightarrow \mathcal{B}$ an exact functor. Then we can consider the full subcategory $Ker(\mathcal{F})$ of $\mathcal{A}$ whose objects are $X \in \mathcal{A}$ for which $\mathcal{F}(X)=0$. 

Then $Ker(\mathcal{F})$ is a Serre subcategory, and the functor $\mathcal{F}$ factors through the functor $\pi: \mathcal{A} \rightarrow \mathcal{A} /{Ker(\mathcal{F})}$: we have a functor \InnaA{$$\bar{\mathcal{F}}:\mathcal{A} /{Ker(\mathcal{F})} \rightarrow \mathcal{B} \, \text{ such that } \, \mathcal{F} = \bar{\mathcal{F}} \circ \pi$$}

One can easily check that the functor $\bar{\mathcal{F}}:\mathcal{A} /{Ker(\mathcal{F})} \rightarrow \mathcal{B} $ is exact and faithful.

\begin{remark}
 Let $\mathcal{A}$ be an abelian category, and $\mathcal{C}$ be a Serre subcategory. Consider the quotient functor, $\pi: \mathcal{A} \rightarrow \mathcal{A} /{\mathcal{C}}$. Then $Ker(\pi) = \mathcal{C}$, and any exact functor $\mathcal{F}:\mathcal{A} \rightarrow \mathcal{B}$ which takes all the objects of $\mathcal{C}$ to zero factors through $\pi$.
\end{remark}

\subsection{Ind-completion of categories}
Let $\mathcal{A}$ be a small category.
\begin{definition}[Ind-completion]
 The Ind-completion of $\mathcal{A}$, denoted by $Ind-\mathcal{A}$, is the full subcategory of the category $Fun(\mathcal{A}^{op}, \mathbf{Set})$, whose objects are functors which are filtered colimits of representable functors $\mathcal{A}^{op} \rightarrow \mathbf{Set}$.
\end{definition}
\begin{remark}
 The Yoneda lemma gives us a fully faithful functor $\jmath: \mathcal{A} \rightarrow Fun(\mathcal{A}^{op}, \mathbf{Set})$ which restricts to a fully faithful functor $\iota: \mathcal{A} \rightarrow Ind-\mathcal{A}$.
\end{remark}

An easy consequence of the definition is the following Lemma:

\begin{lemma}
 The objects of $\iota(\mathcal{A})$ are compact objects in $Ind-\mathcal{A}$.
\end{lemma}

\begin{corollary}
 Given an object $A \in \mathcal{A}$ and a collection of objects $\{A_i\}_{i \in I}, A_i \in \mathcal{A}$ (here $I$ is a discrete set), we have: $$\Hom_{Ind-\mathcal{A}}(A, \bigoplus_{i \in I} A_i) \cong \bigoplus_{i \in I} \Hom_{\mathcal{A}}(A,  A_i)$$
\end{corollary}

We will also use the following property of the Ind-completion (c.f. \cite[Theorem 8.6.5, p.194]{KS}):
\begin{theorem}\label{thrm:ind_completion_prop}
 Assume the category $\mathcal{A}$ is abelian. Then the category $Ind-\mathcal{A}$ is abelian as well, and the functor $\iota$ is exact. Furthermore, the category $Ind-\mathcal{A}$ is a Grothendieck category (in the sense of \cite[Definition 8.3.24, p.186]{KS}), and thus any functor $\mathcal{F} : Ind-\mathcal{A} \rightarrow \mathcal{C}$ commuting with small colimits admits a right adjoint.
\end{theorem}

\subsection{Actions on tensor powers of a vector space}
Let $U$ be a vector space over $\bC$, and let $k \geq 0$.
\begin{notation}\label{notn:action_on_tes_powers}
\mbox{}
\begin{enumerate}
 \item Let $A \in \End(U)$. We denote the operator $\id_U \otimes \id_U \otimes... \otimes A \otimes ...\otimes \id_U$ on $U^{\otimes k}$ (with $A$ acting on the $i$-th factor of the tensor product) by $A^{(i)}$.
 
The diagonal action of $A$ on $U^{\otimes k}$ would then be $$\sum_{1 \leq i \leq n} A^{(i)} = A \otimes \id_U \otimes...\otimes \id_U + \id_U \otimes A \otimes...\otimes \id_U +...+ \id_U \otimes...\otimes \id_U \otimes A$$ and will sometimes be denoted by $A \rvert_{U^{\otimes k}}$.

\item Similarly, given a functional $f \in U^*$, we have an operator $f^{(l)}$ defined as
\begin{align*}
  f^{(l)}: U^{\otimes k} &\rightarrow U^{\otimes k-1}\\
u_1 \otimes ...\otimes u_k &\mapsto f(u_l)u_1 \otimes ...\otimes u_{l-1} \otimes u_{l+1} \otimes ...\otimes u_k
\end{align*}
\item Finally, given $u \in U$, we define the operator $u^{(l)}$ as
\begin{align*}
  u^{(l)}: U^{\otimes k} &\rightarrow U^{\otimes k+1} \\
u_1 \otimes ...\otimes u_k &\mapsto u_1 \otimes ...\otimes u_{l-1} \otimes u \otimes u_{l} \otimes ...\otimes u_k
\end{align*}

\end{enumerate}
\end{notation}
\begin{notation}
 Let $U$ be a finite-dimensional vector space, and let $f \in U^*$, $u \in U$. Denote by $T_{f, u} \in \End(U)$ the rank one operator $v_1 \mapsto f(v_1) u$ (i.e. the image of $f \otimes u$ under the isomorphism $U^* \otimes U \rightarrow \End(U)$).
\end{notation}
\begin{notation}
 Let $\lambda$ be a Young diagram. Denote by $S^{\lambda}$ the Schur functor corresponding to $\lambda$ (c.f. \cite[Chapter 6]{FH}). \InnaA{When applied to a finite-dimensional vector space $U$,} this is \InnaA{either zero (iff $l(\lambda) > \dim U$), or} an irreducible finite-dimensional representation of the Lie algebra $\gl(U)$\InnaB{, which integrates to a representation of the group $GL(U)$}.

 We will denote the full additive subcategory of $Mod_{\mathcal{U}(\gl(U))}$ generated by $\{S^{\lambda} U\}_{\lambda}$ ($\lambda$ running over all Young diagrams) by $Mod_{\mathcal{U}(\gl(U)), poly}$, and call its objects polynomial representations of the Lie algebra $\gl(U)$ \InnaB{(or the algebraic group $GL(U)$)}.

\end{notation}
The category $Mod_{\mathcal{U}(\gl(U)), poly}$ is obviously a semisimple abelian category, and it contains all the finite-dimensional representations of $\gl(U)$ which can be obtained as submodules of a direct sum of tensor powers of the tautological representation $U$ of $\gl(U)$. 

\InnaB{Alternatively, one can describe these representations as finite-dimensional representations $\rho: GL(U) \rightarrow \Aut(W)$ which can be extended to an algebraic map $\End(U) \rightarrow \End(W)$.}

\subsection{Symmetric group and Young diagrams}
\begin{notation}
\mbox{}
 \begin{itemize}[leftmargin=*]
\item $S_n$ will denote the symmetric group ($n \in \bZ_+$).

 \item The notation $\lambda$ will stand for a partition (weakly decreasing sequence of non-negative integers), a Young diagram $\lambda$, and the corresponding irreducible representation of $S_{\abs{\lambda}}$. Here $\abs{\lambda}$ is the sum of entries of the partition, or, equivalently, the number of cells in the Young diagram $\lambda$.

 \item \InnaA{All the Young diagrams will be considered in the English notation, i.e. the lengths of the rows decrease from top to bottom.}

\item The length of the partition $\lambda$, i.e. the number of rows of Young diagram $\lambda$, will be denoted by $\ell(\lambda)$.

\item The $i$-th entry of a partition $\lambda$, as well as the length of the $i$-th row of the corresponding Young diagram, will be denoted by $\lambda_i$ (if $i>\ell(\lambda)$, then $\lambda_i :=0$). 
%

\item $\fh$ (in context of representations of $S_n$) will denote the permutation representation of $S_n$, i.e. the $n$-dimensional representation $\bC^n$ with $S_n$ acting by $g.e_j =e_{g(j)}$ on the standard basis $e_1, .., e_n$ of $\bC^n$.

%

\item \InnaA{For any Young diagram $\lambda$ and an integer $n$ such that $n \geq \abs{\lambda} +\lambda_1$, we denote by $\tilde{\lambda}(n)$ the Young diagram obtained by adding a row of length $n - \abs{\lambda}$ on top of $\lambda$.}

\item Let $\mathcal{I}^{m,+}_{\lambda}$ denote the set of all Young diagrams obtained from $\lambda$ by adding $m$ boxes, no two in the same column, and $\mathcal{I}^{m,-}_{\lambda}$ denote the set of all Young diagrams obtained from $\lambda$ by removing $m$ boxes, no two in the same column. \InnaB{We will also denote}: $\mathcal{I}^{+}_{\lambda} : = \uplus_{ m\geq 0} \mathcal{I}^{m,+}_{\lambda}$, $\mathcal{I}^{-}_{\lambda} : = \uplus_{ 0 \leq m\leq \abs{\lambda}} \mathcal{I}^{m,-}_{\lambda}$.

\end{itemize}

\end{notation}
\InnaA{

\begin{example}
Consider the Young diagram $\lambda$ corresponding to the partition $(6,5,4,1)$:
 $$\yng(6,5,4,1) $$ 
 The length of $\lambda$ is $4$, \InnaB{and $\abs{\lambda}=16$}. For $n = 23$, we have:
 $$\tilde{\lambda}(n) = \yng(7,6,5,4,1)$$
\end{example}
}
\section{Classical Schur-Weyl duality}\label{sec:class_Schur_Weyl}
In this section we give a short overview of the classical Schur-Weyl duality.

Let $V$ be a finite-dimensional vector space over $\bC$, and let $E:= V^{\otimes d}$. Then $S_d$ acts on $E$ by permuting the factors of the tensor product (the action is semisimple, by Mashke's theorem):\InnaA{
$$\sigma.(v_1 \otimes v_2 \otimes...\otimes v_d) := v_{\sigma^{-1}(1)} \otimes v_{\sigma^{-1}(2)} \otimes...\otimes v_{\sigma^{-1}(d)}$$}
Denote by $A$ the image of $\bC[S_d]$ in $\End_{\bC}(E)$.

Since $\bC[S_d]$ is semisimple by Mashke's theorem, we have the following corollary of the Double Centralizer Theorem:

\begin{proposition}
Let $B: = \End_A(E)$. Then
\begin{itemize}
 \item $B$ is semisimple.
\item $A = \End_B(E)$.
\item As an \InnaA{$A \otimes_{\bC} B$-module}, $E$ decomposes as $$E =\bigoplus_i V_i \otimes W_i$$
where $V_i$ are all the irreducible representations of $A$, and $W_i$ are all the irreducible representations of $B$. In particular, there is a bijection between the sets of non-isomorphic irreducible representations of $A$ and $B$.

\end{itemize}
\end{proposition}

Consider the diagonal action of the Lie algebra $\mathfrak{gl}(V)$ on $E$ (i.e. $a \in \mathfrak{gl}(V)$ acts on $E$ by $a\rvert_E = \sum_{1 \leq i \leq d} a^{(i)} $).

Then we have the following result, known as Schur-Weyl duality:

\begin{theorem}[Schur-Weyl]\label{thrm:class_Schur_Weyl}
\mbox{}
\begin{itemize}
 \item $B$ is the image of $\mathcal{U}(\mathfrak{gl}(V))$ (the universal enveloping algebra of $\mathfrak{gl}(V)$) in $\End_{\bC}(E)$, and thus $E$ is a semisimple $\mathfrak{gl}(V)$-module.
 \item The images of $\bC[S_d]$ and $\mathcal{U}(\mathfrak{gl}(V))$ in $\End_{\bC}(E)$ are centralizers of each other.
\item As \InnaA{$\bC[S_d] \otimes_{\bC} \mathcal{U}(\mathfrak{gl}(V))$-module},

$$ E =\bigoplus_{\lambda: \abs{\lambda} =d} \lambda \otimes S^{\lambda} V$$
\end{itemize}

\end{theorem}
%

We now define a contravariant functor $$\mathtt{SW}_{d, V}: Rep(S_d) \rightarrow Mod_{\mathcal{U}(\gl(V)), poly}, \, \mathtt{SW}_{d, V}:=\Hom_{S_d}(\cdot,  V^{\otimes d})$$

The contravariant functor $\mathtt{SW}_{d, V}$ is $\bC$-linear and additive, and sends a simple module $\lambda$ of $S_d$ to $S^{\lambda}V$.

Next, consider the contravariant functor $$\mathtt{SW}_V: \bigoplus_{d \in \bZ_+} Rep(S_d) \rightarrow Mod_{\mathcal{U}(\gl(V)), poly}, \, \mathtt{SW}_V:= \oplus_d \mathtt{SW}_{d, V}$$
(the category $\bigoplus_{d \in \bZ_+} Rep(S_d)$ is equivalent to the category of Schur functors, and is obviously semisimple).
This functor $\mathtt{SW}_V$ is clearly essentially surjective and full (this is easy to see, since $Mod_{\mathcal{U}(\gl(V)), poly}$ is a semisimple category with simple objects $\InnaA{S^{\lambda} V \cong} \mathtt{SW}(\lambda)$).

The kernel of the functor $\mathtt{SW}_V$ is the full additive subcategory (direct factor) of $\bigoplus_{d \in \bZ_+} Rep(S_d)$ generated by simple objects $\lambda$ such that $\ell(\lambda) > \dim V $; taking the quotient, we see that $\mathtt{SW}_V$ defines an equivalence of categories $$\mathtt{SW}_V: \left(\bigoplus_{d \in \bZ_+} Rep(S_d)\right)_{\text{length } \leq \dim V } \rightarrow Mod_{\mathcal{U}(\gl(V)), poly}$$
where $\left(\bigoplus_{d \in \bZ_+} Rep(S_d)\right)_{\text{length } \leq \dim V }$ is the full additive subcategory (direct factor) of $\bigoplus_{d \in \bZ_+} Rep(S_d)$ generated by simple objects $\lambda$ such that $\ell(\lambda) \leq \dim V $.

\section{Deligne category $\underline{Rep}(S_{\nu})$}\label{sec:Del_cat_S_nu}
This section follows \cite{CO, D, E1}. Throughout the paper, we will use the parameter $\nu$ instead of the parameter $t$ used in Introduction.
\subsection{General description}\label{ssec:S_nu_general}
%
For any $\nu \in \bC$, the category $\underline{Rep}(S_{\nu})$ is generated, as a $\bC$-linear Karoubian tensor category, by one object, denoted $\fh$. This object is the analogue of the permutation representation of $S_n$, and any object in $\underline{Rep}(S_{\nu})$ is a direct summand in a direct sum of tensor powers of $\fh$.

For $\nu \notin \bZ_{+}$, $\underline{Rep}(S_{\nu})$ is a semisimple abelian category.

\begin{notation}
 We will denote Deligne's category for integer value $n \geq 0$ of $\nu$ as $\underline{Rep}(S_{\nu=n})$, to distinguish it from the classical category $Rep(S_{n})$ of representations of the symmetric group $S_{n}$. Similarly for other categories arising in this text.
\end{notation}

If $\nu$ is a non-negative integer, then the category $\underline{Rep}(S_{\nu})$ has a tensor ideal $\idealI_{\nu}$, called the ideal of negligible morphisms (this is the ideal of morphisms $f: X \longrightarrow Y$ such that $tr(fu)=0$ for any morphism $u: Y \longrightarrow X$). In that case, the classical category $Rep(S_n)$ of finite-dimensional representations of the symmetric group for $n:=\nu$ is equivalent to $\underline{Rep}(S_{\nu=n})/\idealI_{\nu}$ (equivalent as Karoubian rigid symmetric monoidal categories).

The full, essentially surjective functor $\underline{Rep}(S_{\nu=n}) \rightarrow Rep(S_n)$ defining this equivalence will be denoted by $\mathcal{S}_n$.

Note that $\mathcal{S}_n$ sends $\fh$ to the permutation representation of $S_n$.
\begin{remark}
 Although $\underline{Rep}(S_{\nu})$ is not semisimple and not even abelian when $\nu \InnaA{=n} \in \bZ_+$, a weaker statement holds (see \cite[Proposition 5.1]{D}): consider the full subcategory $\underline{Rep}(S_{\nu \InnaA{=n}})^{( \InnaA{n}/2)}$ of $\underline{Rep}(S_{\nu})$ whose objects are directs summands of sums of $\fh^{\otimes m}, 0\leq m \leq  \InnaA{\frac{n}{2}}$. This subcategory is abelian semisimple, and the restriction $\InnaA{\mathcal{S}}_{n} \lvert_{\underline{Rep}(S_{\nu \InnaA{=n}})^{( \InnaA{n}/{2})}}$ is fully faithful.
\end{remark}

The indecomposable objects of $\underline{Rep}(S_{\nu})$, regardless of the value of $\nu$, are \InnaA{parametrized (up to isomorphism)} by all Young diagrams (of arbitrary size). We will denote the indecomposable object in $\underline{Rep}(S_{\nu})$ corresponding to the Young diagram $\T$ by $X_{\T}$.

For non-negative integer $\nu =:n$, we have: the partitions $\lambda$ for which $X_{\lambda}$ has a non-zero image in the quotient $\underline{Rep}(S_{\nu \InnaA{=n}})/\idealI_{\nu\InnaA{=n}} \cong Rep(S_n)$ are exactly the $\lambda$ for which $\lambda_1+\abs{\lambda} \leq n$.

If $\lambda_1+\abs{\lambda} \leq n$, then the image of $\lambda$ in $Rep(S_n)$ is the irreducible representation of $S_n$ corresponding to the Young diagram $\tilde{\lambda}(n)$ (c.f. notation in Section \ref{sec:notation}).

\InnaA{This allows one to intuitively treat} the indecomposable objects of $\underline{Rep}(S_{\nu})$ as if they were parametrized by ``Young diagrams with a very long top row''. The indecomposable object $X_{\lambda}$ would be treated as if it corresponded to $\tilde{\lambda}(\nu)$, i.e. a Young diagram obtained by adding a very long top row (``of size $\nu -\abs{\lambda}$''). This point of view is useful to understand how to extend constructions for $S_n$ involving Young diagrams to $\underline{Rep}(S_{\nu})$.

\begin{example}
 The indecomposable object $X_{\lambda}$, where $\lambda={\tiny \yng(6,5,4,1)}$ can be thought of as a Young diagram with a ``very long top row of length $(\nu- 16)$'':
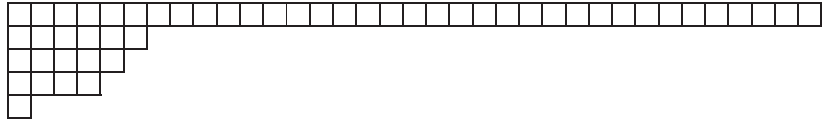
$$ \yng(35,6,5,4,1)$$
\end{example}

%
%

%
%
\subsection{Lifting objects}\label{ssec:nu_class_and_lift}
We start with an equivalence relation on the set of all Young diagrams, defined in \cite[Definition 5.1]{CO}:
\begin{definition}
Let $\lambda$ be any Young diagram, and set
$$\mu_{\lambda}(\nu) = (\nu -\abs{\lambda}, \lambda_1-1, \lambda_2-2, ...)$$
Given two Young diagrams $\lambda, \lambda'$, denote $\mu_{\lambda}(\nu) =: (\mu_0, \mu_1, ...), \mu_{\lambda'}(\nu) =: (\mu'_0, \mu'_1, ...)$.

We put $\lambda \stackrel{\nu}{\sim} \lambda'$ if there exists a bijection $f: \bZ_+ \rightarrow \bZ_+$ such that $\mu_i = \mu'_{f(i)}$ for any $i \geq 0$.

\end{definition}
We will call a $\stackrel{\nu}{\sim}$-class {\it trivial} if it contains exactly one Young diagram.

The following lemma is proved in \cite[Corollary 5.6, Proposition 5.8]{CO}:
\begin{lemma}\label{lem:nu_classes_struct}
\mbox{}
 \begin{enumerate}
 \item If $\nu \notin \bZ_+$, then any Young diagram $\lambda$ lies in a trivial $\stackrel{\nu}{\sim}$-class.
  \item The non-trivial $\stackrel{\nu}{\sim}$-classes are parametrized by all Young diagrams $\lambda$ such that $\tilde{\lambda}(\nu)$ is a Young diagram (in particular, $\nu \in \bZ_+$), and are of the form $ \{\lambda^{(i)}\}_i$, with
\begin{align*}                                                                                                                                                                                                                                       &\lambda=\lambda^{(0)} \subset \lambda^{(1)} \subset \lambda^{(2)} \subset ... \\
&\text{ and } \lambda^{(i+1)} \setminus \lambda^{(i)} = \text{ strip in row } i+1 \text{ of length } \lambda_i -\lambda_{i+1} +1 \text{ for } i>0\\
&\text{ and } \lambda^{(1)} \setminus \lambda^{(0)} = \text{ strip in row } 1 \text{ of length } \nu -\abs{\lambda} -\lambda_1 +1                                                                                                                                                                                                                                       \end{align*}
 \end{enumerate}

\end{lemma}

We now consider Deligne's category $\underline{Rep}(S_T)$, where $T$ is a formal variable (c.f. \cite[Section 3.2]{CO}). This category is $\bC((T-\nu))$-linear, but otherwise it is very similar to Deligne's category $\underline{Rep}(S_{\nu})$ for generic $\nu$. For instance, as a $\bC((T-\nu))$-linear Karoubian tensor category, $\underline{Rep}(S_T)$ is generated by one object, again denoted by $\fh$.

One can show that $\underline{Rep}(S_T)$ is \InnaA{split} semisimple and thus abelian, and its simple objects are parametrized by Young diagrams of arbitrary size.

In \cite[Section 3.2]{CO}, Comes and Ostrik defined a map
 \begin{align*}
lift_{\nu}: \{\substack{\text{objects in } \underline{Rep}(S_{\nu}) \\ \text{up to isomorphism}}\}\rightarrow \{\substack{\text{objects in } \underline{Rep}(S_{T}) \\ \text{up to isomorphism}}\}
 \end{align*}

We will not give the precise definition of this map, but will list some of its useful properties. It is defined to be additive (i.e. $lift_{\nu}(A \oplus B) \cong lift_{\nu}(A) \oplus lift_{\nu}(B)$ for any $A, B \in \underline{Rep}(S_{\nu})$) and satisfies $lift_{\nu}(\fh)\cong \fh$. Moreover, we have:

\begin{proposition}\label{prop:lift_properties}
Let $A, B$ be two objects in $\underline{Rep}(S_{\nu})$.
 \begin{enumerate}
 \item $lift_{\nu}(A \otimes B) \cong lift_{\nu}(A) \otimes lift_{\nu}(B)$.
  \item $\dim_{\underline{Rep}(S_{\nu})} A = (\dim_{\underline{Rep}(S_T)} lift_{\nu}(A))\rvert_{T=\nu}$.
  \item $\dim_{\bC} \Hom_{\underline{Rep}(S_{\nu})} (A, B) = \dim_{Frak(\bC[[T]])} \Hom_{\underline{Rep}(S_{T})} (lift_{\nu}(A), lift_{\nu}(B))$.
  \item The map $lift_{\nu}$ is injective.
  \item \InnaA{For any $\lambda$,} $lift_{\nu}(X_{\lambda})= X_{\lambda}$ for all but finitely many $\nu \in \bC$.
 \end{enumerate}
 \end{proposition}

 \begin{proof}
C.f. \cite[Proposition 3.12]{CO}.
 \end{proof}

\begin{remark}
It was proved both in \cite[Section 7.2]{D} and in \cite[Proposition 3.28]{CO} that the dimensions of the indecomposable objects $X_{\lambda}$ in $\underline{Rep}(S_{T})$ are polynomials in $T$ whose coefficients depend on $\lambda$ (given $\lambda$, this polynomial can be written down explicitly). Such polynomials are denoted by $P_{\lambda}(T)$.

Furthermore, it was proved in \cite[Proposition 5.12]{CO} that given $d \in \bZ_+$ and a Young diagram $\lambda$, $\lambda$ belongs to a trivial $\stackrel{d}{\sim}$-class iff $P_{\lambda}(d)=0$.
\end{remark}

The following result is proved in \cite[Lemma 5.20]{CO}, and is a stronger version of the statement in Proposition \ref{prop:lift_properties}(e):

\begin{lemma}[Comes, Ostrik]\label{lem:nu_classes_lift}
Consider the $\stackrel{\nu}{\sim}$-equivalence relation on Young diagrams.

\begin{itemize}
 \item Whenever $\lambda$ lies in a trivial $\stackrel{\nu}{\sim}$-class, $lift_{\nu}(X_{\lambda}) = X_{\lambda}$.
\item For a non-trivial $\stackrel{\nu}{\sim}$-class $\{ \lambda^{(i)} \}_i$,
\begin{align*}
& lift_{\nu}(X_{\lambda^{(0)}}) = X_{\lambda^{(0)}},
& lift_{\nu}(X_{\lambda^{(i)}}) = X_{\lambda^{(i)}} \oplus X_{\lambda^{(i-1)}} \, \forall i \geq 1
\end{align*}
\end{itemize}
\end{lemma}

Based on Lemmas \ref{lem:nu_classes_struct}, \ref{lem:nu_classes_lift}, Comes and Ostrik prove the following theorem (c.f. \cite[Theorem 5.3, Proposition 5.22, Theorems 6.4, 6.10]{CO}, \cite[Proposition 2.7]{CO2}):

\begin{theorem}\label{thrm:blocks_S_nu}
 The indecomposable objects $X_{\lambda}, X_{\lambda'}$ belong to the same block of $\underline{Rep}(S_{\nu})$ iff $\lambda \stackrel{\nu}{\sim} \lambda'$. The structure of the blocks of $\underline{Rep}(S_{\nu})$ is described below:
 \begin{itemize}[leftmargin=*]
  \item For a trivial $\stackrel{\nu}{\sim}$-class $\{ \lambda \}$, the object $X_{\lambda}$ satisfies:
  $$\dim \End_{\underline{Rep}(S_{\nu})} (X_{\lambda}) =1$$
  and the block of $\underline{Rep}(S_{\nu})$ corresponding to $\{\lambda\}$ is equivalent to the category $Vect_{\bC}$ of finite dimensional complex vector spaces (in particular, it is a semisimple abelian category, so we will call these blocks {\it semisimple}).
  \item Let $ \{\lambda^{(i)}\}_i$ be a non-trivial $\stackrel{\nu}{\sim}$-class, and let $i \geq 1, j \geq 0$. Then the block corresponding to $ \{\lambda^{(i)}\}_i$ is not an abelian category (in particular, not semisimple), and the objects $X_{\lambda^{(i)}}$ satisfy:
\begin{align*}
&\dim \Hom_{\underline{Rep}(S_{\nu})} \left( X_{\lambda^{(j)}}, X_{\lambda^{(i)}} \right) = 0 \text{ if } \InnaA{\abs{j -i} \geq 2}\\
&\dim \Hom_{\underline{Rep}(S_{\nu})} \left( X_{\lambda^{(j)}}, X_{\lambda^{(i)}} \right) = 1 \text{ if } \InnaA{\abs{j -i} =1}\\
&\dim \End_{\underline{Rep}(S_{\nu})} \left( X_{\lambda^{(i)}} \right) = 2 \text{ for } i\geq 1 \\
& \dim \End_{\underline{Rep}(S_{\nu})} \left( X_{\lambda^{(0)}} \right) = 1
\end{align*}

This block has the following associated quiver:
$$ X_{\lambda^{(0)}} \substack{ \alpha_0 \\ \leftrightarrows \\ \beta_0} X_{\lambda^{(1)}}  \substack{ \alpha_1 \\ \leftrightarrows \\ \beta_1} X_{\lambda^{(2)}}  \substack{ \alpha_2 \\ \leftrightarrows \\ \beta_2} ...$$
with relations \InnaA{$ \alpha_0 \circ \beta_0 =0$}, $\beta_i \circ \beta_{i-1} =0$, \InnaA{$\alpha_i \circ \alpha_{i-1} = 0$, $\beta_i \circ \alpha_i = \alpha_{i+1} \circ \beta_{i+1}$ for $i \geq 0$}.
 \end{itemize}

\end{theorem}

\subsection{Objects $\Delta_k$}\label{ssec:Delta_obj}
In this subsection we define the objects $\Delta_k$ in the category $\underline{Rep}(S_{\nu})$, and list some of their properties. These objects are defined for any $k \in \bZ_+$ and any $\nu \in \bC$.

By definition, $\Delta_k$ is the image of an idempotent $x_k \in \End_{\underline{Rep}(S_{\nu})}(\fh^{\otimes k})$ (the latter is given explicitly in \cite[Section 3.1]{CO2}), and satisfies:

\begin{lemma}\label{lem:funct_F_n_Delta_k}
 $\mathcal{S}_n(\Delta_k) \cong \bC Inj(\{1,...,k\} , \{1,...,n\}) \cong Ind_{S_{n-k} \times S_k \times S_k}^{S_n \times S_k} \operatorname{\InnaA{\bC}}$.
\end{lemma}
This is part of the definition of the functor $\mathcal{S}_n$ in \cite[Theorem 6.2]{D}.

\begin{remark}
 The tensor functor $\mathcal{S}_n$ takes $ \fh^{\otimes k} $ to $\bC Fun(\{1,...,k\} , \{1,...,n\})$ (\cite{CO} uses this as part of the definition).
\end{remark}

\begin{example}
 $\Del_0 \cong \InnaA{ \mathbf{1}}$ \InnaA{(unit object in monoidal category $\underline{Rep}(S_{\nu})$)}, $\Del_1 \cong \fh$.
\end{example}

\begin{remark}
 Deligne in \cite{D} denotes the full subcategory of $\underline{Rep}(S_{\nu})$ whose objects are $\{\Delta_k\}_{k\geq 0}$ by $Rep_0 (S_{\nu})$. This subcategory is a tensor subcategory (with respect to the tensor product in $\underline{Rep}(S_{\nu})$), and it is used as the first step in defining the category $\underline{Rep}(S_{\nu})$. Namely, one first describes the structure of $Rep_0 (S_{\nu})$ as a $\bC$-linear rigid symmetric monoidal category (see \cite[Section 2]{D}) and then defines $\underline{Rep}(S_{\nu})$ as the Karoubi envelope of $Rep_0 (S_{\nu})$.

 Comes and Ostrik, on the other hand, consider the full subcategory (denoted by $\underline{Rep}_0 (S_{\nu})$) of $\underline{Rep}(S_{\nu})$ whose objects are $\{\fh^{\otimes k} \}_{k\geq 0}$. This is also a tensor subcategory. They start by defining the structure of $\underline{Rep}_0 (S_{\nu})$ as a $\bC$-linear rigid symmetric monoidal category (see \cite[Section 2]{CO}) and then define $\underline{Rep}(S_{\nu})$ as the Karoubi envelope of $\underline{Rep}_0 (S_{\nu})$.

In \cite[Section 8.2]{D}, Deligne showed that these two definitions are equivalent.
\end{remark}

We now describe the $\Hom$-spaces between the objects $\Delta_k$. We start by introducing the following notation (see \cite[Section 2]{CO}):

\begin{notation}
\mbox{}
\begin{itemize}[leftmargin=*]
\item By a {\it partition} $\pi$ of a set $S$ we will denote a collection $\{ \pi_i \}_{i \in I}, \pi_i \subset S$, such that $\pi_i \cap \pi_j = \emptyset$ if $i \neq j$, and $\bigcup_{i \in I} \pi_i =S$. The subsets $\pi_i$ will be called {\it parts} of $\pi$. The number of parts of $\pi$ will be denoted by $l(\pi)$.
 \item Let $P_{r,s}$ be the set of all partitions of the set $\{1,..., r, 1', ..., s'\}$; $P_{0,s}$ is then the set of all partitions of $\{1', ..., s'\}$, $P_{r, 0}$ is the set of all partitions of $\{1,..., r\}$, $P_{0, 0} := \{ \text{empty partition} \}$.
 \item Let $\bar{P}_{r,s}$ be the subset of $P_{r,s}$ consisting of all the partitions $\pi$ such that $i, j$ do not lie in the same part of $\pi$ whenever $i \neq j$, and similarly for $i', j'$.
 \item The following diagrammatic notation will be used for elements of $P_{r,s}$ (resp. $\bar{P}_{r,s}$): let $\pi \in P_{r,s}$. We will represent $\pi$ by any graph whose vertices are labeled $1,...,r, 1',..., s'$, and whose connected components partition the vertices into disjoint subsets corresponding to parts of $\pi$.

 For our convenience, we will always present such graphs as graphs with two rows of aligned vertices: the top row contains $r$ vertices labeled by numbers $1,...,r$, and the bottom row contains $s$ vertices labeled by numbers $1',...,s'$.

\end{itemize}
\end{notation}

\begin{remark}
 In this diagrammatic representation, partitions $\pi \in \bar{P}_{r,s}$ are exactly those which are represented by bipartite graphs with $deg(v) \leq 1$ for any vertex $v$. These partitions have exactly one diagram which represents them.
\end{remark}

\begin{example}
\mbox{}
\begin{enumerate}[leftmargin=*]
 \item Let $\pi \in P_{6,3}, \pi:= \{\{1, 1', 3\}, \{2, 4, 5\}, \{ 2', 3'\}, \{6\}\}$. The diagram representing $\pi$ can be drawn as:
$$ 
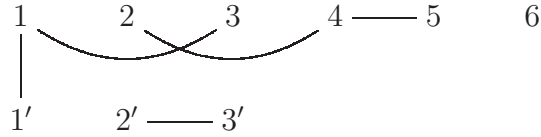
\xymatrix{ &1 \ar@{-}@/_1.3pc/[rdru] \ar@{-}[d] & 2 \ar@{-}@/_1.3pc/[rdru] & 3  & 4 \ar@{-}[r] & 5 & 6 \\ &1' & 2' \ar@{-}[r] &3'}$$

 \item Let $\pi' \in P_{6,3}, \pi':= \{\{1, 1'\}, \{2, 3'\}, \{ 2'\}, \{ 3 \} , \{4 \}, \{5\}, \{6 \}\}$. The diagram representing $\pi'$ is:
$$ 
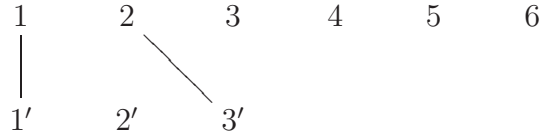
\xymatrix{ &1 \ar@{-}[d] &2  \ar@{-}[dr] &3  &4 &5 &6 \\ &1' & 2' &3'}$$
\end{enumerate}

Notice that $\pi \notin \bar{P}_{6,3}$, but $\pi' \in \bar{P}_{6,3}$.
\end{example}

We now describe how to ``glue'' two diagrams together to obtain a new diagram.

Let $\pi \in P_{r,s}, \rho \in P_{s,t}$. We will denote the vertices in the top (resp. bottom) row of $\pi$ by $1,...,r$ (resp. $1',...,s'$),and the vertices in the top (resp. bottom) row of $\rho$ by $1',...,s'$ (resp. $1'', ...,t''$).

We draw the diagram of $\pi$ on top of the diagram of $\rho$, with the bottom row of $\pi$ (vertices $1', ...,s'$) identified with the top row of $\rho$. We will call the diagram obtained {\it the gluing of $\pi, \rho$}, and will denote it by $D_{\pi, \rho}$.

We next consider the diagram induced by $D_{\pi, \rho}$ on the vertices $1,...,r, 1'',...,t''$ (by ``induced diagram'' we mean the diagram in which two vertices lie in the same connected component iff they were in the same connected component of $D_{\pi, \rho}$). This diagram (and the partition in $P_{r,t}$ it represents) will be denoted by $\rho \star \pi$.

The second piece of information we want to retain from the diagram $D_{\pi, \rho}$ is the number of connected components lying entirely in the middle row. We will denote this number by $n(\rho, \pi)$. Thus
$$\sharp \text{ connected components of } D_{\pi, \rho} = l(D_{\pi, \rho}) = n(\rho, \pi) + l(\rho \star \pi)$$

\begin{example}
 Let $\pi \in P_{6,5}, \pi:= \{\{1, 1', 3\}, \{2, 4, 5\}, \{ 2', 3'\}, \{4'\}, \{5'\}, \{6\}\}$,

 $\rho \in P_{5,4}, \rho=\{ \{1', 2'', 4', 4'' \}, \{2', 3' \}, \{5'\}, \{1'', 3''\}\}$.

 Then the diagrams of $\pi, \rho$ can be drawn as:
$$ \pi = 
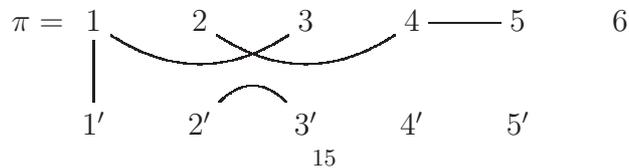
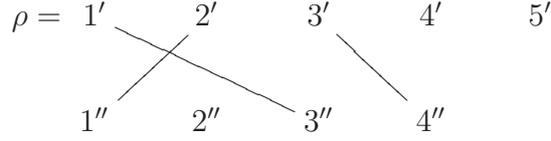
\xymatrix{ 1 \ar@{-}@/_1.3pc/[rdru] \ar@{-}[d] & 2 \ar@{-}@/_1.3pc/[rdru] & 3  & 4 \ar@{-}[r] & 5 & 6 \\ 1' & 2' \ar@{-}@/^1.3pc/[r] &3' &4' &5'}$$

$$ \rho = \xymatrix{ 1' \ar@{-}@/_1.3pc/[rdrru] \ar@{-}[rd] & 2' \ar@{-}[r] & 3'  & 4' & 5'  \\ 1'' \ar@{-}@/^1.3pc/[rurd] & 2'' \ar@{-}@/^1.3pc/[rurd] &3'' & 4''}$$

Next, we draw the gluing of $\pi, \rho$, denoted by $D_{\pi, \rho}$:
$$ D_{\pi, \rho} = \xymatrix{ 1 \ar@{-}@/_1.3pc/[rdru] \ar@{-}[d] & 2 \ar@{-}@/_1.3pc/[rdru] & 3  & 4 \ar@{-}[r] & 5 & 6 \\ 1' \ar@{-}@/_1.3pc/[rdrru]  \ar@{-}[rd] & 2'\ar@{-}@/^1.3pc/[r] \ar@{-}[r] &3' &4' &5' \\  1'' \ar@{-}@/^1.3pc/[rurd] & 2'' \ar@{-}@/^1.3pc/[rurd] &3'' & 4''}$$

Then
$$ \rho \star \pi = \xymatrix{ 1 \ar@{-}@/_1.3pc/[rdru] \ar@{-}[rd] & 2 \ar@{-}@/_1.3pc/[rdru] & 3  & 4 \ar@{-}[r] & 5 & 6 \\ 1'' \ar@{-}@/^1.3pc/[rurd] & 2'' \ar@{-}@/^1.3pc/[rurd] &3'' & 4''}$$
i.e. $ \rho \star \pi = \{ \{1, 2'', 3, 4''\}, \{1'', 3''\}, \{2,4,5 \}, \{ 6\}\}$ (partition of the set $\{ 1,...,6,1'',...,4''\}$), and  $n(\rho, \pi) = 2$.
\end{example}

The following statement is used by Comes and Ostrik in \cite[Definition 2.11]{CO} as the definition in the construction of $\underline{Rep}(S_{\nu})$; Deligne derives it in \cite[Proposition 8.3]{D}.

\begin{definition}\label{def:ten_power_h_homs}
 Let $r, s \geq 1$. The space $\Hom_{\underline{Rep}(S_{\nu})} (\fh^{\otimes r}, \fh^{\otimes s}) $ is defined to be $\bC P_{r,s}$, and the composition of morphisms between tensor powers of $\fh$ is bilinear and given by the following formula: for $\pi \in P_{r,s}, \rho \in P_{s,t}$,
 $$ \rho \circ \pi := \nu^{n(\rho, \pi)}   \rho \star \pi \in \bC P_{r,t}$$
\end{definition}

The following statement is used as a definition in \cite[Definition 3.12]{D}, and can easily be derived from the definition of $\Delta_k$ (c.f. \cite[Section 3.1]{CO2}) and from Definition \ref{def:ten_power_h_homs}.
\begin{lemma}\label{lem:Delta_k_homs}
 Let $r, s \geq 1$. The space $\Hom_{\underline{Rep}(S_{\nu})} (\Del_r, \Del_s) $ is $\bC \bar{P}_{r,s}$, and the composition of morphisms between the objects $\Delta_k$ is given by the following formula: for $\pi \in \bar{P}_{r,s}, \rho \in \bar{P}_{s,t}$,
 $$ \rho \circ \pi = \sum_{ \tau \in \bar{P}_{r,t}: \rho \star \pi \subset \T} \mathit{p}_{\rho, \pi, \T}(\nu)   \T  \in \bC \bar{P}_{r,t}$$
 where
 \begin{itemize}[leftmargin=*]
  \item For $\tau \in \bar{P}_{r,t}$, $ \rho \star \pi \subset \T$ means that the diagram of $\T$ contains the diagram of $ \rho \star \pi$ as a subgraph (equivalently, $ \T$ is a coarser partition of the set $\{1,...,r, 1'', ...,t''\}$ than $ \rho \star \pi$),
  \item $\mathit{p}_{\rho, \pi, \T}$ is the polynomial
 $$\mathit{p}_{\rho, \pi, \T}(x) = (x- l(\T))(x-  l(\T) -1)...(x -  l( \T) - n(\rho, \pi)+1)$$
 \end{itemize}

\end{lemma}

\begin{example}
Let $\pi \in \bar{P}_{5,5}, \rho \in \bar{P}_{5, 4}, \pi := \{\{1, 1'\}, \{2, 3'\}, \{ 2', 4\}, \{ 3 \}, \{4' \}, \{5\}, \{5'\}\}$,

$\rho:=\{ \{1', 3''\}, \{1'', 2'\}, \{2''\}, \{3', 4''\},  \{4'\}, \{5'\}\}$. The diagrams representing $\pi, \rho$ can be drawn as:
$$\pi= \xymatrix{ 1 \ar@{-}[d] & 2  \ar@{-}[dr] & 3  & 4  \ar@{-}[dll] & 5 \\ 1' & 2' & 3' & 4' & 5'}$$
$$\rho = \xymatrix{ 1' \ar@{-}[drr] & 2'  \ar@{-}[dl] & 3' \ar@{-}[dr] & 4' & 5' \\ 1'' & 2'' & 3'' & 4''}$$
Gluing $\pi$ and $\rho$ together, we get:
$$ D_{\pi, \rho} = 
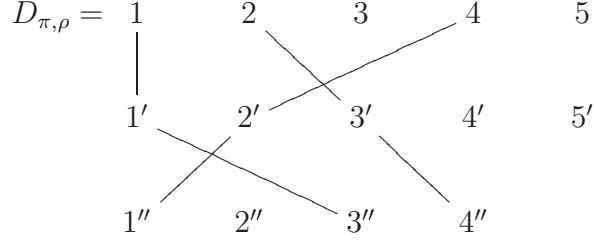
\xymatrix{  1 \ar@{-}[d] & 2  \ar@{-}[dr] & 3  & 4  \ar@{-}[dll] & 5  \\  1' \ar@{-}[drr] & 2'  \ar@{-}[dl] & 3' \ar@{-}[dr] & 4' & 5' \\ 1'' & 2'' & 3'' & 4''}$$
Then $$\rho \star \pi = 
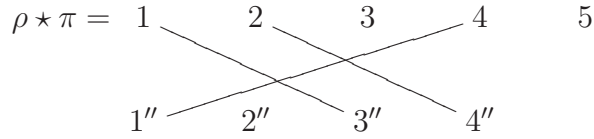
\xymatrix{  1 \ar@{-}[drr] & 2  \ar@{-}[drr] & 3  & 4  \ar@{-}[dlll] & 5 \\ 1'' & 2'' & 3'' & 4''}$$
i.e. $\rho \star \pi = \{\{1, 3''\}, \{1'', 4\}, \{2''\}, \{ 2, 4''\}, \{3\}, \{5\} \}$ and $n(\rho, \pi) = 2$.

Next, we are looking for $\tau \in \bar{P}_{5,4}$ such that the diagram of $\tau$ contains $\rho \star \pi$ as a subgraph. There are three such partitions $\tau$:
\begin{itemize}
 \item $ \tau_1 = \rho \star \pi$, in which case $\mathit{p}_{\rho, \pi, \T_1}(x) = (x- 6)(x-  7)$.
 \item $$\tau_2 = 
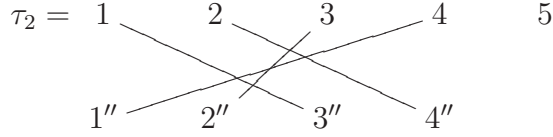
\xymatrix{  1 \ar@{-}[drr] & 2  \ar@{-}[drr] & 3 \ar@{-}[dl]  & 4  \ar@{-}[dlll] & 5 \\ 1'' & 2'' & 3'' & 4''}$$ in which case $\mathit{p}_{\rho, \pi, \T_2}(x) = (x- 5)(x-6)$.
 \item $$\tau_3 = 
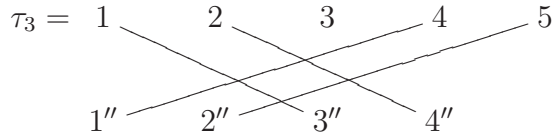
\xymatrix{  1 \ar@{-}[drr] & 2  \ar@{-}[drr] & 3  & 4  \ar@{-}[dlll] & 5 \ar@{-}[dlll]  \\ 1'' & 2'' & 3'' & 4''}$$ in which case $\mathit{p}_{\rho, \pi, \T_3}(x) =  (x- 5)(x-6)$.
\end{itemize}

Thus $\rho \circ \pi =  (\nu- 6)(\nu-7) \tau_1 + (\nu-5)(\nu-6) \tau_2 + (\nu-5)(\nu-6) \tau_3$.
\end{example}

The following morphisms between the objects $\Delta_k$ will be used frequently.

Let $r\geq 0$, $k \geq 1$, $ 1 \leq l \leq k$.
\begin{definition}\label{def:res_morphisms}
Denote by $res_l$ the morphism $ \Del_{k+1} \rightarrow \Del_{k}$ given by the diagram
$$  
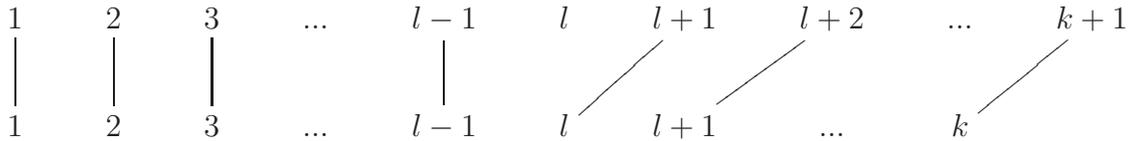
\xymatrix{ 1 \ar@{-}[d] & 2 \ar@{-}[d] & 3 \ar@{-}[d] & ... & l-1 \ar@{-}[d] & l  & l+1 \ar@{-}[dl] & l+2 \ar@{-}[dl] & ... & k+1 \ar@{-}[dl] \\ 1 & 2 & 3 & ... & l-1 & l & l+1 & ... & k }$$

  By abuse of notation, we will also denote by $res_l$ the maps $\bar{P}_{r,k+1} \rightarrow \bar{P}_{r,k}$, $ \bC \bar{P}_{r,k+1} \rightarrow \bC \bar{P}_{r,k}$ given by $\pi \mapsto res_l \circ \pi$.

  Notice that given $\pi \in \bar{P}_{r,k+1}$, a diagram describing the partition $res_l \circ \pi \in \bar{P}_{r,k}$ can be obtained by removing a vertex (labeled $l'$) from position $l$ of the bottom row of the diagram of $\pi$, and shifting the labels of the vertices lying to the right. If the vertex removed was connected to another vertex by an edge, then the edge is removed as well, but the second vertex stays.
\end{definition}
\begin{definition}\label{def:res_star_morphisms}

Denote by $res_l^*$ the morphism $ \Del_k \rightarrow \Del_{k+1}$ given by the diagram
  $$  \xymatrix{ 1 \ar@{-}[d] & 2 \ar@{-}[d] & 3 \ar@{-}[d] & ... & l-1 \ar@{-}[d] & l \ar@{-}[dr] & l+1 \ar@{-}[dr] & ... & k \ar@{-}[dr] \\ 1 & 2 & 3 & ... & l-1 & l & l+1 & l+2& ... & k+1}$$

  By abuse of notation, we will also denote by $res_l^*$ the maps $\bar{P}_{r,k} \rightarrow \bar{P}_{r,k+1}$, $ \bC \bar{P}_{r,k} \rightarrow \bC \bar{P}_{r,k+1}$ given by $\pi \mapsto res_l^* \circ \pi$.

  Notice that given $\pi \in \bar{P}_{r,k}$, a diagram describing the partition $res_l^* \circ \pi \in \bar{P}_{r,k+1}$ can be obtained by inserting a solitary vertex (labeled $l'$) in position $l$ of the bottom row of the diagram of $\pi$, and shifting the labels of the vertices lying to the right.

\end{definition}

\begin{remark}\label{rmrk:res_l_image_under_S_n_functor}
Let $n \in \bZ_+$. Fix $k, l$ such that $1 \leq k \leq n-1,  1 \leq l \leq k$. Denote $$\mathtt{res}_l := \mathcal{S}_n(res_l): \, \bC Inj(\{1,...,k+1\} , \{1,...,n\}) \rightarrow \bC Inj(\{1,...,k\} , \{1,...,n\})$$
$$\mathtt{res}_l^* := \mathcal{S}_n(res_l^*): \, \bC Inj(\{1,...,k\} , \{1,...,n\}) \rightarrow \bC Inj(\{1,...,k+1\} , \{1,...,n\})$$

Denote by $\iota_l$ the injection 
$$ \{1,...,k\}\hookrightarrow \{1,...,k+1\}, \, i \mapsto \begin{cases} i &\text{ if } i < l \\
i+1 &\text{ if } i \geq l                                                                                                  \end{cases}$$ 

Then given $g: \{1,...,k+1\} \hookrightarrow \{1,...,n\}$, we have $\mathtt{res}_l(g) = g \circ \iota_l$, and given $f: \{1,...,k+1\} \hookrightarrow \{1,...,n\}$,  we have $$\mathtt{res}_l^*(f) = \sum_{\substack{\InnaA{g} \in Inj(\{1,...,k+1\} , \{1,...,n\}):\\ g \circ \iota_l =f}} g$$
\end{remark}

\begin{example}
 Let $\pi \in \bar{P}_{5,5}, \pi := \{\{1, 1'\}, \{2' \}, \{2, 3'\},  \{ 3 \}, \{ 4, 4'\}, \{5\}, \{5'\}\}$, i.e.  $$\pi= \xymatrix{ 1 \ar@{-}[d] & 2  \ar@{-}[dr] & 3  & 4  \ar@{-}[d] & 5 \\ 1' & 2' & 3' & 4' & 5'}$$

 Then $$res^*_3(\pi) = \xymatrix{ 1 \ar@{-}[d] & 2  \ar@{-}[drr] & 3  & 4  \ar@{-}[dr] & 5 \\ 1' & 2' & 3' & 4' & 5' & 6'}$$ and $$res_3(\pi) = \xymatrix{ 1 \ar@{-}[d] & 2 & 3  & 4  \ar@{-}[dl] & 5 \\ 1' & 2' & 3' & 4' }$$
\end{example}

\begin{remark}
One can define an endomorphism $\bar{x}_k \in \End_{\underline{Rep}(S_{T})}(\fh^{\otimes k})$ similar to the idempotent $x_k$, so that $\Im(\bar{x}_k) \cong lift_{\nu}(\Delta_k)$ for any $\nu$ (this is a direct consequence of the definition of $lift_{\nu}$).

By abuse of notation, we will denote $\Im(\bar{x}_k)$ by $\Delta_k$ as well, and use the isomorphism $\Delta_k \cong lift_{\nu}(\Delta_k)$ in Lemma \ref{lem:hom_X_tau_Delta_k}.
\end{remark}

\subsection{Abelian envelope}\label{ssec:S_nu_abelian_env}
This section follows \cite{CO2}, \cite[Proposition 8.19]{D}.

As it was mentioned before, the category $\underline{Rep}(S_{\nu})$ is defined as a Karoubian category. For $\nu \notin \bZ_+$, it is semisimple and thus abelian, but for $\nu \in \bZ_+$, it is not abelian. Fortunately, it has been shown that $\underline{Rep}(S_{\nu})$ possesses an ``abelian envelope'', that is, that it can be embedded in an abelian tensor category, and this abelian tensor category has a universal mapping property.

The following result was conjectured by Deligne in \cite[8.21.2]{D}, and proved by Comes and Ostrik in \cite[Theorem 1.2]{CO2}:

\begin{theorem}
 Let $n \in \bZ_+$. There exists an abelian $\bC$-linear rigid symmetric monoidal category $\underline{Rep}^{ab}(S_{\nu=n})$, and an embedding (fully faithful tensor functor) 
 
 $\iota: \underline{Rep}(S_{\nu=n}) \rightarrow \underline{Rep}^{ab}(S_{\nu=n})$ which makes the pair $(\underline{Rep}^{ab}(S_{\nu=n}), \iota )$ the ``abelian envelope'' of $\underline{Rep}(S_{\nu=n})$ in the following sense:

 Let $\mathcal{T}$ be an abelian $\bC$-linear rigid symmetric monoidal category such that all $\Hom$-spaces are finite-dimensional and all objects have finite length; in addition, let there be a tensor functor of Karoubian categories $\mathcal{G}:\underline{Rep}(S_{\nu=n}) \rightarrow \mathcal{T}$. Then the functor $\mathcal{G}$ factors through one of the following:

 \begin{enumerate}
  \item The functor $\mathcal{S}_n: \underline{Rep}(S_{\nu=n}) \rightarrow Rep(S_n)$ (this happens iff $\mathcal{G}(\Delta_{n+1}) = 0$).
  \item The functor $\iota: \underline{Rep}(S_{\nu=n}) \rightarrow \underline{Rep}^{ab}(S_{\nu=n})$ (this happens iff $\mathcal{G}(\Delta_{n+1}) \neq 0$).
 \end{enumerate}

\end{theorem}

For $\nu \notin \bZ_+$, we will put $(\underline{Rep}^{ab}(S_{\nu}), \iota ) := (\underline{Rep}(S_{\nu}), \id_{\underline{Rep}(S_{\nu})} )$.

An explicit construction of the category $\underline{Rep}^{ab}(S_{\nu=n})$ is given in \cite{CO2}. We will only list the results which will be used in this paper.

\begin{remark}\label{rmrk:ab_env_preTannakian}
 The category $\underline{Rep}^{ab}(S_{\nu})$ is a pre-Tannakian category (c.f. \cite[Section 2.1, Corollary 4.7]{CO2}). This means, in particular, that the objects in $\underline{Rep}^{ab}(S_{\nu})$ have finite length, and the $\Hom$-spaces are finite-dimensional.
\end{remark}

We start by introducing the category $\mathcal{C}_q$ of finite-dimensional representations of the quantum $SL(2)$.
The category $\mathcal{C}_q$ is a $\bC$-linear abelian category, and has the structure of a highest weight category \InnaA{(with infinitely many weights). When $q$ is a root of unity, this category can be non-semisimple, and this is the case which will be of interest to us.}

This category has a structure very similar to the structure of the category $\underline{Rep}^{ab}(S_{\nu})$; moreover, the $\bC$-linear Karoubian category $TL(q)$ of tilting modules in $\mathcal{C}_q$ (also known to be equivalent to the Temperley-Lieb category) has a structure very similar to that of $\underline{Rep}(S_{\nu})$. See \cite[Par. 1.4, 2.3, 4.3.3]{CO2} for more details.

We will use the description of the structure of $\mathcal{C}_q$ given in \cite{CO2}, \cite{A}, \cite{APW}. The main facts about $\mathcal{C}_q$ which will be used are concentrated in the following lemma.

\begin{lemma}\label{lem:TL_cat_prop}
\mbox{}
 \begin{enumerate}[leftmargin=*]
  \item All the projective modules in $\mathcal{C}_q$ are injective and conversely. \InnaA{Thus} all the projective modules of $\mathcal{C}_q$ are tilting modules, i.e. lie in the category $TL(q)$.

 \item For $q \neq \pm 1$ being a root of unity of even order, $TL(q)$ has at least one non-semisimple block (all of the non-semisimple blocks of $TL(q)$ are equivalent as Karoubian categories); the isomorphism classes of indecomposable objects in this block can be labeled $\mathrm{Q}_0, \mathrm{Q}_1, ...$. Each $\mathrm{Q}_i$ has a unique highest weight.

 Denote by $\mathrm{L}_i$ the simple module in $\mathcal{C}_q$ having the same highest weight as $\mathrm{Q}_i$, and by $\mathrm{M}_i, \mathrm{M}^{\vee}_i, \mathrm{P}_i$ the corresponding standard, co-standard and indecomposable projective modules in $\mathcal{C}_q$. With these notations, we have:

 \begin{itemize}
 \item For any $i \geq 0$, there exists an injective map $\mathrm{M}_i \rightarrow \mathrm{Q}_i$. Moreover, $[\mathrm{Q}_i :\mathrm{L}_i]=1$.
  \item The module $\mathrm{Q}_0$ is standard, co-standard and simple.
 \item The module $\mathrm{Q}_i$ is a projective module iff $i \geq 1$. Furthermore, $\{\mathrm{Q}_i\}_{i\geq 1}$ is the complete set of isomorphism classes of indecomposable projective modules in the corresponding block of $\mathcal{C}_q$.
 \end{itemize}

 \end{enumerate}
\end{lemma}

\begin{proof}
 \begin{enumerate}[leftmargin=*]
  \item This follows from \cite[Theorem 9.12]{APW}, \cite[5.7]{A}.
  \item For general information on non-semisimple blocks of $TL(q)$ and the indecomposable tilting modules, c.f. \cite[Lemma 2.11, Section 4.3.3]{CO2}, \cite[Theorem 2.5, Corollary 2.6]{A}.
  \begin{itemize}[leftmargin=*]
   \item $\mathrm{Q}_i$ is a standardly filtered having the same highest weight as $\mathrm{L}_i$ (by definition), therefore there exists an injective map $\mathrm{M}_i \rightarrow \mathrm{Q}_i$ (c.f. \cite[Proposition 3.7]{H}). Also, from \cite[Theorem 2.5]{A} we know that the highest weight of $\mathrm{Q}_i$ occurs with multiplicity $1$, therefore $[\mathrm{Q}_i :\mathrm{L}_i]=1$.
   \item C.f. \cite[Section 4]{A}.
   \item Let $St_q$ be the Steinberg module (c.f. \cite[Section 5]{A}, \cite[Section 4.3.3]{CO2}).

   By \cite[Lemma 9.10]{APW}, for any finite dimensional module $E \in \mathcal{C}_q$, the module $St_q \otimes E$ is projective. Furthermore, \cite[Theorem 9.12]{APW} implies that any projective module in $\mathcal{C}_q$ is a direct summand of $St_q \otimes E$ for some $E$.

   Next, it is known that a module $M \in TL(q)$ has quantum dimension zero iff it is a direct summand of $St_q \otimes E$ for some $E \in TL(q)$ (c.f. \cite[Section 4.3.3]{CO2}). Thus a module $M \in TL(q)$ has quantum dimension zero iff it is projective.

   Finally, \cite[Lemma 2.11]{CO2} tells us that $\mathrm{Q}_i$ has quantum dimension zero iff $i>0$. Thus $\mathrm{Q}_i$ is a projective module iff $i>0$. From Part (1) we deduce that any indecomposable projective module in the corresponding block of $\mathcal{C}_q$ is isomorphic to $\mathrm{Q}_i$ for some $i>0$.

  \end{itemize}
   \end{enumerate}

\end{proof}

We can now describe the blocks of the category $\underline{Rep}^{ab}(S_{\nu})$ (c.f. \cite[Proposition 2.9, Section 4]{CO2}).

\begin{theorem}\label{thrm:blocks_ab_envelope}
 The blocks of the category $\underline{Rep}^{ab}(S_{\nu})$, just like the blocks of $\underline{Rep}(S_{\nu})$, are parametrized by $\stackrel{\nu}{\sim}$-equivalence classes. For each $\stackrel{\nu}{\sim}$-equivalence class $C$, the block $\underline{Rep}^{ab}(S_{\nu})_C$ corresponding to $C$ contains $\iota(\underline{Rep}(S_{\nu})_C)$ (the block of $\underline{Rep}(S_{\nu})$ corresponding to $C$, namely, the indecomposable objects $X_{\lambda}$ such that $\lambda \in C$).

 \begin{itemize}[leftmargin=*]
  \item For a trivial $\stackrel{\nu}{\sim}$-class $C=\{ \lambda \}$, the block $\underline{Rep}^{ab}(S_{\nu})_C$ is equivalent to the category $Vect_{\bC}$ of finite dimensional complex vector spaces (and is also equivalent to $\underline{Rep}(S_{\nu})_C$).

  \item For a non-trivial $\stackrel{\nu}{\sim}$-class $C = \{ \lambda^{(i)} \}_{i \geq 0}$, the block $\underline{Rep}^{ab}(S_{\nu})_C$ is equivalent, as an abelian category, to (any) non-semisimple block of the category $\mathcal{C}_q$ (such a block exists if $q \neq \pm 1$ is a root of unity of even order). The indecomposable object $X_{\lambda^{(i)}}$ corresponds to the indecomposable tilting module $\mathrm{Q}_i$ (using the notation of Lemma \ref{lem:TL_cat_prop}).
 \end{itemize}

\end{theorem}

Using the theorem above, we can prove different properties of the category $\Dab$ in the following way:
\begin{enumerate}
 \item Reduce the proof to a block-by-block check;
 \item Prove the property for the semisimple blocks by checking that it holds for the category $Vect_{\bC}$.
 \item Prove the property for the non-semisimple blocks by importing the relevant result for the category $\mathcal{C}_q$.
\end{enumerate}

Using this approach, we prove that $\underline{Rep}^{ab}(S_{\nu})$ is a highest weight category \InnaA{(with infinitely many weights)}.

\begin{proposition}\label{prop:ab_envelope_highest_weight}
 The category $\underline{Rep}^{ab}(S_{\nu})$ is a highest weight category corresponding to the partially ordered set $(\{ \text{Young diagrams} \}, \geq )$, where
 $$\lambda \geq \mu \text{ iff } \lambda \stackrel{\nu}{\sim} \mu, \lambda \subset \mu$$
 (namely, $\lambda^{(i)} \geq \lambda^{(j)}$ if $i \leq j$).

\end{proposition}
\begin{proof}
 \InnaA{As it was said before,} this can be proved by checking each block separately. The semisimple blocks obviously satisfy the requirement; for the non-semisimple blocks, the theorem follows from the fact that $\mathcal{C}_q$ is a highest weight category and from Theorem \ref{thrm:blocks_S_nu}.
\end{proof}

\begin{proposition}\label{prop:proj_in_ab_envelope}
\mbox{}
\begin{enumerate}
 \item 
 In the category $\underline{Rep}^{ab}(S_{\nu})$ all projective objects are injective and conversely.
 \item All projective objects of $\underline{Rep}^{ab}(S_{\nu})$ lie in $\underline{Rep}(S_{\nu})$.
\end{enumerate}
\end{proposition}

\begin{proof}
The statement is obvious for semisimple blocks. For non-semisimple blocks, the first part follows from \cite[Theorem 9.12]{APW}.

To prove second part of the Proposition, we recall that the equivalence between non-semisimple blocks of $\underline{Rep}^{ab}(S_{\nu}), \mathcal{C}_q$, by definition restricts to an equivalence between non-semisimple blocks of $\underline{Rep}(S_{\nu}), TL(q)$ (see \cite{CO2}). Lemma \ref{lem:TL_cat_prop} states that the corresponding statement is true for $\mathcal{C}_q, \InnaA{TL(q)}$, and we are done.
\end{proof}
%

We will use the following notation for simple, standard, co-standard, and indecomposable projective objects in $\underline{Rep}^{ab}(S_{\nu})$:

\begin{notation}

Let $\lambda$ be any Young diagram. We will denote the simple (resp. standard, co-standard, indecomposable projective) object corresponding to $\lambda$ by $\mathbf{L}(\lambda)$ (resp. $\mathbf{M}(\lambda)$, \InnaB{$\mathbf{M}(\lambda)^*$}, $\mathbf{P}(\lambda)$).
\end{notation}
\begin{remark}
 \InnaB{We will show in Corollary \ref{cor:duality_abel_env} that the co-standard object $\mathbf{M}(\lambda)^*$ is the dual (in terms of the tensor structure of $\Dab$) of the standard object $\mathbf{M}(\lambda)$. This justifies the notation $\mathbf{M}(\lambda)^*$.}
\end{remark}

\begin{remark}
 Notice that if $\lambda$ lies in a non-trivial $\stackrel{\nu}{\sim}$-class, the objects $\mathbf{L}(\lambda^{(i)})$, $\mathbf{M}(\lambda^{(i)})$, $\InnaB{\mathbf{M}(\lambda^{(i)})^*}$, $\mathbf{P}(\lambda^{(i)})$ correspond to the modules $\mathrm{L}_i, \mathrm{M}_i, \mathrm{M}^{\vee}_i, \mathrm{P}_i \in \mathcal{C}_q$, respectively.
\end{remark}

\begin{proposition}\label{prop:obj_ab_env}
\mbox{}
\begin{enumerate}
   \item Assume $\lambda$ lies in a trivial $\stackrel{\nu}{\sim}$-class. Then $$X_{\lambda} \cong \mathbf{P}(\lambda) \cong \InnaB{\mathbf{M}(\lambda)^*} \cong \mathbf{M}(\lambda) \cong \mathbf{L}(\lambda) $$
 \item Let $ \{\lambda^{(i)}\}_{\InnaA{i \geq 0}}$ be a non-trivial $\stackrel{\nu}{\sim}$-class, and \InnaA{$\mathcal{B}_{\lambda}$} the corresponding block of $\Dab$. Then

 \begin{itemize}[leftmargin=*]
 \item $X_{\lambda^{(0)}} \cong \mathbf{L}(\lambda^{(0)}) \cong \mathbf{M}(\lambda^{(0)}) \cong \InnaB{\mathbf{M}(\lambda^{(0)})^*}$.
  \item For any $i \geq 0$, $\mathbf{P}(\lambda^{(i)}) \cong X_{\lambda^{(i+1)}}$.
  \item For any $i \geq 0$, we have short exact sequences
$$ 0 \rightarrow \mathbf{M}(\lambda^{(i+1)}) \rightarrow \mathbf{P}(\lambda^{(i)}) \rightarrow \mathbf{M}(\lambda^{(i)}) \rightarrow 0$$
$$ 0 \rightarrow \InnaB{\mathbf{M}(\lambda^{(i)})^*} \rightarrow \mathbf{P}(\lambda^{(i)}) \rightarrow \InnaB{\mathbf{M}(\lambda^{(i+1)})^*} \rightarrow 0$$
  \item The socle filtration of $\mathbf{P}(\lambda^{(i)})$ has successive quotients
$$\mathbf{L}(\lambda^{(i)}); \mathbf{L}(\lambda^{(i+1)}) \oplus \mathbf{L}(\lambda^{(i-1)}); \mathbf{L}(\lambda^{(i)})$$
if $i>0$,
and successive quotients $\mathbf{L}(\lambda^{(0)}); \mathbf{L}(\lambda^{(1)}); \mathbf{L}(\lambda^{(0)})$
if $i=0$.
 \end{itemize}
\end{enumerate}

\end{proposition}

\begin{proof}
\begin{enumerate}[leftmargin=*]
 \mbox{}
 \item Follows directly from Theorem \ref{thrm:blocks_ab_envelope}, part (1).

\item From Lemma \ref{lem:TL_cat_prop} and from the equivalence described in Theorem \ref{thrm:blocks_ab_envelope}, we immediately conclude that
 $$X_{\lambda^{(0)}} \cong \mathbf{L}(\lambda^{(0)}) \cong \mathbf{M}(\lambda^{(0)}) \cong \InnaB{\mathbf{M}(\lambda^{(0)})^*}$$

  From Lemma \ref{lem:TL_cat_prop}, we know that $\{X_{\lambda^{(i)}}\}_{i\geq 1}$ is the set of isomorphism classes of indecomposable projective objects in the block \InnaA{$\mathcal{B}_{\lambda}$}. In other words, there exists a bijective map $f: \bZ_+ \rightarrow \bZ_{>0}$ such that $X_{\lambda^{(f(i))}} \cong \mathbf{P}(\lambda^{(i)})$ for any $i>0$.

Lemma \ref{lem:TL_cat_prop} also tells us that there exists an injective map $\psi_{j}: \mathbf{M}(\lambda^{(j)}) \hookrightarrow X_{\lambda^{(j)}}$ for any $j \geq 0$.

  Composing $\psi_{f(i)}$ with the map $X_{\lambda^{(f(f(i)))}} \cong \mathbf{P}(\lambda^{(f(i))}) \twoheadrightarrow \mathbf{M}(\lambda^{(f(i))})$, we get a non-zero map $X_{\lambda^{(f(f(i)))}} \rightarrow X_{\lambda^{(f(i))}}$. Using Theorem \ref{thrm:blocks_S_nu}, we see that $\abs{ f(f(i)) - f(i) } \leq 1$ for any $i \geq 0$, i.e. $\abs{ f(i) - i } \leq 1$ for any $i \geq 1$ (since $f$ is surjective).

  Notice that $$\dim \Hom_{\Dab}(\mathbf{P}(\lambda^{(0)}),  \mathbf{M}(\lambda^{(0)})) = \dim \Hom_{\Dab}(X_{\lambda^{(f(0))}}, X_{\lambda^{(0)}}) \geq 1$$ which means that $f(0)=1$. Together with the condition: $\abs{ f(i) - i } \leq 1$ for any $i \geq 1$, this implies that $f(i) = i+1$ for any $i \geq 1$.

Thus we proved that for any $i \geq 0$, $\mathbf{P}(\lambda^{(i)}) \cong X_{\lambda^{(i+1)}}$.
We now use the BGG reciprocity for the highest weight category $\Dab$:
  \begin{sublemma}\label{sublem:BGG_Deligne_cat}
  \mbox{}
  \begin{enumerate}[label=(\alph*)]
   \item For any $M \in \Dab$, we have: $$\dim \Hom_{\Dab} (\mathbf{P}(\lambda^{(j)}), M) = \left[ M \InnaA{:} \mathbf{L}(\lambda^{(j)}) \right]$$ for any $j \geq 0$.
   \item $$\left( \mathbf{P}(\lambda^{(j)}) \InnaA{:} \mathbf{M}(\lambda^{(i)}) \right) = \left[ \mathbf{M}(\lambda^{(i)}) \InnaA{:} \mathbf{L}(\lambda^{(j)}) \right] = \left[ \InnaB{\mathbf{M}(\lambda^{(i)})^*} \InnaA{:} \mathbf{L}(\lambda^{(j)}) \right]$$ for any $i, j \geq 0$ (the brackets in left hand side denote multiplicity in the standard filtration).
\end{enumerate}
  \end{sublemma}
  \begin{proof}
   The proof is standard (see e.g. \cite[Theorem 3.9(c), Theorem 3.11]{H}).

 \end{proof}

  Applying Sublemma \ref{sublem:BGG_Deligne_cat}(a) to $M :=  \mathbf{P}(\lambda^{(i)}) \cong X_{\lambda^{(i+1)}}$ and using Theorem \ref{thrm:blocks_S_nu}, we see that the composition factors of $\mathbf{P}(\lambda^{(i)})$ are $$\mathbf{L}(\lambda^{(i)}), \mathbf{L}(\lambda^{(i+1)}), \mathbf{L}(\lambda^{(i-1)}), \mathbf{L}(\lambda^{(i)}) \, \text{ if  } i>0$$
and $$\mathbf{L}(\lambda^{(0)}), \mathbf{L}(\lambda^{(1)}), \mathbf{L}(\lambda^{(0)}) \, \text{ if  } i=0$$
  (notice that the cosocle of $\mathbf{P}(\lambda^{(i)})$ is necessarily $ \mathbf{L}(\lambda^{(i)})$ for any $i \geq 0$).

  Fix $i \geq 1$. We have a map $\mathbf{P}(\lambda^{(i)}) \twoheadrightarrow \mathbf{M}(\lambda^{(i)}) $ and a map $ \mathbf{M}(\lambda^{(i)}) \stackrel{\psi_i}{\hookrightarrow} \mathbf{P}(\lambda^{(i-1)})$.

  Comparing the composition factors of $\mathbf{P}(\lambda^{(i)}), \mathbf{P}(\lambda^{(i+1)})$, we conclude that one of the following holds:
  \begin{itemize}
   \item $\mathbf{M}(\lambda^{(i)}) \cong \mathbf{L}(\lambda^{(i)})$.
   \item The socle filtration of $\mathbf{M}(\lambda^{(i)})$ has successive quotients $\mathbf{L}(\lambda^{(i-1)}); \mathbf{L}(\lambda^{(i)})$.
  \end{itemize}

  Applying Sublemma \ref{sublem:BGG_Deligne_cat}(b) to $j:=i-1$, we see that $$ \left[ \mathbf{M}(\lambda^{(i)}) \InnaA{:} \mathbf{L}(\lambda^{(i-1)}) \right] = \left( \mathbf{P}(\lambda^{(i-1)}) \InnaA{:} \mathbf{M}(\lambda^{(i)}) \right) = 1$$
  Thus we conclude that the socle filtration of $\mathbf{M}(\lambda^{(i)})$ has successive quotients $\mathbf{L}(\lambda^{(i-1)}); \mathbf{L}(\lambda^{(i)})$.

  The socle filtration of the standard objects immediately implies that the following sequence is exact:
  $$ 0 \rightarrow \mathbf{M}(\lambda^{(i+1)}) \stackrel{\psi_{i+1}}{\longrightarrow} \mathbf{P}(\lambda^{(i)}) \longrightarrow \mathbf{M}(\lambda^{(i)}) \rightarrow 0$$

  Next, the socle filtration of $\InnaB{\mathbf{M}(\lambda^{(i)})^*}$ can be obtained from the socle filtration of $\mathbf{M}(\lambda^{(i)})$, and it has successive quotients $\mathbf{L}(\lambda^{(i)}); \mathbf{L}(\lambda^{(i-1)})$.

  Since $\mathbf{P}(\lambda^{(i-1)})$ is projective, we get a map $\phi: \mathbf{P}(\lambda^{(i-1)}) \rightarrow \InnaB{\mathbf{M}(\lambda^{(i)})^*}$ such that the following diagram is commutative:

  $$\xymatrix{ &\mathbf{P}(\lambda^{(i-1)}) \ar@{>>}[d]  \ar[r]^-{\phi} &\InnaB{\mathbf{M}(\lambda^{(i)})^*}   \ar@{>>}[dl]  \\ &\mathbf{L}(\lambda^{(i-1)})}$$
  The socle filtration of $\InnaB{\mathbf{M}(\lambda^{(i)})^*}$ then implies that the map $\phi$ is surjective, and we get an exact sequence $$ 0 \rightarrow \InnaB{\mathbf{M}(\lambda^{(i-1)})^*} \rightarrow \mathbf{P}(\lambda^{(i-1)}) \rightarrow \InnaB{\mathbf{M}(\lambda^{(i)})^*} \rightarrow 0$$

  Finally, the socle filtration of $\mathbf{P}(\lambda^{(i)})$ can be deduced from the above exact sequences and the socle filtrations of the standard and the co-standard objects.
 \end{enumerate}

\end{proof}

\begin{corollary}\label{cor:duality_abel_env}
 \InnaB{The co-standard object $\InnaB{\mathbf{M}(\lambda)^*}$ is the dual (in terms of the tensor structure of $\Dab$) of the standard object  $\InnaB{\mathbf{M}}(\lambda)$.}
\end{corollary}
\begin{proof}
 \InnaB{By the construction of the Deligne category $\underline{Rep}(S_{\nu})$, all the objects in $\underline{Rep}(S_{\nu})$ are automatically self-dual (c.f. \cite[Section 2.16]{D}, \cite[Section 2.2]{CO}). So Proposition \ref{prop:obj_ab_env} immediately implies that $\InnaB{\mathbf{M}(\lambda)^*}$ is the dual of $\InnaB{\mathbf{M}}(\lambda)$ if whenever $\lambda$ lies in a trivial $\stackrel{\nu}{\sim}$-class, or is minimal in its non-trivial $\stackrel{\nu}{\sim}$-class.
 
 It remains to check the case when $\lambda$ lies in a non-trivial $\stackrel{\nu}{\sim}$-class $ \{\lambda^{(i)}\}_{\InnaA{i \geq 0}}$, and $\lambda = \lambda^{(i)}, i>0$. Then we have an exact sequence 
 $$\mathbf{P}(\lambda^{(i+1)}) \stackrel{f}{\longrightarrow} \mathbf{P}(\lambda^{(i)}) \longrightarrow \mathbf{M}(\lambda^{(i)}) \rightarrow 0$$
 Since the category $\Dab$ is pre-Tannakian, the duality functor $X \mapsto X^*$ is contravariant and exact, and we conclude that the dual of $ \mathbf{M}(\lambda^{(i)})$ is the kernel of the map $ f^*: \mathbf{P}^*(\lambda^{(i)}) \rightarrow \mathbf{P}^*(\lambda^{(i+1)}) $. The autoduality of $\mathbf{P}(\lambda^{(i+1)}), \mathbf{P}(\lambda^{(i)})$, together with Proposition \ref{prop:obj_ab_env} and the fact that $f^* \neq 0$, immediately implies that the dual of $ \mathbf{M}(\lambda^{(i)})$ is $ \mathbf{M}(\lambda^{(i)})^*$, as wanted.
 }
\end{proof}

\section{Parabolic category $\co$}\label{sec:par_cat_o}
In this section, we present the results on the parabolic category $\co$ which we will use throughout the paper. The material of this section is mostly based on \cite[Chapter 9]{H}.

We start with some definitions.
\begin{definition}
 A unital vector space is a vector space $V$ with a distinguished \InnaB{non-zero vector $\InnaB{ \triv}$}.
\end{definition}

Fix a unital vector space $(V, \InnaB{\triv})$ with $\dim V  <\infty$.

\InnaB{Throughout the paper, we will use Notation \ref{notn:par_subalg} for the parabolic subalgebra $\mathfrak{p}_{(V, \InnaB{\bC \triv})}$ of $\gl(V)$, as well as the mirabolic subgroup $\bar{\mathfrak{P}}_{\triv}$, the mirabolic subalgebra $\bar{\mathfrak{p}}_{\InnaB{\bC \triv}}$, the unipotent group $\mathfrak{U}_{\triv}$ and the nilpotent subalgebra $\mathfrak{u}_{\mathfrak{p}}^{+}$.}

\begin{remark}
 Notice that $\mathfrak{p}$ has a one-dimensional center (scalar endomorphisms of $V$), and we have: $\mathfrak{p} \cong \bC \id_V \oplus \bar{\mathfrak{p}}_{\InnaB{\bC \triv}}$. 
\end{remark}

\InnaB{
We now want to talk about a subcategory of the category of finite-dimensional representations of the mirabolic group $\bar{\mathfrak{P}}_{\triv}$. For this, we will use the following lemma:

\begin{lemma}\label{lem:mirabolic_irr_repr}
 \begin{enumerate}[label=(\alph*)]
  \item There is a short exact sequence of groups $$ 1 \rightarrow \mathfrak{U}_{\triv} \longrightarrow \bar{\mathfrak{P}}_{\triv} \longrightarrow GL \left( \quotient{V}{\bC \triv} \right) \rightarrow 1$$
  \item For any irreducible finite-dimensional algebraic representation $\rho: \bar{\mathfrak{P}}_{\triv} \rightarrow \Aut(E)$ of the mirabolic subgroup, $\mathfrak{U}_{\triv}$ acts trivially on $E$, and thus $\rho$ factors through $GL \left( V / \bC \triv \right)$.
 \end{enumerate}
 
\end{lemma}

\begin{proof}
 In part (a), one only needs to check that this sequence is exact at $ \bar{\mathfrak{P}}_{\triv}$. This is obvious once we choose a splitting $V = \bC \triv \oplus U$; with a chosen splitting, this short exact sequence splits and we get an isomorphism $\bar{\mathfrak{P}}_{\triv} \cong GL(U) \ltimes U^*$.
 
 In part (b), recall that the group $\mathfrak{U}_{\triv}$ is unipotent (and even abelian), so there exists a non-zero vector $v \in E$ which is fixed by $\mathfrak{U}_{\triv}$. 
 
 Since $E$ is an irreducible finite-dimensional representation of $ \bar{\mathfrak{P}}_{\triv}$, $E$ has a basis consisting of vectors of the form $\rho(x)v$ for some $x \in \bar{\mathfrak{P}}_{\triv}$. On each such vector $\rho(x)v$, $\mathfrak{U}_{\triv}$ acts trivially; thus $\mathfrak{U}_{\triv}$ acts trivially on $E$.
 
%


\end{proof}


We now consider the category $Rep(\bar{\mathfrak{P}}_{\triv})$ of finite-dimensional algebraic representations of $\bar{\mathfrak{P}}_{\triv}$. In this category, we define:

\begin{definition}\label{def:poly_rep_mirabolic}
\mbox{}
\begin{itemize}[leftmargin=*]
 \item Let $\rho: \bar{\mathfrak{P}}_{\triv} \rightarrow \Aut(E)$ be an irreducible finite-dimensional algebraic representation of the mirabolic subgroup. The above lemma states that $\rho$ factors through $GL \left( V / \bC \triv \right)$. We say that $\rho$ is a {\it $GL \left( V / \bC \triv \right)$-polynomial representation of $\bar{\mathfrak{P}}_{\triv}$} if $\rho: GL \left( V / \bC \triv \right) \rightarrow \Aut(E)$ is a polynomial representation (c.f. Notation \ref{notn:action_on_tes_powers}). Recall that the latter condition is equivalent to saying that $\rho$ extends to an algebraic map $\End \left( V / \bC \triv \right) \rightarrow \End(E)$.
 
 \item The category of {\it $GL \left( V / \bC \triv \right)$-polynomial representations of $\bar{\mathfrak{P}}_{\triv}$} is defined to be the Serre subcategory of $Rep(\bar{\mathfrak{P}}_{\triv})$ generated by the irreducible $GL \left( V / \bC \triv \right)$-polynomial representations of $\bar{\mathfrak{P}}_{\triv}$.
 
 That is, a finite-dimensional algebraic representation $E$ of $\bar{\mathfrak{P}}_{\triv}$ is called {\it $GL \left( V / \bC \triv \right)$-polynomial} if the Jordan-Holder components of $E$ are $GL \left( V / \bC \triv \right)$-polynomial representations of $\bar{\mathfrak{P}}_{\triv}$.
 
 We denote the category of $GL \left( V / \bC \triv \right)$-polynomial representations of $\bar{\mathfrak{P}}_{\triv}$ by $Rep(\bar{\mathfrak{P}}_{\triv})_{GL \left( V / \bC \triv \right)-poly}$.
 
\end{itemize}
\end{definition}

We are now ready to give a definition of the parabolic category $\co$ which we are going to consider in this paper:}
\begin{definition}\label{def:par_cat_O}
 We define the parabolic category $\co$ for $(\gl(V), \mathfrak{p})$, denoted by $\co^{\mathfrak{p}}_{V}$, to be the full subcategory of $Mod_{\mathcal{U}(\gl(V))}$ whose objects $M$ satisfy the following conditions:
 \begin{itemize}
  \item \InnaB{$M$ is a Harish-Chandra module for the pair $(\gl(V), \bar{\mathfrak{P}}_{\triv})$, i.e. the action of the Lie subalgebra ${\mathfrak{p}}_{\InnaB{\bC \triv}}$ on $M$ integrates to the action of the mirabolic group $\bar{\mathfrak{P}}_{\triv}$.
 
 Furthermore, we require that as a representation of $\bar{\mathfrak{P}}_{\triv}$, $M$ be a filtered colimit of $GL \left( V / \bC \triv \right)$-polynomial representations, i.e. $$M \rvert_{\bar{\mathfrak{P}}_{\triv}} \in Ind-Rep(\bar{\mathfrak{P}}_{\triv})_{GL \left( V / \bC \triv \right)-poly}$$
  }
  \item \InnaA{$M$ is a finitely generated $\mathcal{U}(\gl(V))$-module.}
 \end{itemize}

 \InnaA{We will also use the notation $Ind-\co^{\mathfrak{p}}_{V}$ to denote the Ind-completion of $\co^{\mathfrak{p}}_{V}$ (i.e.  full subcategory of $Mod_{\mathcal{U}(\gl(V))}$ whose objects $M$ satisfy the first of the above conditions).}
\end{definition}
When the space $V$ is fixed, we will sometimes omit the subscript $V$ and write $\co^{\mathfrak{p}}$ for short.

We now fix a splitting $V = \InnaB{\bC \triv} \oplus U$. The Lie algebra $\InnaA{\bar{\mathfrak{p}}_{\InnaB{\bC \triv}}}$ can then be expressed as $\InnaA{\bar{\mathfrak{p}}_{\InnaB{\bC \triv}}} \cong \mathfrak{gl}(U) \ltimes U^*$, and we have: $\mathfrak{u}_{\mathfrak{p}}^{+} \cong U^*$, $\bar{\mathfrak{p}}_{\InnaB{\bC \triv}} \cong \mathfrak{u}_{\mathfrak{p}}^{+} \oplus \gl(U)$.

Moreover, in that case we have a splitting $\gl(V) \cong \mathfrak{p} \oplus \mathfrak{u}_{\mathfrak{p}}^{-}$, where $\mathfrak{u}_{\mathfrak{p}}^{-} \cong U$. This gives us an analogue of the triangular decomposition:
$$\gl(V) \cong \bC   \id_V \oplus \mathfrak{u}_{\mathfrak{p}}^{-} \oplus \mathfrak{u}_{\mathfrak{p}}^{+}  \oplus \gl(U)$$

We can now rewrite the above definition of the parabolic category $\co$ (compare with the usual definition in \cite[Section 9.3]{H}):
\begin{definition}\label{def:par_cat_O_splitting}
 We define the parabolic category $\co$ for $(\gl(V), \mathfrak{p})$, denoted by $\co^{\mathfrak{p}}_{V}$, to be the full subcategory of $Mod_{\mathcal{U}(\gl(V))}$ whose objects $M$ satisfy the following conditions:
 \begin{itemize}
 \item Viewed as a $\mathcal{U}(\gl(U))$-module, $M$ is a direct sum of polynomial simple $\mathcal{U}(\gl(U))$-modules (that is, $M$ belongs to $Ind-Mod_{\mathcal{U}(\gl(U)), poly}$).
  \item $M$ is locally finite over $\mathfrak{u}_{\mathfrak{p}}^{+}$.
  \item \InnaA{$M$ is a finitely generated $\mathcal{U}(\gl(V))$-module.}
 \end{itemize}
\end{definition}
\begin{remark}
 \InnaA{One can replace the requirement that $\mathfrak{u}_{\mathfrak{p}}^{+}$ act locally finitely on $M$ by the requirement that $\mathcal{U}(\mathfrak{u}_{\mathfrak{p}}^{+})$ act locally nilpotently on $M$}.
\end{remark}

The next propositions are based on \cite[Section 9.3]{H} as well:
\begin{proposition}
 The category $\co^{\mathfrak{p}}_{V}$ (resp. $Ind-\co^{\mathfrak{p}}_{V}$) is closed under taking direct sums, submodules, quotients and extensions in $\co_{\gl(V)}$, as well as tensoring with finite dimensional $\gl(V)$-modules. 
\end{proposition}

\InnaB{
Recall that in the category $\co$ we have the notion of a duality (c.f. \cite[Section 3.2]{H}): namely, given a $\gl(V)$-module $M$ with finite-dimensional weight spaces, we can consider the twisted action of $\gl(V)$ on the dual space $M^*$, given by $A.f := f \circ A^{T}$, where $A^T$ means the transpose of $A \in \gl(V)$. This gives us a $\gl(V)$-module $M^{\ast}$. We then take $M^{\vee}$ to be the maximal submodule of $M^{\ast}$ lying in the category $\co$. The module $M^{\vee}$ is called the dual of $M$ in $\co$, and we get an exact functor $(\cdot)^{\vee}: \co \rightarrow \co^{op}$. 

\begin{proposition}
 The category \InnaA{$\co^{\mathfrak{p}}_{V}$ is closed under taking duals, and} the duality functor $(\cdot)^{\vee}: \co^{\mathfrak{p}}_{V} \rightarrow \left( \co^{\mathfrak{p}}_{V} \right)^{op}$ is an equivalence of categories.
\end{proposition}}

\begin{definition}\label{def:degree_gl_module}
 A module $M$ over the Lie algebra $\gl(V)$ will be said to be of \InnaA{degree} $K \in \bC$ if $\id_V \in \gl(V)$ acts by $K   \id_M$ on $M$.
\end{definition}

We will denote by $\co^{\mathfrak{p}}_{\nu, V}$ the full subcategory of $\co^{\mathfrak{p}}_{V}$ whose objects are modules of \InnaA{degree} $\nu$.

\subsection{Structure of the category $\co^{\mathfrak{p}}_{\nu, V}$}\label{ssec:structure_cat_O}

In this subsection, we present some basic facts about the parabolic category $\co$ for $(\gl(V), \mathfrak{p})$ of \InnaA{degree} $\nu$ and the indecomposable objects inside it.

Fix $\nu \in \bC$, \InnaB{and let $(V, \triv)$ be a finite-dimensional unital vector space with a fixed splitting $V = \bC \triv \oplus U$.}

\begin{definition}\label{def:parabolic_Verma_mod}
Let $\lambda$ be a Young diagram.
\InnaA{
 $M_{\mathfrak{p}}(\nu-\abs{\lambda}, \lambda)$ is defined to be the $\gl(V)$-module $$\mathcal{U}(\gl(V)) \otimes_{\mathcal{U}(\mathfrak{p})} S^{\lambda} U$$ where $\gl(U)$ acts naturally on $S^{\lambda} U$, $\id_V \in \mathfrak{p}$ acts on $S^{\lambda} U$ by scalar $\nu$, and $\mathfrak{u}_{\mathfrak{p}}^{+}$ acts on $S^{\lambda} U$ by zero.
 
 Thus $M_{\mathfrak{p}}(\nu-\abs{\lambda}, \lambda)$ is the parabolic Verma module for $(\mathfrak{gl}(V), \mathfrak{p})$ with highest weight $(\nu-\abs{\lambda}, \lambda)$ iff $\dim V -1 \geq \ell(\lambda)$, and zero otherwise.
}
\end{definition}

 We will sometimes refer to the parabolic Verma modules as ``standard modules''.

\begin{definition}
 $L(\nu-\abs{\lambda}, \lambda)$ is defined to be zero (if $\dim V -1 < \ell(\lambda)$), or the simple module for $\mathfrak{gl}(V)$ of highest weight $(\nu-\abs{\lambda}, \lambda)$ otherwise.
\end{definition}

The following basic lemma will be very helpful:
\begin{lemma}\label{lem:gl_u_struct_o_cat}
 Let $\lambda$ such that $  \ell(\lambda) \leq \dim V -1$. We then have an isomorphism of $\gl(U)$-modules:
$$M_{\mathfrak{p}}(\nu-\abs{\lambda}, \lambda) \cong SU \otimes S^{\lambda} U$$
\end{lemma}

\begin{proof}
Follows directly from the definition of $M_{\mathfrak{p}}(\nu-\abs{\lambda}, \lambda)$ and the PBW theorem for $\gl(V)$.
\end{proof}

\begin{proposition}\label{prop:par_cat_O_Vermas}
\mbox{}
\begin{enumerate}
\item Let $\mu, \T$ be two Young diagrams. Then $$\dim \Hom_{\co^{\mathfrak{p}}_{\nu, V}} \left( M_{\mathfrak{p}}(\nu-\abs{\mu}, \mu), M_{\mathfrak{p}}(\nu-\abs{\T}, \T) \right) =0$$
if $\mu$, $\T$ lie in different $\stackrel{\nu}{\sim}$-classes.

\item Fix a non-trivial $\stackrel{\nu}{\sim}$-class $ \{\lambda^{(i)}\}$, $\lambda^{(0)} \subset \lambda^{(1)} \subset \lambda^{(2)} \subset ...$. For any $i, j \in \bZ_+$, we have
\begin{align*}
&\dim \Hom_{\co^{\mathfrak{p}}_{\nu, V}} \left( M_{\mathfrak{p}}(\nu-\abs{\lambda^{(j)}}, \lambda^{(j)}), M_{\mathfrak{p}}(\nu-\abs{\lambda^{(i)}}, \lambda^{(i)}) \right) = 0 \, \text{ if } j \neq i, i+1, \text{ and }\\
&\dim \Hom_{\co^{\mathfrak{p}}_{\nu, V}} \left( M_{\mathfrak{p}}(\nu-\abs{\lambda^{(i+1)}}, \lambda^{(i+1)}), M_{\mathfrak{p}}(\nu-\abs{\lambda^{(i)}}, \lambda^{(i)}) \right) = 1 \, \text{ if } \dim V  -1  \geq \ell(\lambda^{(i+1)})
\end{align*}
\item \InnaA{For any Young diagram $\lambda$ such that $\dim V  -1  \geq \ell(\lambda) $, we have:
$$\dim \End_{\co^{\mathfrak{p}}_{\nu, V}} \left( M_{\mathfrak{p}}(\nu-\abs{\lambda}, \lambda) \right) =1$$}
\end{enumerate}
\end{proposition}
\begin{proof}
\begin{enumerate}[leftmargin=*]
\item Consider a $\mathfrak{gl}(V)$-morphism $M_{\mathfrak{p}}(\nu-\abs{\mu}, \mu) \longrightarrow M_{\mathfrak{p}}(\nu-\abs{\T}, \T)$, and assume it is not zero. Then the weights $(\nu-\abs{\mu}, \mu)$, $(\nu-\abs{\T}, \T)$ are $W$-linked, i.e. there exists an element $w$ in the Weyl group such that $ w ((\nu-\abs{\mu}, \mu) +\rho) = (\nu-\abs{\T}, \T) +\rho$, where $\rho = (\dim V , \dim V -1,\dim V -2, ...) = \dim V (1, 1, 1,...) -(1, 2, 3,...)$.

This is equivalent to saying that $w(\nu-\abs{\mu}, \mu_1-1, \mu_2-2, ...) = (\nu-\abs{\T}, \T_1-1, \T_2-2, ...)$, which means that $\mu$, $\T$ lie in the same $\stackrel{\nu}{\sim}$-class (in fact, we get that $w =(1,2, ..., k) \in S_{\dim V } = Weyl(\mathfrak{gl}(V))$ for some $k>1$).

\item
Consider a non-trivial $\stackrel{\nu}{\sim}$-class $ \{\lambda^{(i)}\}_i, \lambda^{(0)} \subset \lambda^{(1)} \subset \lambda^{(2)} \subset ...$ (recall that this can occur only if $\nu \in \bZ_+$). Let $i, j \geq 0$, and assume there is a non-zero $\mathfrak{gl}(V)$-morphism $M_{\mathfrak{p}}(\nu-\abs{\lambda^{(j)}}, \lambda^{(j)}) \longrightarrow M_{\mathfrak{p}}(\nu-\abs{\lambda^{(i)}}, \lambda^{(i)})$.

Frobenius reciprocity then gives us a morphism of $\mathfrak{gl}(U)$-modules: $$S^{ \lambda^{(j)}}U \longrightarrow M_{\mathfrak{p}}(\nu-\abs{\lambda^{(i)}}, \lambda^{(i)}) \rvert_{\mathfrak{gl}(U)}$$ By Lemma \ref{lem:gl_u_struct_o_cat}, the \InnaA{$\gl(U)$-module $M_{\mathfrak{p}}(\nu-\abs{\lambda^{(i)}}, \lambda^{(i)}) \rvert_{\mathfrak{gl}(U)}$} is either zero (if $\dim(U) = \dim V  -1  < \ell(\lambda^{(i)})$), or isomorphic to $$ S U \otimes S^{\lambda^{(i)}} U \cong \bigoplus_{\mu \in \mathcal{I}^{+}_{\lambda^{(i)}}} S^{\mu} U$$

We immediately conclude that $\lambda^{(i)} \subset \lambda^{(j)}$ (which means that $i \leq j$), and that $\lambda^{(j)} \in \mathcal{I}^{+}_{\lambda^{(i)}}$. In fact, Lemma \ref{lem:nu_classes_struct} implies that $j=i+1$ in that case, since for $j \geq i+2$, we get:  $\lambda^{(j)}_k = \lambda^{(i)}_k+1$ for any $k=i+2, ..., j$, contradicting $\lambda^{(j)} \in \mathcal{I}^{+}_{\lambda^{(i)}}$.

It remains to check that
$$\dim \Hom_{\co^{\mathfrak{p}}_{\nu, V}} \left( M_{\mathfrak{p}}(\nu-\abs{\lambda^{(i+1)}}, \lambda^{(i+1)}), M_{\mathfrak{p}}(\nu-\abs{\lambda^{(i)}}, \lambda^{(i)}) \right) = 1 \, \text{ if } \dim V  -1  \geq \ell(\lambda^{(i+1)})
$$
We start by noticing that the same Frobenius reciprocity argument used above guarantees us that $$\dim \Hom_{\co^{\mathfrak{p}}_{\nu, V}} \left( M_{\mathfrak{p}}(\nu-\abs{\lambda^{(i+1)}}, \lambda^{(i+1)}), M_{\mathfrak{p}}(\nu-\abs{\lambda^{(i)}}, \lambda^{(i)}) \right) \leq 1$$ so we only need to check that if $\dim V  -1  \geq \ell(\lambda^{(i+1)})$, then there exists a non-zero morphism $$M_{\mathfrak{p}}(\nu-\abs{\lambda^{(i+1)}}, \lambda^{(i+1)}) \longrightarrow M_{\mathfrak{p}}(\nu-\abs{\lambda^{(i)}}, \lambda^{(i)})$$

This statement can be proved by induction on $i \geq 0$.

Base: Assume $\dim V  -1  \geq \ell(\lambda^{(0)})$. We need to check that $M_{\mathfrak{p}}(\nu-\abs{\lambda^{(0)}}, \lambda^{(0)})$ is not simple, i.e. isn't equal to $L(\nu-\abs{\lambda^{(0)}}, \lambda^{(0)})$. But the latter is finite-dimensional (since $(\nu-\abs{\lambda^{(0)}}, \lambda^{(0)})$ is an integral dominant weight), while $M_{\mathfrak{p}}(\nu-\abs{\lambda^{(0)}}, \lambda^{(0)})$ clearly isn't finite-dimensional (due to Lemma \ref{lem:gl_u_struct_o_cat}, for example).

Step: Let $i \geq 1$, and assume $\dim V  -1  \geq \ell(\lambda^{(i+1)})$. If there exists a non-zero morphism $M_{\mathfrak{p}}(\nu-\abs{\lambda^{(i)}}, \lambda^{(i)}) \longrightarrow M_{\mathfrak{p}}(\nu-\abs{\lambda^{(i-1)}}, \lambda^{(i-1)})$, then this morphism is not injective (can be seen from Lemma \ref{lem:gl_u_struct_o_cat}), therefore $M_{\mathfrak{p}}(\nu-\abs{\lambda^{(i)}}, \lambda^{(i)})$ is not simple; so there exists a non-zero morphism $M_{\mathfrak{p}}(\nu-\abs{\lambda^{(i+1)}}, \lambda^{(i+1)}) \longrightarrow M_{\mathfrak{p}}(\nu-\abs{\lambda^{(i)}}, \lambda^{(i)})$, as needed.
\item \InnaA{This statement follows immediately from Lemma \ref{lem:gl_u_struct_o_cat}, which gives us an isomorphism of of $\gl(U)$-modules:
$$M_{\mathfrak{p}}(\nu-\abs{\lambda}, \lambda) \cong SU \otimes S^{\lambda} U \cong \bigoplus_{\mu \in \mathcal{I}^{+}_{\lambda}} S^{\mu} U$$}

\end{enumerate}

\end{proof}

The previous proposition immediately implies:
\begin{corollary}\label{cor:parab_Verma_simple}

Let $\lambda$ lie in a trivial $\stackrel{\nu}{\sim}$-class. Then $M_{\mathfrak{p}}(\nu-\abs{\lambda}, \lambda)$ is either zero (iff $\dim V -1 < \ell(\lambda)$), or a simple $\mathfrak{gl}(V)$-module. In particular, if $\nu \notin \bZ_+$, this is true for any Young diagram $\lambda$.
\end{corollary}
\begin{proof}
 Recall that since $\nu \notin \bZ_+$, each Young diagram $\lambda$ lies in a trivial $\stackrel{\nu}{\sim}$-class (see Lemma \ref{lem:nu_classes_struct}). The result follows from Proposition \ref{prop:par_cat_O_Vermas}.
\end{proof}

\begin{remark}
 Note that Proposition \ref{prop:par_cat_O_Vermas} implies that the category $\co^{\mathfrak{p}}_{\nu, V}$ decomposes into blocks (each of the blocks is an abelian category in its own right). To each $\stackrel{\nu}{\sim}$-class of Young diagrams corresponds a block of $\co^{\mathfrak{p}}_{\nu, V}$ (if for all Young diagrams $\lambda$ in this $\stackrel{\nu}{\sim}$-class, $\ell(\lambda) > \dim V  -1$, then the corresponding block is zero), and to each non-zero block of $\co^{\mathfrak{p}}_{\nu, V}$ corresponds a unique $\stackrel{\nu}{\sim}$-class. 
 
 Proposition \ref{prop:par_cat_O_Vermas} also implies that the block corresponding to a trivial $\stackrel{\nu}{\sim}$-class is either semisimple (i.e. equivalent to the category $Vect_{\bC}$), or zero. 
\end{remark}

Now fix a non-trivial $\stackrel{\nu}{\sim}$-class $ \{\lambda^{(i)}\}_i$, and $i \geq 0$ such that $ \ell(\lambda^{(i\InnaA{+1})}) \leq \dim V -1 $.

Proposition \ref{prop:par_cat_O_Vermas} implies that the maximal non-trivial submodule of $M_{\mathfrak{p}}(\nu-\abs{\lambda^{(i)}}, \lambda^{(i)})$ is $L(\nu-\abs{\lambda^{(i+1)}}, \lambda^{(i+1)})$. We conclude that

\begin{corollary}\label{cor:parab_Verma_ses}
Let $ \{\lambda^{(i)}\}_i$ be a non-trivial $\stackrel{\nu}{\sim}$-class, and $i \geq 0$ be such that $ \ell(\lambda^{(i)}) \leq \dim V -1 $.

Then there is a short exact sequence
$$ 0 \rightarrow L(\nu-\abs{\lambda^{(i+1)}}, \lambda^{(i+1)}) \rightarrow M_{\mathfrak{p}}(\nu-\abs{\lambda^{(i)}}, \lambda^{(i)}) \rightarrow L(\nu-\abs{\lambda^{(i)}}, \lambda^{(i)}) \rightarrow 0$$

\end{corollary}

\begin{remark}
 Notice that for $i : =\max \{ i\geq 0 \mid \ell(\lambda^{(i)}) \leq \dim V  -1 \}$, we have $$M_{\mathfrak{p}}(\nu-\abs{\lambda^{(i)}}, \lambda^{(i)}) \cong M^{\vee}_{\mathfrak{p}}(\nu-\abs{\lambda^{(i)}}, \lambda^{(i)}) \cong L(\nu-\abs{\lambda^{(i)}}, \lambda^{(i)})$$ 
\end{remark}

We also get the BGG resolution in category $\co^{\mathfrak{p}}_{\nu, V}$ as an immediate corollary:

\begin{corollary}\label{cor:parab_Verma_les}
Let $ \{\lambda^{(i)}\}_i$ be a non-trivial $\stackrel{\nu}{\sim}$-class. Then there is a long exact sequence of $\mathfrak{gl}(V)$-modules (BGG resolution of $L(\nu-\abs{\lambda^{(i)}}, \lambda^{(i)})$ by parabolic Verma modules)
\begin{align*}
 &... \rightarrow M_{\mathfrak{p}}(\nu-\abs{\lambda^{(i+2)}}, \lambda^{(i+2)}) \rightarrow M_{\mathfrak{p}}(\nu-\abs{\lambda^{(i+1)}}, \lambda^{(i+1)}) \rightarrow M_{\mathfrak{p}}(\nu-\abs{\lambda^{(i)}}, \lambda^{(i)}) \rightarrow \\
&\rightarrow L(\nu-\abs{\lambda^{(i)}}, \lambda^{(i)}) \rightarrow 0
\end{align*}
\end{corollary}
\begin{proof}
 Follows immediately from Corollary \ref{cor:parab_Verma_ses}.
\end{proof}
\begin{remark}
 For $i=0$, such a resolution is a special case of BGG resolutions in parabolic category $\co$ discussed in \cite[Chapter 9, Par. 16]{H}.
\end{remark}

We now consider the projective cover $P_{\mathfrak{p}}(\nu-\abs{\lambda}, \lambda)$ of $L(\nu-\abs{\lambda}, \lambda)$ in $\co^{\mathfrak{p}}_{\nu, V}$. The existence of $P_{\mathfrak{p}}(\nu-\abs{\lambda}, \lambda)$ and some of its properties are listed in the following proposition:

\begin{proposition}\label{prop:proj_parab_cat_O_gen}
 \mbox{}
 \begin{enumerate}[leftmargin=*, label=(\alph*)]
  \item Category $\co^{\mathfrak{p}}_{\nu, V}$ has enough projectives; in particular, there exists a projective cover of $L(\nu-\abs{\lambda}, \lambda)$, which will be denoted by $P_{\mathfrak{p}}(\nu-\abs{\lambda}, \lambda)$.
    \item For any Young diagram $\lambda$, the following equality holds:
  $$\dim \Hom_{\co^{\mathfrak{p}}_{\nu, V}}(P_{\mathfrak{p}}(\nu-\abs{\lambda}, \lambda), M) = [M: L(\nu-\abs{\lambda}, \lambda)]$$
  \item The projective module $P_{\mathfrak{p}}(\nu-\abs{\lambda}, \lambda)$ is indecomposable and standardly filtered (i.e. has a filtration where all the successive quotients are parabolic Verma modules).
  \item (BGG reciprocity) The following equality holds for any Young diagrams $\lambda, \mu$:
  $$\left( P_{\mathfrak{p}}(\nu-\abs{\lambda}, \lambda): M_{\mathfrak{p}}(\nu-\abs{\mu}, \mu)\right) = \left[ M_{\mathfrak{p}}(\nu-\abs{\mu}, \mu):L(\nu-\abs{\lambda}, \lambda) \right] $$
  (the brackets in left hand side denote multiplicity in the standard filtration).
  \item The duality functor $(\cdot)^{\vee}: \co^{\mathfrak{p}}_{\nu, V} \rightarrow \left( \co^{\mathfrak{p}}_{\nu, V} \right)^{op}$ takes projective modules to injective modules and vice versa. In particular, there are enough injectives in the category $\co^{\mathfrak{p}}_{\nu, V}$, and the indecomposable injective modules are exactly $P^{\vee}_{\mathfrak{p}}(\nu-\abs{\lambda}, \lambda)$ (which is the injective hull of $L(\nu-\abs{\lambda}, \lambda)$).
  \item Whenever $\lambda$ is not the minimal Young diagram in a non-trivial $\stackrel{\nu}{\sim}$-class, the modules $P_{\mathfrak{p}}(\nu-\abs{\lambda}, \lambda)$ are self-dual and therefore injective. In these cases the following equality holds:
  $$\dim \Hom_{\co^{\mathfrak{p}}_{\nu, V}}(M, P_{\mathfrak{p}}(\nu-\abs{\lambda}, \lambda)) = [M: L(\nu-\abs{\lambda}, \lambda)]$$
 \end{enumerate}

\end{proposition}

\begin{proof}
 The proofs of (a) - (e) can be found in \cite[Chapter 9, Par. 8 and Chapter 3, Par. 9-11]{H}; the proof of the first part of (f) is based on \cite[Chapter 9, Par. 14]{H} and on Corollary \ref{cor:parab_Verma_ses}. The equality in (f) can then be proved in the same way as the equality in (b).
\end{proof}

We can now describe the standard filtration of the indecomposable projectives 

$P_{\mathfrak{p}}(\nu-\abs{\lambda}, \lambda)$, and their other useful properties:
\begin{proposition}\label{prop:proj_parab_cat_O_struct}
 \mbox{}
 \begin{enumerate}[leftmargin=*, label=(\alph*)]
 \item Assume $\lambda$ lies in a trivial $\stackrel{\nu}{\sim}$-class. Then $$P_{\mathfrak{p}}(\nu-\abs{\lambda}, \lambda) \cong M_{\mathfrak{p}}(\nu-\abs{\lambda}, \lambda) = L(\nu-\abs{\lambda}, \lambda) $$
 \item Let $ \{\lambda^{(i)}\}_i$ be a non-trivial $\stackrel{\nu}{\sim}$-class. Then $$P_{\mathfrak{p}}(\nu-\abs{\lambda^{(0)}}, \lambda^{(0)}) \cong M_{\mathfrak{p}}(\nu-\abs{\lambda^{(0)}}, \lambda^{(0)})$$
  \item Let $ \{\lambda^{(i)}\}_i$ be a non-trivial $\stackrel{\nu}{\sim}$-class and let $i \geq 1$. Then for $i$ such that $\ell(\lambda^{(i)}) \leq \dim V  -1$, we have short exact sequences
$$ 0 \rightarrow M_{\mathfrak{p}}(\nu-\abs{\lambda^{(i-1)}}, \lambda^{(i-1)}) \rightarrow P_{\mathfrak{p}}(\nu-\abs{\lambda^{(i)}}, \lambda^{(i)}) \rightarrow M_{\mathfrak{p}}(\nu-\abs{\lambda^{(i)}}, \lambda^{(i)}) \rightarrow 0$$

$$ 0 \rightarrow M^{\vee}_{\mathfrak{p}}(\nu-\abs{\lambda^{(i)}}, \lambda^{(i)}) \rightarrow P_{\mathfrak{p}}(\nu-\abs{\lambda^{(i)}}, \lambda^{(i)}) \rightarrow M^{\vee}_{\mathfrak{p}}(\nu-\abs{\lambda^{(i-1)}}, \lambda^{(i-1)}) \rightarrow 0$$

and the socle filtration of $P_{\mathfrak{p}}(\nu-\abs{\lambda^{(i)}}, \lambda^{(i)})$ has successive quotients
$$L(\nu-\abs{\lambda^{(i)}}, \lambda^{(i)}); L(\nu-\abs{\lambda^{(i+1)}}, \lambda^{(i+1)}) \oplus L(\nu-\abs{\lambda^{(i-1)}}, \lambda^{(i-1)}); L(\nu-\abs{\lambda^{(i)}}, \lambda^{(i)})$$

For $i$ such that $\ell(\lambda^{(i)}) > \dim V  -1$, $$P_{\mathfrak{p}}(\nu-\abs{\lambda^{(i)}}, \lambda^{(i)}) = M_{\mathfrak{p}}(\nu-\abs{\lambda^{(i)}}, \lambda^{(i)}) = M^{\vee}_{\mathfrak{p}}(\nu-\abs{\lambda^{(i)}}, \lambda^{(i)}) = L(\nu-\abs{\lambda^{(i)}}, \lambda^{(i)}) =0$$
  \item Let $ \{\lambda^{(i)}\}_i$ be a non-trivial $\stackrel{\nu}{\sim}$-class, and let $i \geq 1, j \geq 0$. Then
   \InnaA{$$ \dim \Hom_{\co^{\mathfrak{p}}_{\nu, V}} \left( P_{\mathfrak{p}}(\nu-\abs{\lambda^{(j)}}, \lambda^{(j)}), P_{\mathfrak{p}}(\nu-\abs{\lambda^{(i)}}, \lambda^{(i)}) \right) = 
  \begin{cases}                                                                                                                                                                       2 &\text{ if } i=j,  \\
 &\text{    }\ell(\lambda^{(i)}) \leq  \dim V  -1\\                                                                                                                                                                   1 &\text{ if } \abs{i-j}=1, \\
 &\text{    } \ell(\lambda^{(i)}), \ell(\lambda^{(j)}) \leq  \dim V  -1\\        
  0 &\text{ else}
  \end{cases}
  $$}
%
 \end{enumerate}

\end{proposition}

\begin{proof}
Parts (a), (b) follow directly from the fact that $ P_{\mathfrak{p}}(\nu-\abs{\lambda^{(i)}}, \lambda^{(i)})$ is standardly filtered, the BGG reciprocity (see Proposition \ref{prop:proj_parab_cat_O_gen}) and Proposition \ref{prop:par_cat_O_Vermas}. The BGG reciprocity also implies that for a non-trivial $\stackrel{\nu}{\sim}$-class, denoted by  $ \{\lambda^{(i)}\}_i$, we have the short exact sequence
$$ 0 \rightarrow M_{\mathfrak{p}}(\nu-\abs{\lambda^{(i-1)}}, \lambda^{(i-1)}) \rightarrow P_{\mathfrak{p}}(\nu-\abs{\lambda^{(i)}}, \lambda^{(i)}) \rightarrow M_{\mathfrak{p}}(\nu-\abs{\lambda^{(i)}}, \lambda^{(i)}) \rightarrow 0$$ whenever $i \geq 1$. Taking duals in of the modules in this sequence, we obtain a short exact sequence
$$ 0 \rightarrow M^{\vee}_{\mathfrak{p}}(\nu-\abs{\lambda^{(i)}}, \lambda^{(i)}) \rightarrow P_{\mathfrak{p}}(\nu-\abs{\lambda^{(i)}}, \lambda^{(i)}) \rightarrow M^{\vee}_{\mathfrak{p}}(\nu-\abs{\lambda^{(i-1)}}, \lambda^{(i-1)}) \rightarrow 0$$

To compute the socle filtration, notice that $$Soc(P_{\mathfrak{p}}(\nu-\abs{\lambda^{(i)}}, \lambda^{(i)})) \cong L(\nu-\abs{\lambda^{(i)}}, \lambda^{(i)})$$ since $$\dim \Hom_{\co^{\mathfrak{p}}_{\nu, V}}(L(\nu-\abs{\lambda^{(j)}}, \lambda^{(j)}), P_{\mathfrak{p}}(\nu-\abs{\lambda^{(i)}}, \lambda^{(i)})) = [L(\nu-\abs{\lambda^{(j)}}, \lambda^{(j)}): L(\nu-\abs{\lambda^{(i)}}, \lambda^{(i)})] = \delta_{i,j}$$ (see Proposition \ref{prop:proj_parab_cat_O_gen}(f)).

The short exact sequences above then imply that $$Soc^2(P_{\mathfrak{p}}(\nu-\abs{\lambda^{(i)}}, \lambda^{(i)})) \cong L(\nu-\abs{\lambda^{(i+1)}}, \lambda^{(i+1)}) \oplus L(\nu-\abs{\lambda^{(i-1)}}, \lambda^{(i-1)})$$
The socle filtration description then follows, and this proves Part (c).

The dimensions of the $\Hom$-spaces between the indecomposable projectives can be inferred from the socle filtration and Proposition \ref{prop:proj_parab_cat_O_gen}(b, f), which proves Part (d).
\end{proof}

\section{Complex tensor powers of a unital object}\label{sec:comp_tens_power}
\InnaB{In this section, we fix a finite-dimensional unital vector space $(V, \triv)$. The goal of this section is to define an object $V^{\underline{\otimes}  \nu}$ which will be an interpolation of the tensor powers $V^{\otimes n}$ for $n \in \bZ_+$ to arbitrary $\nu \in \bC$.}

\subsection{\InnaA{Description} of the category $Ind-(\underline{Rep}^{ab}(S_{\nu}) \boxtimes \co^{\mathfrak{p}}_{\nu, V})$}\label{ssec:def_cat_Del_dash_par_O}

In this subsection we will \InnaB{describe the category $Ind-(\underline{Rep}^{ab}(S_{\nu}) \boxtimes \co^{\mathfrak{p}}_{\nu, V})$, which will be used to define the complex tensor power of the vector space $V$. 

The notation $\boxtimes$ stands for the Deligne tensor product of abelian categories (c.f. \cite[Section 5]{D1}); in this subsection, we will explain that this category can also be described as the category of Ind-objects of $\underline{Rep}^{ab}(S_{\nu})$ carrying an ``$\co^{\mathfrak{p}}_{\nu, V}$-type'' action of $\gl(V)$.}

\InnaB{Fix a splitting $V = \bC \triv \oplus U$.}
\begin{definition}
 Let $X \in Ind-\underline{Rep}^{ab}(S_{\nu})$. We say that $\gl(V)$ acts on $X$ if given a homomorphism of $\bC$-algebras $$\rho_X: \mathcal{U}(\gl(V)) \rightarrow \End_{Ind-\underline{Rep}^{ab}(S_{\nu})}(X)$$

 We say that this action is an $\co^{\mathfrak{p}}_{\nu, V}$-action if:
 \begin{itemize}
  \item $\gl(U) \subset \gl(V)$ acts polynomially on $X$, i.e. $X$ decomposes as a direct sum $X \cong \bigoplus_{i \in I} Y_i \otimes E_i$ in $Ind-\underline{Rep}^{ab}(S_{\nu})$, with
  \begin{enumerate}
   \item $Y_i \in Ind-\underline{Rep}^{ab}(S_{\nu})$,
   \item $E_i$ being polynomial $\gl(U)$-modules,
  \end{enumerate}
and the following commutative diagram holds for any $a \in \mathcal{U}(\gl(U))$:
$$ \begin{CD}
    X @>>> \bigoplus_{i \in I} Y_i \otimes E_i \\
    @V\rho_X(a)VV @V{\oplus_i a\rvert_{E_i}}VV \\
    X @>>> \bigoplus_{i \in I} Y_i \otimes E_i
   \end{CD}
$$
\item $\mathcal{U}(\mathfrak{u}_{\mathfrak{p}}^{+})$ acts locally finitely on $X$, i.e. for any $Y \in \Dab$, $f: Y \rightarrow X$, we have: $$ \sum_{a \in \mathcal{U}(\mathfrak{u}_{\mathfrak{p}}^{+})} (\rho_X(a)\circ f)(Y)$$ belongs to $\Dab$ (i.e. is a compact object).
\item $\id_V$ acts on $X$ by the morphism $\nu \cdot \id_X$.
\end{itemize}

\end{definition}

 The category $Ind-(\underline{Rep}^{ab}(S_{\nu}) \boxtimes \co^{\mathfrak{p}}_{\nu, V})$ is the category of pairs $(X, \rho_X)$ where $X \in Ind-\underline{Rep}^{ab}(S_{\nu})$ and $\rho_X$ is an $\co^{\mathfrak{p}}_{\nu, V}$-action on $X$. The morphisms in $Ind-(\underline{Rep}^{ab}(S_{\nu}) \boxtimes \co^{\mathfrak{p}}_{\nu, V})$ are
 $$\Hom((X, \rho_X), (Y, \rho_Y)):=\{f \in \Hom_{Ind-\underline{Rep}^{ab}(S_{\nu})}(X, Y) \, \rvert  \, f \circ \rho_X(a) = \rho_Y(a) \circ f \, \forall a \in \mathcal{U}(\gl(V)) \}$$

\InnaA{
\begin{remark}
 Inside the category $Ind-(\underline{Rep}^{ab}(S_{\nu}) \boxtimes \co^{\mathfrak{p}}_{\nu, V})$ we have the full subcategory $Ind-(\underline{Rep}(S_{\nu}) \boxtimes \co^{\mathfrak{p}}_{\nu, V})$ whose objects are $(X, \rho_X)$ where $X \in Ind-\underline{Rep}(S_{\nu})$ and $\rho_X$ is an $\co^{\mathfrak{p}}_{\nu, V}$-action on $X$. 
\end{remark}

Now let $\nu =n \in \bZ_+$.
%

By \cite[Corollary 6.3.2]{KS}, there exists a unique (up to unique isomorphism) functor $$\hat{\mathcal{S}}_n: (Ind-\underline{Rep}(S_{\nu=n})) \longrightarrow Ind-Rep(S_n)$$
which commutes with (small) filtered colimits and satisfies $$\hat{\mathcal{S}}_n \circ \iota_{\underline{Rep}(S_{\nu}) \rightarrow (Ind-\underline{Rep}(S_{\nu}))} \cong \mathcal{S}_n$$ (the notation is the same as in Section \ref{sec:notation} and Subsection \ref{ssec:S_nu_general}). 

Now consider the category $Ind- (Rep(S_n) \boxtimes \co^{\mathfrak{p}}_{n, V})$; it is the full subcategory of the category of \InnaA{modules $Mod_{\bC[S_n]\otimes_{\bC} \mathcal{U}(\gl(V))}$} whose objects $M$ satisfy: as a $\bC[S_n]$-module, $M$ is a direct sum of finite dimensional simple modules, while as a $\mathcal{U}(\gl(V))$-module, $M \in \co^{\mathfrak{p}}_{n, V}$.

We can define the functor $$\hat{\mathcal{S}}_n: Ind- (\underline{Rep}(S_{\nu=n}) \boxtimes \co^{\mathfrak{p}}_{n, V}) \longrightarrow Ind -(Rep(S_n)\boxtimes \co^{\mathfrak{p}}_{n, V})$$
by setting $\hat{\mathcal{S}}_n(X, \rho_X) :=\hat{\mathcal{S}}_n(X)$ (with action of $\gl(V)$ given by $\hat{\mathcal{S}}_n(\rho_X)$). The above description of $Ind- (\underline{Rep}(S_{\nu=n}) \boxtimes \co^{\mathfrak{p}}_{n, V})$ guarantees that this functor is well-defined.
}

\subsection{Definition of a complex tensor power: split unital vector space}\label{ssec:def_comp_ten_power_split}
Fix $\nu \in \bC$, and fix a splitting $V = \InnaB{\bC \triv} \oplus U$.

\begin{definition}[Complex tensor power]\label{def:complex_ten_power_splitting}
 
 Define the object $V^{\underline{\otimes}  \nu}$ of $Ind-(\underline{Rep}^{ab}(S_{\nu}) \boxtimes \co^{\mathfrak{p}}_{\nu, V})$ \InnaA{by setting}
$$ V^{\underline{\otimes}  \nu} := \bigoplus_{k \geq 0} (U^{\otimes k} \otimes \Del_k)^{S_k}$$

The action on $\mathfrak{gl}(V)$ on $V^{\underline{\otimes}  \nu}$ is given as follows:

\InnaA{$$ \entrymodifiers={+!!<0pt,\fontdimen22\textfont2>} \xymatrix{ &{\mathbf{1}\;} \ar@/^1pc/[r]^U &{\phantom{UUU} U \otimes \Del_1 \phantom{UU}}  \ar@/^1pc/[r]^U \ar@/^1pc/[l]^{U^*} \ar@(dl,dr)_{\gl(U)} &{\,\phantom{UU} (U^{\otimes 2} \otimes \Del_2)^{S_2} \phantom{UU}} \ar@/^1pc/[r]^U \ar@/^1pc/[l]^{U^*} \ar@(dl,dr)_{\gl(U)} &{\,\phantom{UU} (U^{\otimes 3} \otimes \Del_3)^{S_3} \phantom{UUU}} \ar@/^1pc/[r]^-U \ar@/^1pc/[l]^{U^*} \ar@(dl,dr)_{\gl(U)} &{\,\phantom{U} \ldots \,} \ar@/^1pc/[l]^{U^*} }$$}

\begin{itemize}[leftmargin=*]
 \item $\id_{V} \in \mathfrak{gl}(V)$ acts by scalar $\nu$,
\item $u \in  U \cong \mathfrak{u}_{\mathfrak{p}}^{-}$ acts by operator
\begin{align*}
 (U^{\otimes k} \otimes \Del_k)^{S_k}  &\stackrel{F_u}{\longrightarrow} (U^{\otimes k+1} \otimes \Del_{k+1})^{S_{k+1}}  \\
F_u := \frac{1}{k+1}\sum_{1 \leq l \leq k+1} u^{(l)} \otimes res_l^*
\end{align*}

Here $res_l^*: \Del_k \rightarrow \Del_{k+1}$ as in Definition \ref{def:res_star_morphisms}, $u^{(l)}$ as in Notation \ref{notn:action_on_tes_powers} and $k \geq 0$.

\item $f \in  U^* \cong \mathfrak{u}_{\mathfrak{p}}^{+}$ acts by operator
\begin{align*}
 (U^{\otimes k} \otimes \Del_k)^{S_k}  &\stackrel{E_f}{\longrightarrow} (U^{\otimes k-1} \otimes \Del_{k-1})^{S_{k-1}} \\
E_f := \sum_{1 \leq l \leq k} f^{(l)} \otimes res_l
\end{align*}

Here $res_l: \Del_k \rightarrow \Del_{k-1}$ as in Definition \ref{def:res_morphisms}, $f^{(l)}$ as in Notation \ref{notn:action_on_tes_powers} and $k \geq 1$. The action of $f$ on $(U^{\otimes 0} \otimes \InnaB{\mathbf{1}})^{S_0} \cong  \InnaB{\mathbf{1}}$ is zero.

\item $A \in \mathfrak{gl}(U) \subset \mathfrak{gl}(V)$ acts on $(U^{\otimes k} \otimes \Del_k)^{S_k} $ by

\begin{align*}
\sum_{1 \leq i \leq k} A^{(i)} \rvert_{U^{\otimes k}} \otimes \id_{\Del_k}: (U^{\otimes k} \otimes \Del_k)^{S_k}  &\longrightarrow (U^{\otimes k} \otimes \Del_k)^{S_k}
\end{align*}

\end{itemize}
\end{definition}



\begin{lemma}\label{lem:complex_ten_power_action_well_def}
The action of $\gl(V)$ described in Definition \ref{def:complex_ten_power_splitting} is well-defined.
\end{lemma}

\begin{proof}
 See Appendix \ref{appssec:complex_ten_power_action_well_def}.
\end{proof}

\begin{remark}
 Notice that $\mathfrak{gl}(U)$ acts semisimply on $V^{\underline{\otimes}  \nu}$, $U^* \cong \mathfrak{u}_{\mathfrak{p}}^{+}$ acts locally finitely on $V^{\underline{\otimes}  \nu}$, thus making it an object of $Ind-(\underline{Rep}^{ab}(S_{\nu}) \boxtimes \co^{\mathfrak{p}}_{\nu, V})$ (in fact, an object of the category $Ind-(\underline{Rep}(S_{\nu=n})) \boxtimes \co^{\mathfrak{p}}_{n, V})$, since $\Delta_k \in \underline{Rep}(S_{\nu})$ for any $k \in \bZ_+$).
\end{remark}

\begin{remark}
 The definition of $V^{\underline{\otimes}  \nu}$ makes it a $\bZ_+$-graded object in \InnaA{$Ind-\underline{Rep}(S_{\nu})$}. This grading corresponds to the natural grading on $V^{\otimes n}$ seen as a tensor power of the graded object $V = \InnaB{\bC \triv} \oplus U, gr_0(V):=\InnaB{\bC \triv}, gr_1(V):=U$.
\end{remark}

\begin{remark}
 \InnaB{In Subsection \ref{ssec:comp_tens_power_indep_splitting}, we will show that the object $V^{\underline{\otimes}  \nu}$ of $Ind-(\underline{Rep}^{ab}(S_{\nu}) \boxtimes \co^{\mathfrak{p}}_{\nu, V})$ does not really depend on the splitting of $V$, but rather only on the pair $(V, \triv)$. In the case when $\nu \notin \bZ_+$, one can actually give an equivalent definition of $V^{\underline{\otimes}  \nu}$ without using a splitting (c.f. \cite{E1} and Appendix \ref{app:complex_ten_power_generic_nu}).}
\end{remark}

The following technical lemma will be useful to us later on. The meaning of this lemma is that the operator $F_u$ acting on $V^{\underline{\otimes}  \nu}$ is ``almost injective''.
\begin{lemma}\label{lem:tens_power_weak_torsion_free}
Let \InnaA{$l \leq k$, and consider a non-zero morphism in $\underline{Rep}(S_{\nu})$} $$\phi: U^{\otimes l} \otimes \Del_l \longrightarrow U^{\otimes k} \otimes \Del_k $$ Let $u \in U \cong \mathfrak{u}_{\mathfrak{p}}^{-}, u \neq 0$.
Then $F_u \circ \phi \neq 0$, where $F_u \circ \phi := \frac{1}{k+1}\sum_{1 \leq l \leq k+1} (u^{(l)} \otimes res_l^*) \circ \phi$.
\end{lemma}
\begin{proof}
C.f. Appendix \ref{appssec:tens_power_weak_torsion_free}.
\end{proof}

\subsection{}\label{ssec:comp_tens_power_compatible_classical}

Finally, we prove that the definition of a complex tensor power of a split unital vector space is compatible with the usual notion of a tensor power of a vector space.

\InnaB{
We continue with a fixed splitting $V = \InnaB{\bC \triv} \oplus U$.
Let $V^{\underline{\otimes}  n}$ be as in Definition \ref{def:complex_ten_power_splitting}. 

}
\mbox{}
Define the action of $\gl(V)$ on the space $\bigoplus_{k=0,...,n} (U^{\otimes k} \otimes \bC Inj(\{1,...,k\} , \{1,...,n\}))^{S_k}$ via the decomposition $$\gl(V) \cong \bC \id_V \oplus \mathfrak{u}_{\mathfrak{p}}^{-} \oplus \mathfrak{u}_{\mathfrak{p}}^{+} \oplus \gl(U)$$ by setting
\begin{itemize}[leftmargin=*]
 \item $\id_V$ acts by the scalar $n$,
 \item $u \in  U \cong \mathfrak{u}_{\mathfrak{p}}^{-}$ acts by operator \InnaA{
\begin{align*}
(U^{\otimes k} \otimes \bC Inj(\{1,...,k\} , \{1,...,n\}))^{S_{k}} &\stackrel{F_u}{\longrightarrow} (U^{\otimes k+1} \otimes \bC  Inj(\{1,...,k+1\} , \{1,...,n\}))^{S_{k+1}}\\
F_u = \frac{1}{k+1}\sum_{1 \leq l \leq k+1} u^{(l)} \otimes \mathtt{res}_l^*
\end{align*}
for any $k \geq 0$.
Here $\mathtt{res}_l^*$ is the map defined in Remark \ref{rmrk:res_l_image_under_S_n_functor}, $u^{(l)}$ as in Notation \ref{notn:action_on_tes_powers}.
\item $f \in  U^* \cong \mathfrak{u}_{\mathfrak{p}}^{+}$ acts by operator
\begin{align*}
(U^{\otimes k} \otimes \bC Inj(\{1,...,k\} , \{1,...,n\}))^{S_{k}} &\stackrel{E_f}{\longrightarrow} (U^{\otimes k-1} \otimes \bC Inj(\{1,...,k-1\} , \{1,...,n\}))^{S_{k-1}} \\
E_f = \sum_{1 \leq l \leq k} f^{(l)} \otimes \mathtt{res}_l
\end{align*}}
whenever $k \geq 1$. Here $\mathtt{res}_l$ is the map defined in Remark \ref{rmrk:res_l_image_under_S_n_functor}, $f^{(l)}$ as in Notation \ref{notn:action_on_tes_powers}.

The action of $f$ on $(U^{\otimes 0} \otimes \bC)^{S_0} \cong \bC$ is zero.
 \item $ \gl(U)$ acts naturally on each summand $(U^{\otimes k} \otimes \bC Inj(\{1,...,k\} , \{1,...,n\}))^{S_k}$: 
 
 $A \in \mathfrak{gl}(U)$ acts by
\begin{align*}
&\sum_{1 \leq i \leq k} A^{(i)} \rvert_{U^{\otimes k}} \otimes \id_{\bC Inj(\{1,...,k\} , \{1,...,n\})}: \\
&( U^{\otimes k} \otimes \bC Inj(\{1,...,k\} , \{1,...,n\}))^{S_k} &\longrightarrow (U^{\otimes k} \otimes \bC Inj(\{1,...,k\} , \{1,...,n\}))^{S_k} 
\end{align*}
\end{itemize}

\InnaA{
Notice that the space $\bigoplus_{k=0,...,n} (U^{\otimes k} \otimes \bC Inj(\{1,...,k\} , \{1,...,n\}))^{S_k}$ automatically possesses a structure of a $\bC[S_n] \otimes_{\bC} \mathcal{U}(\gl(V))$-module: the group $S_n$ acts on each summand $\bC Inj(\{1,...,k\} , \{1,...,n\}))$ through its action on the set $\{1,...,n\}$.
}
\begin{lemma}\label{lem:gl_V_action_usual_tens_power}
There is an isomorphism of $\mathfrak{gl}(V)$-modules \InnaA{$$\Phi: V^{\otimes n} \stackrel{\sim}{\longrightarrow} \bigoplus_{k=0,...,n} (U^{\otimes k} \otimes \bC Inj(\{1,...,k\} , \{1,...,n\}))^{S_k}$$
where $\Phi (\triv \otimes \triv \otimes ... \otimes \triv) =1$ \InnaB{(lies in degree zero of the right hand side)}.

Moreover, this isomorphism is an isomorphism of $\bC[S_n] \otimes_{\bC} \mathcal{U}(\gl(V))$-modules.}
\end{lemma}
 
\begin{proof}
C.f. Appendix \ref{appssec:action_on_usual_tensor_power_lemmas}.
\end{proof}

\InnaA{

\begin{proposition}\label{prop:comp_tens_power_F_n}
 Consider the functor $$\hat{\mathcal{S}}_n: Ind-(\underline{Rep}(S_{\nu=n}) \InnaB{\boxtimes} \co^{\mathfrak{p}}_{n, V}) \longrightarrow Ind-(Rep(S_n)\InnaB{\boxtimes} \co^{\mathfrak{p}}_{n, V})$$ (c.f. Subsection \ref{ssec:def_cat_Del_dash_par_O}). Then $\hat{\mathcal{S}}_n (V^{\underline{\otimes}  n}) \cong V^{\otimes n}$.
\end{proposition}}

\begin{remark}
 Restricting the action of $\gl(V)$ to $\gl(U)$, it can be seen from the proof we get an isomorphism of $\bZ_+$-graded \InnaA{$\gl(U)$-modules}.
\end{remark}

\begin{proof}
 Recall from Lemma \ref{lem:funct_F_n_Delta_k} that $\mathcal{S}_n (\Delta_k) = \bC Inj(\{1,...,k\} , \{1,...,n\})$ (this is zero if $k >n$). By definition of the action of $\gl(V)$ on $\bigoplus_{k=0,...,n} (U^{\otimes k} \otimes \bC Inj(\{1,...,k\} , \{1,...,n\}))^{S_k}$ given above, we have an isomorphism of \InnaA{$\bC[S_n] \otimes_{\bC} \mathcal{U}(\gl(V))$}-modules
 $$\Psi: \hat{\mathcal{S}}_n (V^{\underline{\otimes}  n}) \rvert_{\gl(U)} \rightarrow \bigoplus_{k=0,...,n} (U^{\otimes k} \otimes \bC Inj(\{1,...,k\} , \{1,...,n\}))^{S_k}$$
 
 Using Lemma \ref{lem:gl_V_action_usual_tens_power} above, we obtain the desired result.

\end{proof}
%
%
%
%

\begin{example}
 Let $\dim U = 1$. In that case $V$ can be viewed as the tautological representation of $\gl_2$, with standard basis $v_0, v_1$. The tensor power $V^{\otimes n}$ is then a span of weight vectors of the form
$$v_{i_1} \otimes v_{i_2} \otimes ...\otimes v_{i_n}, \; i_1,...,i_n \in \{ 0,1\}$$
(the weight of this vector is $(n- \sum_{j=1,..,n} i_j, \sum_{j=1,..,n} i_j)$). This allows us to establish an isomorphism
\begin{align*}
 V^{\otimes n} &\rightarrow \bigoplus_{k=0,...,n} \bC \{ S \subset \{1,...,n\} \rvert \abs{S}=k\} \cong \bigoplus_{k=0,...,n} \bC Inj(\{1,...,k\} , \{1,...,n\})^{S_k}\\
v_{i_1} \otimes v_{i_2} \otimes ...\otimes v_{i_n} &\mapsto  S:=\{j\in \{1,...,n\} \rvert i_j =1\}
\end{align*}

The operators $E \in  U^* \cong \mathfrak{u}_{\mathfrak{p}}^{+}, F \in  U \cong \mathfrak{u}_{\mathfrak{p}}^{-}$ act on $\bigoplus_{k=0,...,n} \bC \{ S \subset \{1,...,n\} \rvert \abs{S}=k\}$ by
$$  E(S) =\sum_{T: T \subset S, \abs{T}=k-1} T, \, F(S) =\frac{1}{k+1} \sum_{T: S \subset T \subset \{1,...,n\}, \abs{T}=k+1} T$$
where $S$ is a subset of $\{1,...,n\}$ of size $k$.
\InnaA{
The operator $\id_{U} \in \End(V)$ acts on $\bigoplus_{k=0,...,n} \bC \{ S \subset \{1,...,n\} \rvert \abs{S}=k\}$ by $S \mapsto \abs{S} \cdot S$.

In particular, we immediately see that
 $$\hat{\mathcal{S}}_n(V^{\underline{\otimes}  n}) \cong \bigoplus_{k \geq 0} \mathcal{S}_n((\Delta_k)^{S_k}) \cong  \bigoplus_{k \geq 0} \bC Inj(\{1,...,k\} , \{1,...,n\})^{S_k} \cong  V^{\otimes n}$$}
\end{example}

\subsection{\InnaB{Independence of $V^{\underline{\otimes} \nu}$ on the choice of splitting}}\label{ssec:comp_tens_power_indep_splitting}
\InnaB{In this subsection, we show that our definition of $V^{\underline{\otimes} \nu}$ does not depend on the choice of splitting $V = \bC \triv \oplus U$ we made earlier, but rather only on the pair $(V, \triv)$.

Consider the following category $\mathbf{Uni}$ of unital vector spaces. The objects of this category will be tuples $(V, \triv, U)$, where $V$ is a finite-dimensional vector space, $\triv \in V \setminus \{0\}$, and $U$ is a subspace of $V$ such that $V = \bC \triv \oplus U$. 

The morphisms in this category are given by $$Mor_{\mathbf{Uni}} \left( (V, \triv, U), (V', \triv', U') \right) := \{ \phi \in \Hom_{\bC}(V, V') : \phi(\triv) = \triv' \}$$

\begin{remark}
 This category has a natural structure of a symmetric monoidal tensor category, with $$(V, \triv, U) \otimes (V', \triv', U') :=  \left( V \otimes V', \triv \otimes \triv', U \otimes \bC \triv' \oplus \bC \triv \otimes U' \oplus U \otimes U' \right)$$
\end{remark}

We now construct a functor 
$$(\cdot)^{\underline{\otimes} \nu}: \mathbf{Uni} \longrightarrow Ind-\underline{Rep}(S_{\nu}) $$ 
On objects of $\mathbf{Uni}$, this is just $(V, \triv, U) \mapsto V^{\underline{\otimes} \nu} := \bigoplus_{k \geq 0} (U^{\otimes k} \otimes \Del_k)^{S_k}$. On morphisms, this functor is
$$\phi: (V, \triv, U) \rightarrow (V', \triv', U') \; \; \mapsto \; \; \Phi: \bigoplus_{k \geq 0} (U^{\otimes k} \otimes \Del_k)^{S_k} \rightarrow \bigoplus_{k \geq 0} (U'^{\otimes k} \otimes \Del_k)^{S_k}$$
with the matrix coefficients $\Phi^{ l, k}: (U^{\otimes k} \otimes \Del_k)^{S_k} \rightarrow (U'^{\otimes l} \otimes \Del_l)^{S_l}$ of $\Phi$ coming from maps $U^{\otimes k} \otimes \Del_k \rightarrow U'^{\otimes l} \otimes \Del_l$ which are defined by the formula
\begin{align*}
&\sum_{\substack{\iota: \{1,..., l \} \rightarrow \{1,..., k \} \\ \text{ strictly increasing} }} \phi_{U, U'}^{(Im(\iota))} \otimes \phi_{U, \triv'}^{(\{1,..., k \} \setminus Im(\iota))} \otimes \left( \Del_k \stackrel{res_{\iota}}{\rightarrow} \Del_l \right)
\end{align*}

Here 
\begin{itemize}[leftmargin=*]
      \item The map $\phi_{U, U'}: U \rightarrow U'$ is the composition $U \hookrightarrow V \stackrel{\phi}{\rightarrow} V' \twoheadrightarrow U'$.
      
      The notation $\phi_{U, U'}^{(Im(\iota))}$ means that we apply the map $\phi_{U, U'}$ only to the factors $\iota(1), \iota(2), ..., \iota(l)$ of $U^{\otimes k}$.
      \item The map $\phi_{U, \triv'}: U \rightarrow \bC$ is defined so that the composition $$U \hookrightarrow V \stackrel{\phi}{\rightarrow} V' \twoheadrightarrow \bC \triv'$$ is the map $u \mapsto \phi_{U, \triv'}(u) \triv'$. The notation $\phi_{U, \triv'}^{(\{1,..., k \} \setminus Im(\iota))}$ means that we apply the map $\phi_{U, \triv'}$ only to those factors $i$ of $U^{\otimes k}$ for which $i \not\in Im(\iota)$.
      \item The map $res_{\iota}$ is the map $\Del_k \rightarrow \Del_l$ given by the diagram $\pi  \in \bar{P}_{k,l}$ with edges $\{(\iota(i), i')\}_{1 \leq i \leq  l}$.
     \end{itemize}

Note that $\Phi$ is upper-triangular in terms of the matrix coefficients $\Phi^{ l, k}$.
     
\begin{lemma}
 The functor $(\cdot)^{\underline{\otimes} \nu}: \mathbf{Uni} \longrightarrow Ind-\underline{Rep}(S_{\nu}) $ is well-defined.
\end{lemma}
\begin{proof}
 We only need to check that this functor preserves composition. Indeed, let $\phi: (V, \triv, U) \rightarrow (V', \triv', U')$, $\psi: (V', \triv', U') \rightarrow (V'', \triv'', U'')$, and denote by $\Phi, \Psi$ their respective images under the functor $(\cdot)^{\underline{\otimes} \nu}$. 
 
 We have:
 $(\psi \circ \phi)_{U, U''} = \psi_{U', U''} \circ \phi_{U, U'}$, and $(\psi \circ \phi)_{U, \triv''} = \psi_{U', \triv''} \circ \phi_{U, U'} + \phi_{U, \triv'}$.
 
 Thus $$(\psi \circ \phi)^{\underline{\otimes} \nu}: \bigoplus_{k \geq 0} (U^{\otimes k} \otimes \Del_k)^{S_k} \rightarrow \bigoplus_{k \geq 0} (U''^{\otimes k} \otimes \Del_k)^{S_k}$$ is a map in $Ind-\underline{Rep}(S_{\nu})$, with matrix coefficients coming from the maps $U^{\otimes k} \otimes \Del_k \rightarrow U''^{\otimes l} \otimes \Del_l$ given by
\begin{align*}
&\sum_{\substack{\iota: \{1,..., l \} \rightarrow \{1,..., k \} \\ \text{ strictly increasing} }} (\psi_{U', U''} \circ \phi_{U, U'})^{(Im(\iota))} \otimes \left( \psi_{U', \triv''} \circ \phi_{U, U'} + \phi_{U, \triv'} \right) ^{(\{1,..., k \} \setminus Im(\iota))} \otimes \left( \Del_k \stackrel{res_{\iota}}{\rightarrow} \Del_l \right)
\end{align*}
Next, 
\begin{align*}
&\sum_{\substack{\iota: \{1,..., l \} \rightarrow \{1,..., k \} \\ \text{ strictly increasing} }} (\psi_{U', U''} \circ \phi_{U, U'})^{(Im(\iota))} \otimes \left( \psi_{U', \triv''} \circ \phi_{U, U'} + \phi_{U, \triv'} \right) ^{(\{1,..., k \} \setminus Im(\iota))} \otimes \left( \Del_k \stackrel{res_{\iota}}{\rightarrow} \Del_l \right) =\\
&=\sum_{l \leq j \leq k} \sum_{\substack{\iota_2: \{1,..., l \} \rightarrow \{1,..., j \} \\ \text{ str. incr.} }} \sum_{\substack{\iota_1: \{1,..., j \} \rightarrow \{1,..., k \} \\ \text{ str. incr.} }} (\psi_{U', U''} \circ \phi_{U, U'})^{(Im(\iota_1 \circ \iota_2))} \otimes \left( \psi_{U', \triv''} \circ \phi_{U, U'} \right)^{(Im(\iota_1) \setminus Im(\iota_1 \circ \iota_2))} \otimes \\
&\otimes \phi_{U, \triv'}^{(\{1,..., k \} \setminus Im(\iota_1))} \otimes \left( \Del_k \stackrel{res_{\iota_1}}{\rightarrow} \Del_j \right) \circ \left( \Del_j \stackrel{res_{\iota_2}}{\rightarrow} \Del_l \right) 
\end{align*}
Thus the $(l, k)$ matrix coefficient of $(\psi \circ \phi)^{\underline{\otimes} \nu}$ is $\sum_{l \leq j \leq k} \Psi^{ l, j} \circ \Phi^{ j, k}$, as wanted.
\end{proof}

 Let $(V, \triv, U) \in \mathbf{Uni}$. Then $\Aut_{\mathbf{Uni}} ((V, \triv, U)) = \bar{\mathfrak{P}}_{V, \triv}$ (the mirabolic subgroup of $GL(V)$ preserving $\triv$; c.f. Definition \ref{notn:par_subalg}). Given two splittings $V = \bC \triv \oplus U$, $V = \bC \triv \oplus W$, we have a map $\id_V: (V, \triv, U) \rightarrow (V, \triv, W)$, and we get a commutative diagram
 $$ \begin{CD}
     \bar{\mathfrak{P}}_{V, \triv} = \Aut_{\mathbf{Uni}} ((V, \triv, U)) @>{(\cdot)^{\underline{\otimes} \nu}}>> \Aut_{Ind-\underline{Rep}(S_{\nu})}(\bigoplus_{k \geq 0} (U^{\otimes k} \otimes \Del_k)^{S_k})\\
     @V{Ad_{\id_V}}VV @V{Ad_{(\id_V)^{\underline{\otimes} \nu}}}VV\\
     \bar{\mathfrak{P}}_{V, \triv} = \Aut_{\mathbf{Uni}} ((V, \triv, W)) @>{(\cdot)^{\underline{\otimes} \nu}}>> \Aut_{Ind-\underline{Rep}(S_{\nu})}(\bigoplus_{k \geq 0} (W^{\otimes k} \otimes \Del_k)^{S_k})
    \end{CD}
$$

Consider the action $$\rho_U: \gl(V) \longrightarrow \End_{Ind-\underline{Rep}(S_{\nu})}(\bigoplus_{k \geq 0} (U^{\otimes k} \otimes \Del_k)^{S_k})$$ given in Section \ref{ssec:def_comp_ten_power_split}. 
\begin{lemma}
 The action $\rho_U \rvert_{\gl(U) \oplus U^*}$ integrates to the action of $\bar{\mathfrak{P}}_{V, \triv}$ on $\bigoplus_{k \geq 0} (U^{\otimes k} \otimes \Del_k)^{S_k})$ given above, i.e. we have a commutative diagram 
 $$ \begin{CD}
 \bar{\mathfrak{p}}_{V, \InnaB{\bC \triv}} \cong \gl(U) \oplus U^* @>{\rho_U}>> \End_{Ind-\underline{Rep}(S_{\nu})}(\bigoplus_{k \geq 0} (U^{\otimes k} \otimes \Del_k)^{S_k})\\
 @V{exp}VV @V{exp}VV \\
 \bar{\mathfrak{P}}_{V, \triv} = \Aut_{\mathbf{Uni}} ((V, \triv, U)) @>{(\cdot)^{\underline{\otimes} \nu}}>> \Aut_{Ind-\underline{Rep}(S_{\nu})}(\bigoplus_{k \geq 0} (U^{\otimes k} \otimes \Del_k)^{S_k}) \end{CD}
$$
\end{lemma}
\begin{proof}
 Let $\phi \in \bar{\mathfrak{p}}_{V, \InnaB{\bC \triv}}$, and denote by $\widetilde{\Phi}_s$ the image of $exp(s\phi)$ ($s \in \bC$) under the functor $(\cdot)^{\underline{\otimes} \nu}$. We want to show that $\frac{d}{ds} \widetilde{\Phi}_s = \rho_U (\phi)$. Writing these expressions in terms of matrix coefficients, we want to show that for any $l, k \geq 0$, we have:
 $$\widetilde{\Phi}_s^{l, k} \stackrel{?}{=} \begin{cases}
                         s \rho_U (\phi) +o(s) &\text{ if } l \neq k \\
                         \id + s \rho_U (\phi) +o(s) &\text{ if } l = k 
                        \end{cases}
$$

Indeed, 
\begin{align*}
 &\sum_{\substack{\iota: \{1,..., l \} \rightarrow \{1,..., k \} \\ \text{ strictly increasing} }} exp(s\phi)_{U, U}^{Im(\iota)} \otimes exp(s\phi)_{U, \triv}^{(\{1,..., k \} \setminus Im(\iota))} \otimes \left( \Del_k \stackrel{res_{\iota}}{\rightarrow} \Del_l \right) = \\
 &= \sum_{\substack{\iota: \{1,..., l \} \rightarrow \{1,..., k \} \\ \text{ strictly increasing} }} (\id_U + s\phi_{U, U} +o(s))^{Im(\iota)} \otimes (s\phi_{U, \triv} +o(s))^{(\{1,..., k \} \setminus Im(\iota))} \otimes \left( \Del_k \stackrel{res_{\iota}}{\rightarrow} \Del_l \right) =\\
 &= \begin{cases}
     s \sum_{i \in \{1,..., k \}} \phi_{U, \triv}^{(i)} \otimes res_i + o(s) &\text{ if } l = k-1 \\
     \id_{U^{\otimes k} \otimes \Del_k} + s \sum_{i \in \{1,..., k \}} \phi_{U, U}^{(i)} \otimes \id_{\Delta_k} +o(s) &\text{ if } l = k \\
      o(s) &\text{ else }           
    \end{cases}
\end{align*}
We conclude that $$\widetilde{\Phi}_s^{l, k} = \begin{cases}
                         s \rho_U (\phi) +o(s) &\text{ if } l \neq k \\
                         \id + s \rho_U (\phi) +o(s) &\text{ if } l = k 
                        \end{cases}
$$ as wanted.
\end{proof}

We obtained an action of the subalgebra $\bar{\mathfrak{p}}_{V, \InnaB{\bC \triv}}$ on the $Ind$-object $V^{\underline{\otimes} \nu}$ of $\underline{Rep}(S_{\nu})$, and this action does not depend on the choice of splitting, in the sense that the diagram 
$$ \xymatrix{
     &{} &{\gl(U) \oplus U^*} \ar[r] &{\End_{Ind-\underline{Rep}(S_{\nu})}(\bigoplus_{k \geq 0} (U^{\otimes k} \otimes \Del_k)^{S_k})}
     \ar[dd]_{Ad_{(\id_V)^{\underline{\otimes} \nu}}}\\
     &\bar{\mathfrak{p}}_{V, \InnaB{\bC \triv}}  \ar[ur] \ar[rd] &{} &{}\\
     &{} &{\gl(W) \oplus W^*} \ar[r] &{\End_{Ind-\underline{Rep}(S_{\nu})}(\bigoplus_{k \geq 0} (W^{\otimes k} \otimes \Del_k)^{S_k}) } }$$ 
is commutative for any two splittings $V = \bC \triv \oplus U$, $V = \bC \triv \oplus W$.

It remains to show the action of $\gl(V)$ on $V^{\underline{\otimes} \nu}$ does not depend on the choice of splitting:
\begin{lemma}
The diagram 
 $$ \xymatrix{
     &{} &{\bar{\mathfrak{p}}_{V, \InnaB{\bC \triv}} \oplus U \oplus \bC \id_V} \ar[r] &{\End_{Ind-\underline{Rep}(S_{\nu})}(\bigoplus_{k \geq 0} (U^{\otimes k} \otimes \Del_k)^{S_k})}
     \ar[dd]_{Ad_{(\id_V)^{\underline{\otimes} \nu}}}\\
     &\gl(V)  \ar[ur] \ar[rd] &{} &{}\\
     &{} &{\bar{\mathfrak{p}}_{V, \InnaB{\bC \triv}} \oplus W \oplus \bC \id_V} \ar[r] &{\End_{Ind-\underline{Rep}(S_{\nu})}(\bigoplus_{k \geq 0} (W^{\otimes k} \otimes \Del_k)^{S_k})} }$$
is commutative.
\end{lemma}
\begin{proof}
This follows directly from the definition of action of $W$ (respectively, $U$) on $\bigoplus_{k \geq 0} (W^{\otimes k} \otimes \Del_k)^{S_k}$ (respectively, $\bigoplus_{k \geq 0} (U^{\otimes k} \otimes \Del_k)^{S_k}$).
\end{proof}

Thus we conclude that 
\begin{corollary}
 The definition of the complex tensor power $V \mapsto V^{\underline{\otimes} \nu}$ as an object in $Ind-(\underline{Rep}(S_{\nu=n})) \boxtimes \co^{\mathfrak{p}}_{n, V})$ depends only on the distinguished non-zero vector $\triv$, rather than on the splitting $V = \bC \triv \oplus U$.
\end{corollary}

\begin{remark}
 The functor $(\cdot)^{\underline{\otimes} \nu}$ is a symmetric monoidal functor.
%
 
 Indeed, let $(V, \triv,U), (V', \triv', U') \in \mathbf{Uni}$. The canonical isomorphism of $S_n$ representations $$\Upsilon_n: V^{\otimes n} \otimes V'^{\otimes n} \longrightarrow (V \otimes V')^{\otimes n}$$ and its inverse $\Upsilon^{-1}_n$ can be rewritten using the isomorphism in Lemma \ref{lem:gl_V_action_usual_tens_power}; these interpolate easily to morphisms in $Ind-\underline{Rep}(S_{\nu})$:
 $$ \underline{\Upsilon}_{\nu}: V^{\underline{\otimes} \nu} \otimes V'^{\underline{\otimes} \nu} \longrightarrow (V \otimes V')^{\underline{\otimes} \nu}$$
 and $$ \underline{\Upsilon}'_{\nu}: (V \otimes V')^{\underline{\otimes} \nu}\longrightarrow V^{\underline{\otimes} \nu} \otimes V'^{\underline{\otimes} \nu} $$ so that $\underline{\Upsilon}_{\nu} \circ \underline{\Upsilon}'_{\nu} = \id$, $\underline{\Upsilon}'_{\nu} \circ \underline{\Upsilon}_{\nu} = \id$.
\end{remark}








}

\begin{remark}
 \InnaB{We can now consider the category $\mathbf{Uni}'$ of finite-dimensional unital vector spaces: that is, the objects in $\mathbf{Uni}'$ are pairs $(V, \triv)$, with $\dim V < \infty$, $\triv \in V \setminus \{0\}$, and the morphisms are $$\Hom_{\mathbf{Uni}'}((V, \triv) , (V', \triv')) := \{\phi \in \Hom_{\bC}(V, V'): \phi(\triv)=\triv' \}$$
 
 By definition, we have a forgetful functor $Forget: \mathbf{Uni} \rightarrow \mathbf{Uni}'$, and this functor is an equivalence of categories. 
 
 This allows us to define a functor $(\cdot)^{\underline{\otimes} \nu}: \mathbf{Uni}' \rightarrow Ind-\underline{Rep}(S_{\nu})$ for each choice of functor $Forget^{-1}: \mathbf{Uni}' \rightarrow \mathbf{Uni}$; the latter is defined up to isomorphism. We do not currently have a definition of the functor $(\cdot)^{\underline{\otimes} \nu}: \mathbf{Uni}' \rightarrow Ind-\underline{Rep}(S_{\nu})$ which does not involve a choice of $Forget^{-1}$.}
\end{remark}

\section{Schur-Weyl duality in Deligne setting: $\underline{Rep}^{ab}(S_{\nu})$ and $\co^{\mathfrak{p}}_{\nu, V}$}\label{sec:SW_functor}

We fix $\nu \in \bC$, and a finite-dimensional unital vector space $(V, \triv)$ (so $V^{\underline{\otimes}  \nu}$ is defined).
\begin{definition}\label{def:SW_functor}
 Define the Schur-Weyl functor $$SW^{ind}_{\nu}: Ind-\underline{Rep}^{ab}(S_{\nu}) \longrightarrow Mod_{\mathcal{U}(\gl(V))}$$ by $$SW^{ind}_{\nu}:= \Hom_{Ind-\underline{Rep}^{ab}(S_{\nu})}( \cdot , V^{\underline{\otimes}  \nu})$$

 We will also consider the restriction of the functor $SW^{ind}_{\nu}$ to the category $\underline{Rep}^{ab}(S_{\nu})$, which will be denoted by $SW_{\nu}$.
\end{definition}

\begin{remark}
 The functor $SW^{ind}_{\nu}: Ind-\underline{Rep}^{ab}(S_{\nu}) \longrightarrow Mod_{\mathcal{U}(\gl(V))}$ (as well as its restriction $SW_{\nu}$) is a contravariant $\bC$-linear additive left-exact functor.
\end{remark}

It turns out that the image of the functor $SW_{\nu} :\underline{Rep}^{ab}(S_{\nu}) \rightarrow Mod_{\mathcal{U}(\gl(V))}$ lies in $\co^{\mathfrak{p}}_{\nu, V}$, as we will prove in Lemma \ref{lem:image_SW_fin_gen}.

The following technical lemma will be used to perform most of the computations:
\begin{lemma}\label{lem:hom_X_tau_Delta_k}
Let $\T$ be a Young diagram, $k \in \bZ_+$ and $\mu$ a partition of $k$.
\begin{itemize}
 \item Assume $\T$ lies in a trivial $\stackrel{\nu}{\sim}$-class. Then
$$ \dim \Hom_{\underline{Rep}(S_{\nu}) \boxtimes Rep(S_k)}( X_{\T} \otimes \mu, \Delta_k) = \begin{cases} 1 &\text{ if } \mu \in \mathcal{I}^{+}_{\T} \\
 0 &\text{ else }  \\
 \end{cases}$$

\item Assume $\T$ lies in a non-trivial $\stackrel{\nu}{\sim}$-class $ \{\lambda^{(i)}\}_i$, and let $j$ be such that $\T = \lambda^{(j)}$. Then we have:

$$
 \dim \Hom_{\underline{Rep}(S_{\nu}) \boxtimes Rep(S_k)}( X_{\lambda^{(0)}} \otimes \mu, \Delta_k) = \begin{cases} 1 &\text{ if } \mu \in \mathcal{I}^{+}_{\lambda^{(0)}} \\
 0 &\text{ else }  \\
 \end{cases}
$$
and if $j>0$, we have:
$$
 \dim \Hom_{\underline{Rep}(S_{\nu}) \boxtimes Rep(S_k)}( X_{\lambda^{(j)}} \otimes \mu, \Delta_k) = \begin{cases} 2 &\text{ if } \mu \in \mathcal{I}^{+}_{\lambda^{(j)}} \cap \mathcal{I}^{+}_{\lambda^{(j-1)}} \\
1 &\text{ if } \mu \in \left(\mathcal{I}^{+}_{\lambda^{(j)}} \setminus \mathcal{I}^{+}_{\lambda^{(j-1)}}\right) \cup \left(\mathcal{I}^{+}_{\lambda^{(j-1)}} \setminus \mathcal{I}^{+}_{\lambda^{(j)}}\right)\\
 0 &\text{ else }  \\
 \end{cases}
$$
\end{itemize}

\end{lemma}
\begin{proof}
The proof will use the $lift_{\nu}$ maps discussed in Subsection \ref{ssec:nu_class_and_lift}.

We know that
\begin{align*}
&\dim \Hom_{\underline{Rep}(S_{\nu}) \boxtimes Rep(S_k)}( X_{\T} \otimes \mu, \Delta_k) = \dim \Hom_{\underline{Rep}(S_{T}) \boxtimes Rep(S_k)}( lift_{\nu}(X_{\T}) \otimes \mu, lift_{\nu}(\Delta_k)) = \\
&=\dim \Hom_{\underline{Rep}(S_{T}) \boxtimes Rep(S_k)}( lift_{\nu}(X_{\T}) \otimes \mu, \Delta_k)
\end{align*}
So it is enough to prove that $$\dim \Hom_{\underline{Rep}(S_{T}) \boxtimes Rep(S_k)}( X_{\T} \otimes \mu, \Delta_k) =\begin{cases} 1 &\text{ if } \mu \in \mathcal{I}^{+}_{\T} \\
 0 &\text{ else }  \\
 \end{cases} $$ and the statement of the lemma will then follow from Lemma \ref{lem:nu_classes_lift}.

To compute $\dim \Hom_{\underline{Rep}(S_{T}) \boxtimes Rep(S_k)}( X_{\T} \otimes \mu, \Delta_k)$, note that
\begin{align*}
&\dim \Hom_{\underline{Rep}(S_{T}) \boxtimes Rep(S_k)}( X_{\T} \otimes \mu, \Delta_k) = \dim \Hom_{\underline{Rep}(S_{\nu=n}) \boxtimes Rep(S_k)}( X_{\T} \otimes \mu, \Delta_k) = \\
&= \dim \Hom_{S_n \times S_k}( \mathcal{S}_n(X_{\T}) \otimes \mu, \mathcal{S}_n(\Delta_k))= \\
&= \dim \Hom_{S_n \times S_k}(\tilde{\T}(n) \otimes \mu, \bC Inj(\{1,...,k\} , \{1,...,n\}))
\end{align*}
for $n >>0, n\in \bZ$ (the first equality follows from Proposition \ref{prop:lift_properties}, while the second relies on the fact that $\mathcal{S}_n$ is fully faithful on $\underline{Rep}(S_{\nu=n})^{(n/2)}$).

But
$$ \bC Inj(\{1,...,k\} , \{1,...,n\}) \cong \bigoplus_{\rho: \abs{\rho} = k} \bigoplus_{\lambda \in \mathcal{I}^{+}_{\rho}, \abs{\lambda}  =n} \lambda \otimes \rho$$
so
$$\dim \Hom_{S_n \times S_k}(\tilde{\T}(n) \otimes \mu, \bC Inj(\{1,...,k\} , \{1,...,n\}))= \begin{cases} 1 &\text{ if } \tilde{\T}(n) \in \mathcal{I}^{+}_{\mu} \\
 0 &\text{ else }  \\
 \end{cases} $$
It remains to check that for $n>>0$, $$\tilde{\T}(n) \in \mathcal{I}^{+}_{\mu} \Leftrightarrow \mu \in \mathcal{I}^{+}_{\T}$$
The first condition is equivalent to saying that
 $$ ... \leq \mu_{i+1} \leq \tilde{\T}(n)_{i+1} \leq \mu_i \leq ...\leq \tilde{\T}(n)_{2} \leq \mu_{1} \leq \tilde{\T}(n)_{1} = n-\abs{\T}$$
which is equivalent to $$ ... \leq \mu_{i+1} \leq \T_{i} \leq \mu_i \leq ...\leq  \T_{1} \leq \mu_{1}$$ (a reformulation of the condition $\mu \in \mathcal{I}^{+}_{\T}$) provided that $n>>0$.
\end{proof}

 \begin{lemma}\label{lem:image_SW_fin_gen}
The image of the functor $SW_{\nu}: \underline{Rep}^{ab}(S_{\nu}) \rightarrow Mod_{\mathcal{U}(\gl(V))}$ lies in $\co^{\mathfrak{p}}_{\nu, V}$. 
 \end{lemma}
\begin{proof}
\InnaB{Fix a splitting $V = \InnaB{\bC \triv} \oplus U$. We want to prove that} for any $M \in \Dab$, $SW_{\nu}(M)$ is a Noetherian $\mathcal{U}(\gl(V))$-module \InnaB{of degree $\nu$} on which $\gl(U)$ acts \InnaB{polynomially (i.e. $M\rvert_{\gl(U)} \in Ind-Mod_{\mathcal{U}(\gl(U)), poly}$)} and $\mathfrak{u}_{\mathfrak{p}}^{+}$ acts locally finitely.

%
Recall that for any $M \in \Dab$, $SW_{\nu}(M) = \Hom_{Ind-\Dab}(M, V^{\underline{\otimes}  \nu})$, with the action of $\gl(V)$ coming from its action on $V^{\underline{\otimes}  \nu}$ (c.f. Definition \ref{def:complex_ten_power_splitting}).
So $$SW_{\nu}(M) \rvert_{\gl(U)} = \Hom_{Ind-\Dab}(M, V^{\underline{\otimes}  \nu} \rvert_{\gl(U)}) \cong \bigoplus_{k \geq 0} \Hom_{\Dab}(M, (U^{\otimes k} \otimes \Delta_k)^{S_k}) $$ with $\gl(U)$ acting through its action on each $(U^{\otimes k} \otimes \Delta_k)^{S_k}$. \InnaB{This immediately implies that $SW_{\nu}(M)$ has degree $\nu$.}

Next, $(U^{\otimes k} \otimes \Delta_k)^{S_k}$ is an object of \InnaA{$\Dab \boxtimes Mod_{\mathcal{U}(\gl(U)), poly}$}, so the spaces $\Hom_{\Dab}(M, (U^{\otimes k} \otimes \Delta_k)^{S_k})$ are polynomial $\gl(U)$-modules. Thus $M\rvert_{\gl(U)} \in Ind-Mod_{\mathcal{U}(\gl(U)), poly}$.

The above $\gl(U)$-decomposition of $SW_{\nu}(M)$ gives us a $\bZ_+$-grading on $SW_{\nu}(M)$, with each grade being finite-dimensional. Definition \ref{def:complex_ten_power_splitting} tells us that $\mathfrak{u}_{\mathfrak{p}}^{+}$ acts on this space by operators of degree $-1$, so $\mathfrak{u}_{\mathfrak{p}}^{+}$ acts locally finitely on $SW_{\nu}(M)$.

We now prove that for any $M \in \Dab$, $SW_{\nu}(M)$ is a Noetherian $\mathcal{U}(\gl(V))$-module.
Recall that the functor $SW_{\nu}$ is a contravariant left-exact functor. Using this, together with the fact that the category $\Dab$ has enough projectives, it is enough to prove that for any indecomposable projective $\mathbf{P} \in \Dab$, we have: $SW_{\nu}(\mathbf{P})$ is a Noetherian $\mathcal{U}(\gl(V))$-module.

Next, recall that any indecomposable projective in $\Dab$ is isomorphic to $X_{\lambda}$, where $\lambda$ either lies in a trivial $ \stackrel{\nu}{\sim}$-class, or is not minimal in its non-trivial $\stackrel{\nu}{\sim}$-class (c.f. Proposition \ref{prop:obj_ab_env}).

Lemma \ref{lem:hom_X_tau_Delta_k} tells us that if $\lambda$ lies in a trivial $ \stackrel{\nu}{\sim}$-class, then
\begin{align*}
 &SW_{\nu}(X_{\lambda}) \rvert_{\gl(U)} \cong \Hom_{Ind-\underline{Rep}^{ab}(S_{\nu})}(X_{\lambda}, \bigoplus_{k \geq 0} (U^{\otimes k} \otimes \Delta_k)^{S_k}) \cong\\
 &\cong \bigoplus_{k \geq 0} \bigoplus_{\mu: \abs{\mu}=k} \Hom_{\underline{Rep}(S_{\nu}) \boxtimes Rep(S_k)}(X_{\lambda} \otimes \mu, \Delta_k) \otimes S^{\mu} U \cong \bigoplus_{\mu \in \mathcal{I}^{+}_{\lambda}} S^{\mu} U                                                                                                                                                                                                                                                                                                                                                                                                                                   \end{align*}

If $\lambda$ lies in a non-trivial $ \stackrel{\nu}{\sim}$-class, $\lambda = \lambda^{(i)}, i>0$, then
\begin{align*}
&SW_{\nu}(X_{\lambda^{(i)}}) \rvert_{\gl(U)} \cong \Hom_{Ind-\underline{Rep}^{ab}(S_{\nu})}(X_{\lambda^{(i)}}, \bigoplus_{k \geq 0} (U^{\otimes k} \otimes \Delta_k)^{S_k}) \cong \\
&\cong \bigoplus_{k \geq 0} \bigoplus_{\mu: \abs{\mu}=k} \Hom_{\underline{Rep}(S_{\nu}) \boxtimes Rep(S_k)}(X_{\lambda^{(i)}} \otimes \mu,  \Delta_k) \otimes S^{\mu} U \cong
\bigoplus_{\mu \in \mathcal{I}^{+}_{\lambda^{(i)}}} S^{\mu} U \oplus \bigoplus_{\mu \in \mathcal{I}^{+}_{\lambda^{(i-1)}}} S^{\mu} U
\end{align*}
\InnaA{
In both cases, we can consider $SW_{\nu}(X_{\lambda}) \rvert_{\gl(U)}$ as a $\bZ_+$-graded space, with grade $j$ being the direct sum of those $S^{\mu} U$ for which $\abs{\mu} = j$. Then the non-zero elements of $\mathfrak{u}_{\mathfrak{p}}^+$ act by operators of degree $-1$ (see Definition \ref{def:complex_ten_power_splitting}).
}
Next, recall that in any case, the parabolic Verma module restricted to $\gl(U)$ decomposes as
$$M_{\mathfrak{p}}(\nu-\abs{\lambda}, \lambda)\lvert_{\gl(U)} \cong  \bigoplus_{\mu \in \mathcal{I}^{+}_{\lambda}} S^{\mu} U $$ and has $\mathcal{U}(\gl(V))$-length at most $2$. \InnaA{Using the above property of the action of $\mathfrak{u}_{\mathfrak{p}}^+$, we see that $SW_{\nu}(X_{\lambda})$ has a finite filtration where each quotient is the image of a parabolic Verma module (therefore all the quotients have finite length).
}
\end{proof}

We can now define another Schur-Weyl functor which we will consider: it is the contravariant functor $\widehat{SW}_{\nu}: \underline{Rep}^{ab}(S_{\nu}) \longrightarrow \widehat{\co}^{\mathfrak{p}}_{\nu, V}$ where $$\widehat{\co}^{\mathfrak{p}}_{\nu, V} := \quotient{\co^{\mathfrak{p}}_{\nu, V}}{\InnaA{Mod_{\mathcal{U}(\gl(V)), poly, \nu}}}$$ is the Serre quotient of $\co^{\mathfrak{p}}_{\nu, V}$ by the Serre subcategory \InnaA{$\InnaA{Mod_{\mathcal{U}(\gl(V)), poly, \nu}}$ of polynomial $\gl(V)$-modules of \InnaA{degree} $\nu$}. We will denote the quotient functor by $$\hat{\pi}: \co^{\mathfrak{p}}_{\nu, V} \longrightarrow \widehat{\co}^{\mathfrak{p}}_{\nu, V}$$ and define $$\widehat{SW}_{\nu} := \hat{\pi} \circ SW_{\nu}$$

The main goal of this section is to prove the following theorem:
\begin{theorem}\label{thrm:SW_almost_equiv}
The contravariant functor $\widehat{SW}_{\nu}:\underline{Rep}^{ab}(S_{\nu}) \rightarrow \widehat{\co}^{\mathfrak{p}}_{\nu, V}$ is exact and essentially surjective.

Moreover, the induced contravariant functor $$ \quotient{\underline{Rep}^{ab}(S_{\nu})}{Ker(\widehat{SW}_{\nu})} \rightarrow \widehat{\co}^{\mathfrak{p}}_{\nu, V}$$ is an anti-equivalence of abelian categories, thus making $\widehat{\co}^{\mathfrak{p}}_{\nu, V}$ a Serre quotient of $\underline{Rep}^{ab}(S_{\nu})^{op}$.

\end{theorem}

The exactness of $\widehat{SW}_{\nu}$ will be proved in Lemma \ref{lem:overline_SW_exact}.

The \InnaA{rest} of the proof of Theorem \ref{thrm:SW_almost_equiv} will be done by considering separately semisimple and non-semisimple blocks in $\underline{Rep}^{ab}(S_{\nu})$. The semisimple block case will be discussed in Subsection \ref{ssec:SW_duality_ss_block}, and the non-semisimple block case will be discussed in Subsection \ref{ssec:SW_duality_non_ss_block} (specifically, Proposition \ref{prop:SW_ess_surj} and Theorem \ref{thrm:hat_overline_SW_equiv}).

\InnaB{We also prove the following proposition (Proposition \ref{prop:dualities_SW_commute}):
\begin{proposition}
 For any $\nu \in \bC$, there is an isomorphism of (covariant) functors $$ \widehat{SW}_{\nu}(( \cdot )^*) \longrightarrow \hat{\pi}(SW_{\nu}( \cdot )^{\vee})$$
\end{proposition}

This proposition is proved in Subsection \ref{ssec:dualities_SW_commute}.}
$$ $$

We now introduce the following notation.

\begin{definition}
 We will denote by $SW^{*, ind}_{\nu}$ the contravariant functor $$Mod_{\mathcal{U}(\gl(V))} \longrightarrow Ind-\underline{Rep}^{ab}(S_{\nu})$$ which is right adjoint to $SW^{ind}_{\nu}$ (such a functor exists by Theorem \ref{thrm:ind_completion_prop}, since $SW^{ind}_{\nu}$ obviously commutes with small colimits).

 We will also denote by $SW^*_{\nu}$ the restriction of $SW^{*, ind}_{\nu}$ to $\co^{\mathfrak{p}}_{\nu, V}$.
\end{definition}
\begin{remark}
 The functors $SW^{*, ind}_{\nu}$, $SW^*_{\nu}$ are contravariant, $\bC$-linear, additive and left-exact (due to $SW^{*, ind}_{\nu}$ being a right-adjoint).
\end{remark}

 We will use the following notation:
 \begin{notation}\label{notn:eta_eps_units}
  The unit natural transformations corresponding to the contravariant adjoint functors $SW^{ind}_{\nu}, SW^{*, ind}_{\nu}$ will be denoted by $$\eta: \id_{Mod_{\mathcal{U}(\gl(V))}} \rightarrow SW^{ind}_{\nu} \circ SW^{*, ind}_{\nu}, \, \epsilon: \id_{Ind-\Dab} \rightarrow SW^{*, ind}_{\nu} \circ SW^{ind}_{\nu}$$

 \end{notation}
  In particular, we have the restriction of the natural transformation $\epsilon$:
  $$ \epsilon: \iota_{\Dab \rightarrow Ind-\Dab} \rightarrow SW^{*}_{\nu} \circ SW_{\nu}$$
These transformations satisfy the following conditions (see \cite[Chapter 1, par. 1, Theorem 1]{MacL}):
\begin{lemma}\label{lem:cond_on_units}
\begin{align*}
 &\forall E \in \co^{\mathfrak{p}}_{\nu, V}, \, SW^*_{\nu}(\eta_E) \circ \epsilon_{SW^*_{\nu}(E)} = \id_{SW^{*, ind}_{\nu}(E)}\\
 &\forall X \in \Dab, \, SW_{\nu}(\epsilon_X)  \circ \eta_{SW_{\nu}(X)} = \id_{SW_{\nu}(X)}
\end{align*}

\end{lemma}

We can now demonstrate the exactness of $\widehat{SW}_{\nu}$:

\begin{lemma}\label{lem:overline_SW_exact}
 The functor $\widehat{SW}_{\nu}: \underline{Rep}^{ab}(S_{\nu}) \rightarrow \widehat{\co}^{\mathfrak{p}}_{\nu, V}$ is exact.
\end{lemma}

\begin{proof}
Let $M \in \underline{Rep}^{ab}(S_{\nu})$, and let $i>0$. We want to show that the $\gl(V)$-module $\Ext^i(M, V^{\underline{\otimes}  \nu})$ is finite dimensional.

Consider $V^{\underline{\otimes}  \nu}$ as an object in $\underline{Rep}(S_{\nu})$. As such, it is a direct sum $ \bigoplus_{\lambda} X_{\lambda} \otimes V_{\lambda}$, where $V_{\lambda}$ is the multiplicity space of $X_{\lambda}$ (in fact, \InnaB{for a fixed splitting $V = \bC\triv \oplus U$,} $V_{\lambda}$ has the structure of a $\gl(U)$-module).

We know from Proposition \ref{prop:obj_ab_env} that $X_{\lambda}$ are injective objects iff $\tilde{\lambda}(\nu)$ is not a Young diagram. Furthermore, there are only finitely many Young diagrams $\lambda$ such that $\tilde{\lambda}(\nu)$ is a Young diagram as well; for these $\lambda$, the space $V_{\lambda}$ is finite dimensional and isomorphic to $S^{\tilde{\lambda}(\nu)} V$ (by Proposition \ref{prop:comp_tens_power_F_n}).

Finally, notice that $\Ext^i(M, X_{\lambda})$ is finite dimensional for any Young diagram $\lambda$, since all the $\Hom$-spaces in $\underline{Rep}^{ab}(S_{\nu})$ are finite-dimensional (c.f. Remark \ref{rmrk:ab_env_preTannakian}). We conclude that \InnaA{$$\Ext^i(M, V^{\underline{\otimes}  \nu}) \cong \bigoplus_{\substack{\lambda:\\ \tilde{\lambda}(\nu) \text{ is a Young diagram}}} \Ext^i(M, X_{\lambda})\otimes S^{\tilde{\lambda}(\nu)} V $$} is finite dimensional.
\end{proof}

\subsection{Case of a semisimple block}\label{ssec:SW_duality_ss_block}
In this subsection we consider a semisimple block in $\underline{Rep}^{ab}(S_{\nu})$. We know that semisimple blocks are parametrized by Young diagrams lying in a trivial $\stackrel{\nu}{\sim}$-class. Let us denote our block by $\mathcal{B}_{\lambda}$, $\lambda$ being the corresponding Young diagram.

The objects of such a block are finite direct sums of the simple object $ X_{\lambda}$, so the block is equivalent to $\mathbf{Vect}_{\bC}$ as an abelian category.

If $\ell(\lambda) \leq \dim V -1$, then the block $\mathfrak{B}_{\lambda}$ corresponding to $\lambda$ in $\co^{\mathfrak{p}}_{\nu, V}$ is also semisimple, and its objects are finite direct sums of the parabolic Verma module $M_{\mathfrak{p}}(\nu-\abs{\lambda}, \lambda)$ (which, in this case, is simple and coincides with $L_{\mathfrak{p}}(\nu-\abs{\lambda}, \lambda)$). 

Notice that $M_{\mathfrak{p}}(\nu-\abs{\lambda}, \lambda)$ is infinite-dimensional and simple, so the functor $\hat{\pi}$ restricted to $\mathfrak{B}_{\lambda}$ is an equivalence of abelian categories.

\InnaA{The proof of Theorem \ref{thrm:SW_almost_equiv} for $\mathcal{B}_{\lambda}$ is then reduced to proving following proposition:}

\begin{proposition}\label{prop:image_SW_ss_case}
 Let $\lambda$ be a Young diagram which lies in a trivial $\stackrel{\nu}{\sim}$-class. Then $SW_{\nu}(X_{\lambda}) \cong M_{\mathfrak{p}}(\nu-\abs{\lambda}, \lambda)$.
\end{proposition}
\begin{remark}
Recall that $M_{\mathfrak{p}}(\nu-\abs{\lambda}, \lambda)$ is zero if $\ell(\lambda) > \dim V -1$, see Definition \ref{def:parabolic_Verma_mod}.
\end{remark}

\begin{proof}
 \InnaB{Fix a splitting $V = \bC \triv \oplus U$.} Based on Lemma \ref{lem:hom_X_tau_Delta_k}, we see that as a $\gl(U)$-module, the space 
 $$SW_{\nu}(X_{\lambda}) = \Hom_{Ind-\underline{Rep}^{ab}(S_{\nu})}(X_{\lambda}, V^{\underline{\otimes}  \nu})$$
 is isomorphic to \begin{align*}
 &\Hom_{Ind-\underline{Rep}^{ab}(S_{\nu})}(X_{\lambda}, \bigoplus_{k \geq 0} (U^{\otimes k} \otimes \Delta_k)^{S_k}) \cong \bigoplus_{k \geq 0} \bigoplus_{\mu: \abs{\mu}=k} \Hom_{\underline{Rep}(S_{\nu}) \boxtimes Rep(S_k)}(X_{\lambda} \otimes \mu, \Delta_k) \otimes S^{\mu} U \cong \\
 &\cong \bigoplus_{\mu \in \mathcal{I}^{+}_{\lambda}} S^{\mu} U                                                                                                                                                                                                                                                                                                                                                                                                                                   \end{align*}
 Notice that this expression is zero if $\ell(\lambda) > \dim(U) = \dim V -1$.
  Recall that by definition of $V^{\underline{\otimes}  \nu}$, $\mathfrak{u}_{\mathfrak{p}}^+$ acts on the graded space $V^{\underline{\otimes}  \nu} \cong \bigoplus_{k \geq 0} (U^{\otimes k} \otimes \Delta_k)^{S_k}$ by operators of degree $-1$, therefore, it acts by zero on the subspace $S^{\lambda} U$ of $\Hom_{Ind-\underline{Rep}^{ab}(S_{\nu})}(X_{\lambda}, V^{\underline{\otimes}  \nu})$.

  We conclude that if $\ell(\lambda) \leq \dim(U) = \dim V -1$, then $M_{\mathfrak{p}}(\nu-\abs{\lambda}, \lambda)$ maps to $SW_{\nu}(X_{\lambda})$ inducing an identity map on the subspaces $S^{\lambda} U$. Now, $M_{\mathfrak{p}}(\nu-\abs{\lambda}, \lambda)$ is simple, so this map is injective, and since $$SW_{\nu}(X_{\lambda})\lvert_{\gl(U)} \cong\bigoplus_{\mu \in \mathcal{I}^{+}_{\lambda}} S^{\mu} U \cong M_{\mathfrak{p}}(\nu-\abs{\lambda}, \lambda)\lvert_{\gl(U)}$$
  as $\gl(U)$-modules, we conclude that the above map from $M_{\mathfrak{p}}(\nu-\abs{\lambda}, \lambda)$ to $SW_{\nu}(X_{\lambda})$ is an isomorphism.
\end{proof}
Thus we proved that
\begin{corollary}
The functor $SW_{\nu}$ restricted to a semisimple block $\mathcal{B}_{\lambda}$ of $\underline{Rep}^{ab}(S_{\nu})$ is either zero (iff $\ell(\lambda) > \dim V -1$), or is an equivalence of abelian categories between $\mathcal{B}_{\lambda}$ and the block $\mathfrak{B}_{\lambda}$ of $\co^{\mathfrak{p}}_{\nu, V}$; furthermore, the functor $\widehat{SW}_{\nu}$ restricted to $\mathcal{B}_{\lambda}$ is either zero or an equivalence of abelian categories between $\mathcal{B}_{\lambda}$ and the block $\hat{\pi}(\mathfrak{B}_{\lambda})$ of $\widehat{\co}^{\mathfrak{p}}_{\nu, V}$.
\end{corollary}

Recall from Section \ref{sec:Del_cat_S_nu} that for $\nu \notin \bZ_+$, $\underline{Rep}(S_{\nu})$ is abelian semisimple and in particular $\underline{Rep}^{ab}(S_{\nu}) = \underline{Rep}(S_{\nu})$. 
Denote by $\underline{Rep}(S_{\nu})^{(\leq \dim V -1)}$ the full semisimple abelian subcategory of $\underline{Rep}(S_{\nu})$ generated by simple objects $X_{\lambda}$ where $\lambda$ runs over all the Young diagrams of length at most $\dim V -1$. 

Note that $\underline{Rep}(S_{\nu})^{(\leq \dim V -1)}$ is the Serre quotient of $\underline{Rep}(S_{\nu})$ by the full semisimple abelian subcategory generated by simple objects $X_{\lambda}$ where $\lambda$ runs over all the Young diagrams of length at least $\dim V $.

Then we immediately get the following corollary:

\begin{corollary}\label{cor:SW_equiv_non_int_case}
 Assume $\nu \notin \bZ_+$. Then $SW_{\nu}: \underline{Rep}(S_{\nu}) \rightarrow \co^{\mathfrak{p}}_{\nu, V}$ is a full, essentially surjective, additive $\bC$-linear contravariant functor between semisimple abelian categories, inducing an anti-equivalence of abelian categories between $\underline{Rep}(S_{\nu})^{(\leq \dim V -1)}$ and $\co^{\mathfrak{p}}_{\nu, V}$.
\end{corollary}

\begin{remark}
 \InnaA{If $\nu \notin \bZ_+$, then $\widehat{\co}^{\mathfrak{p}}_{\nu, V} \cong \co^{\mathfrak{p}}_{\nu, V}$ with $\hat{\pi} \cong \id_{\co^{\mathfrak{p}}_{\nu, V}}$}, so Corollary \ref{cor:SW_equiv_non_int_case} is just Theorem \ref{thrm:SW_almost_equiv} in the case $\nu \notin \bZ_+$.
\end{remark}

\subsection{Case of a non-semisimple block}\label{ssec:SW_duality_non_ss_block}
Throughout this subsection, we will use the \InnaA{results from} Sections \ref{sec:Del_cat_S_nu} and \ref{sec:par_cat_o},\InnaA{ and we will denote $\co^{\mathfrak{p}}_{\nu, V}$ by $\co^{\mathfrak{p}}_{\nu}$ for short.}

\InnaB{Fix a splitting $V = \bC \triv \oplus U$.}

In this subsection we consider a non-semisimple block in $\underline{Rep}^{ab}(S_{\nu})$. Recall that such blocks occur only when $\nu \in \bZ_+$, \InnaA{so we will assume that this is the case}. 

We know that non-semisimple blocks are parametrized by Young diagrams $\lambda$ such that $\lambda_1 +\abs{\lambda} \leq \nu$; the projective objects in \InnaA{such a} block correspond to the elements of the (non-trivial) $\stackrel{\nu}{\sim}$-class of $\lambda$ (see Proposition \ref{prop:obj_ab_env}). 

Let us denote our block by $\mathcal{B}_{\lambda}$.

If $\ell(\lambda) \leq \dim V -1$, then the block $\mathfrak{B}_{\lambda}$ corresponding to $\lambda$ in $\co^{\mathfrak{p}}_{\nu}$ is also non-semisimple.
We will continue with the blocks $\mathcal{B}_{\lambda}$, $\mathfrak{B}_{\lambda}$ fixed, and insert some notation for the convenience of the reader.

\begin{notation}
 We will denote the simple objects, standard objects, co-standard and indecomposable projective objects in $\mathcal{B}_{\lambda}$ by $\mathbf{L}_i, \mathbf{M}_i, \InnaB{\mathbf{M}^*_i}, \mathbf{P}_i$ ($i \in \bZ_+$) respectively, with $\mathbf{L}_i$ standing for $\mathbf{L}(\lambda^{(i)})$  and similarly for $\mathbf{M}_i$, $\InnaB{\mathbf{M}^*_i}$ and $\mathbf{P}_i$. The structure of these objects is discussed in Subsection \ref{ssec:S_nu_abelian_env}.

 Notice that $\mathbf{M}_0 = \InnaB{\mathbf{M}^*_0} =\mathbf{L}_0 = X_{\lambda^{(0)}}$, $\mathbf{P}_i  = X_{\lambda^{(i+1)}}$ for $i\in \bZ_+$ (see Proposition \ref{prop:obj_ab_env}).
\end{notation}

\begin{notation}
We will denote the simple modules, the parabolic Verma modules, their duals (the co-standard objects in $\co^{\mathfrak{p}}_{\nu}$) and the indecomposable parabolic projective modules in $\mathfrak{B}_{\lambda}$ by $L_i, M_i, M^{\vee}_i, P_i$ ($i \in \bZ_+$) respectively, with $M_i$ standing for $M_{\mathfrak{p}}(\nu-\abs{\lambda^{(i)}}, \lambda^{(i)})$ and similarly for $L_i, M^{\vee}_i$ and $P_i$. 

The structure of the modules $L_i, M_i, M^{\vee}_i, P_i, (i \in \bZ_+)$ is discussed in Section \ref{sec:par_cat_o} and in \cite[Chapter 9]{H}.

We put $k_{\lambda} := \min \{k \geq 0 \mid \ell(\lambda^{(k)}) > \dim V -1 \}$. Then $P_i = M_i = M^{\vee}_{\InnaB{i}} = L_i = 0$ whenever $ i \geq k_{\lambda}$.
\end{notation}

The goal of this section is to prove Theorem \ref{thrm:SW_almost_equiv} for the blocks $\mathcal{B}_{\lambda}$, $\mathfrak{B}_{\lambda}$. In order to do this, we will prove the following theorem:

\begin{theorem}\label{thrm:images_SW}
The functor $SW_{\nu}$ satisfies:
 \begin{enumerate}[label=(\alph*)]
 \item $SW_{\nu}(\mathbf{L}_i) \cong L_{i+1}$ whenever $i\geq 1$.
 \item $SW_{\nu}(\mathbf{M}_i) \cong M_{i}$ whenever $i\geq 0$.
 \item $SW_{\nu}(\InnaB{\mathbf{M}^*_i}) \cong M^{\vee}_{i}$ whenever $i \geq 2$.
 \item $SW_{\nu}(\mathbf{P}_{i}) \cong P_{i+1}$ whenever $i \geq 0$ and $i < k_{\lambda} -1$ or $i \geq k_{\lambda}$ (recall that in the latter case $P_{i+1}=0$); $SW_{\nu}(\mathbf{P}_{k_{\lambda}-1}) \cong L_{k_{\lambda}-1}$.
 \item $SW_{\nu}(\mathbf{M}_0 = \InnaB{\mathbf{M}^*_0} =\mathbf{L}_0) \cong M_0$.
  \item $SW_{\nu}(\InnaB{\mathbf{M}^*_1}) \cong Ker(P_1 \twoheadrightarrow L_1)$.
\end{enumerate}

\end{theorem}
\begin{proof}
Statement (a) is proved in Proposition \ref{prop:image_SW_simples}.
Statements (b)-(d), (f) are proved in Proposition \ref{prop:images_SW}. Statement (e) is proved in Lemma \ref{lem:images_exceptionals}.
\end{proof}

We start by establishing some useful properties of the functor $\hat{\pi}$ and of the category $\hat{\pi}(\mathfrak{B}_{\lambda})$.
\begin{proposition}\label{prop:propert_functor_pi}
$\hat{\pi}(P_i), i>0$ are indecomposable injective and projective objects in $\widehat{\co}^{\mathfrak{p}}_{\nu}$.
\end{proposition}
Recall that for $i>0$, $P_i$ is an indecomposable injective and projective module and has no finite-dimensional submodules nor quotients (c.f. Proposition
\ref{prop:proj_parab_cat_O_struct}). So Proposition \ref{prop:propert_functor_pi} is a special case of the following lemma:

\begin{lemma}\label{lem:proj_in_serre_quotient}
 Let $\mathcal{A}$ be an abelian category where all objects have finite length, and $\mathcal{A}'$ be a Serre subcategory of $\mathcal{A}$. We consider the Serre quotient $ \pi: \mathcal{A} \rightarrow \mathcal{A} /{\mathcal{A}'}$.
 
 Let $I \in \mathcal{A}$ (respectively, $P \in \mathcal{A}$) be an injective (respectively, projective) object, such that $I$ has no non-trivial subobject nor quotient lying in $\mathcal{A}'$.
 
 Then $\pi(I)$ is an injective (respectively, projective) object in $\mathcal{A} /{\mathcal{A}'}$.
 Moreover, if $I$ is indecomposable, so is $\pi(I)$.
\end{lemma}

 \begin{proof}
We start by noticing that we have two functors $R_1, R_2: \mathcal{A} \rightarrow \mathcal{A}'$ which are adjoint to the inclusion $\mathcal{A}' \rightarrow \mathcal{A}$ on different sides: the first functor, $R_1$, takes an object $A \in \mathcal{A}$ to its maximal subobject lying in $\mathcal{A}'$, and the second, $R_2$, takes $A$ to its maximal quotient lying in $\mathcal{A}'$.

These functors are defined since for any $A \in \mathcal{A}$, we can take its maximal (in terms of length) subobject lying in $\mathcal{A}'$, and this subobject will be well-defined. Similarly for the maximal quotient of $A$ lying in $\mathcal{A}'$.

We need to prove that $\Hom_{\mathcal{A} /{\mathcal{A}'}}(\cdot, {\pi}(I))$ is an exact functor.
  By definition, for any $E \in \mathcal{A}$,
  $$\Hom_{\mathcal{A} /{\mathcal{A}'}}({\pi}(E), {\pi}(I)) := \varinjlim_{\substack{Y \subset E, X \subset I\\ E/Y, X \in  \mathcal{A}' }} \Hom_{\mathcal{A}}(Y, I/X) = \varinjlim_{\substack{Y \subset E\\ E/Y \in \mathcal{A}' }} \Hom_{\mathcal{A}}(Y, I) $$
  (since $I$ has no non-trivial subobjects lying in $\mathcal{A}'$).

  The colimit is taken with respect to the direct system
  \begin{align*}
  &\{\Hom_{\mathcal{A}}(Y, I): Y \subset E, E/ Y \in  \mathcal{A}' \},\\
  &\text{ arrows } \Hom_{\mathcal{A}}(Y_2, I) \rightarrow \Hom_{\mathcal{A}}(Y_1, I)
 \text{ whenever }  Y_1 \hookrightarrow Y_2
  \end{align*}
  Now, $I$ is an injective object in $\mathcal{A}$, so the arrows in this direct system are surjective:
  \begin{align*}
  &\{\Hom_{\mathcal{A}}(Y, I): Y \subset E, E/ Y \in  \mathcal{A}' \},\\
  &\text{ arrows } \Hom_{\mathcal{A}}(Y_2, I) \twoheadrightarrow \Hom_{\mathcal{A}}(Y_1, I) \text{ whenever }  Y_1 \hookrightarrow Y_2
  \end{align*}
  Then one easily sees that the colimit is $\Hom_{\mathcal{A}}(Y_E, I) $, where $Y_E := Ker( E \twoheadrightarrow R_2(E))$. Thus
  $$\Hom_{\mathcal{A} /{\mathcal{A}'}}(\pi(E), \pi(I)) :=\Hom_{\mathcal{A}}(Y_E, I)$$
  So we need to prove that given an exact sequence $$0 \rightarrow E' \longrightarrow E \longrightarrow E'' \rightarrow 0$$
  of objects in $\mathcal{A}$, the sequence $$0 \rightarrow \Hom_{\mathcal{A}}(Y_{E'}, I) \longrightarrow \Hom_{\mathcal{A}}(Y_E, I) \longrightarrow \Hom_{\mathcal{A}}(Y_{E''}, I) \rightarrow 0$$ is also exact.
  Notice that since $I$ is injective in $\mathcal{A}$, we have an exact sequence
  $$0 \rightarrow \Hom_{\mathcal{A}}(R_2(E), I) \rightarrow \Hom_{\mathcal{A}}(E, I) \rightarrow \Hom_{\mathcal{A}}(Y_E, I)\rightarrow 0$$
  and since $I$ has no non-trivial \InnaA{subobjects} in $\mathcal{A}'$, we get $\Hom_{\mathcal{A}}(E, I) \cong  \Hom_{\mathcal{A}}(Y_E, I)$. The sequence $$0 \rightarrow \Hom_{\mathcal{A}}({E'}, I) \longrightarrow \Hom_{\mathcal{A}}(E, I) \longrightarrow \Hom_{\mathcal{A}}({E''}, I) \rightarrow 0$$ is exact (since $I$ is injective in $\mathcal{A}$), so the sequence $$0 \rightarrow \Hom_{\mathcal{A}}(Y_{E'}, I) \longrightarrow \Hom_{\mathcal{A}}(Y_E, I) \longrightarrow \Hom_{\mathcal{A}}(Y_{E''}, I) \rightarrow 0$$ is exact as well.

  Thus we proved that $\Hom_{\mathcal{A} /{\mathcal{A}'}}(\cdot, \pi(I))$ is an exact functor, so $\pi(I)$ is an injective object in $\mathcal{A} /{\mathcal{A}'}$.

  The fact that $\pi(P)$ is a projective object in $\mathcal{A} /{\mathcal{A}'}$ is proved in the same way.

  Now, assume $\pi(I)$ is decomposable, $\pi(I) \cong X_1 \oplus X_2$ in \InnaA{$\mathcal{A} /{\mathcal{A}'}$}, $X_1, X_2 \neq 0$. Then we can find $E_1, E_2 \in \mathcal{A}$ such that $E_1, E_2$ have no non-trivial subobject nor quotient lying in $\mathcal{A}'$, and such that $\pi(E_i) =X_i$, $i=1,2$. Then one immediately sees that 
  \begin{align*}
   &\Hom_{\mathcal{A}}(I, E_i) =  \Hom_{\mathcal{A} /{\mathcal{A}'}}({\pi}(I), X_i), \Hom_{\mathcal{A}}( E_i, I) = \Hom_{\mathcal{A} /{\mathcal{A}'}}(X_i, {\pi}(I)), \\
   &\Hom_{\mathcal{A}}( E_i, E_j) = \Hom_{\mathcal{A} /{\mathcal{A}'}}(X_i, X_j)
  \end{align*}
  for $i=1,2$. In particular, since $\pi$ is exact and $I, E_1, E_2$ have no non-trivial subobject nor quotient lying in $\mathcal{A}'$, we see that $E_1 \oplus E_2 \cong I$. We conclude that if $I$ is indecomposable, so is \InnaA{$\pi(I)$}.
 \end{proof}

 The following corollary will be useful when proving that the functor $\widehat{SW}_{\nu}$ is full and essentially surjective:
 \begin{corollary}\label{cor:hat_O_p_enough_inj_proj}
 The image of the category $\mathfrak{B}_{\lambda}$ under the functor $\hat{\pi}$ has enough injectives and enough projectives. Moreover, \InnaA{$\{ \hat{\pi}(P_i) \}_{0<i\leq k_{\lambda}-1}$} is the full set of representatives of isomorphism classes of indecomposable injective (respectively, projective) objects in $\hat{\pi}(\mathfrak{B}_{\lambda})$.
\end{corollary}

 \begin{proof}
 Let $E \in \mathfrak{B}_{\lambda}$. We know from Proposition \ref{prop:proj_in_ab_envelope} that the category $\mathfrak{B}_{\lambda}$ has enough injective and enough projective modules. This means that there exist an injective module $I$ and a projective module $P$, together with an injective map $E \rightarrow I$ and a surjective map $P \rightarrow E$.

 Since the functor $\hat{\pi}$ is exact, we get an injective map $\hat{\pi}(E) \rightarrow \hat{\pi}(I)$ and a surjective map $\hat{\pi}(P) \rightarrow \hat{\pi}(E)$.

 Next, recall that $\{P_i, \InnaA{0 \leq i\leq k_{\lambda}-1}\}$ (respectively, $\{P^{\vee}_i, \InnaA{0 \leq i\leq k_{\lambda}-1}\}$) is the full set of representatives of isomorphism classes of indecomposable projective (respectively, injective) modules in $\mathfrak{B}_{\lambda}$. For $i>0$, the object $\hat{\pi}(P_i \cong P^{\vee}_i)$ was proved to be injective and projective (c.f. Proposition \ref{prop:propert_functor_pi}), so \InnaA{it remains to check the following statement}:
 
 The object $\hat{\pi}(P_0) \cong \hat{\pi}(L_1) \cong \hat{\pi}(P^{\vee}_0)$ is neither injective nor projective, and has a projective cover and an injective hull in $\mathfrak{B}_{\lambda}$ which are direct sums of objects $\hat{\pi}(P_i), i>0$.

 To prove the \InnaA{latter} claim, notice that the maps $L_1 \hookrightarrow P_1$, $P_1 \twoheadrightarrow L_1$ in $\mathfrak{B}_{\lambda}$ become maps $\hat{\pi}(L_1) \hookrightarrow \hat{\pi}(P_1)$, $\hat{\pi}(P_1) \twoheadrightarrow \hat{\pi}(L_1)$ \InnaA{in $\hat{\pi}(\mathfrak{B}_{\lambda})$} (since the functor $\hat{\pi}$ is exact). Knowing that $\hat{\pi}(P_1)$ is an indecomposable injective and projective object, we conclude that $\hat{\pi}(L_1)$ is neither injective nor projective, and $\hat{\pi}(P_1)$ is both its projective cover and its injective hull.
 \end{proof}

We now compute the decomposition into $\gl(U)$-irreducibles of $SW_{\nu}(\mathbf{L}_i), SW_{\nu}(\mathbf{M}_{i}), SW_{\nu}(\mathbf{P}_{i})$:
 \begin{lemma}\label{lem:gl_u_decomp_images_SW}
  We have the following isomorphisms of $\gl(U)$-modules:
  \InnaA{
  \begin{align*}
&SW_{\nu}(\mathbf{L}_i)\rvert_{\gl(U)} \cong \bigoplus_{\mu \in \mathcal{I}^{+}_{\lambda^{(i)}}\cap \mathcal{I}^{+}_{\lambda^{(i+1)}}} S^{\mu} U \; \text{,} \; SW_{\nu}(\mathbf{M}_{i})\rvert_{\gl(U)} \cong \bigoplus_{\mu \in  \mathcal{I}^{+}_{\lambda^{(i)}}} S^{\mu} U   &\forall i\geq 1
\\
&SW_{\nu}(\mathbf{L}_0=\mathbf{M}_0= \InnaB{\mathbf{M}^*_0} \cong X_{\lambda^{(0)}})\rvert_{\gl(U)} \cong \bigoplus_{\mu \in \mathcal{I}^{+}_{\lambda^{(0)}}} S^{\mu} U \\
&SW_{\nu}(\mathbf{P}_{i} \cong X_{\lambda^{(i+1)}})\rvert_{\gl(U)} \cong \bigoplus_{\mu \in \mathcal{I}^{+}_{\lambda^{(i)}}} S^{\mu} U   \oplus \bigoplus_{\mu \in  \mathcal{I}^{+}_{\lambda^{(i+1)}}} S^{\mu} U &\forall i\geq 0\\
&SW_{\nu}(\InnaB{\mathbf{M}^*_{i}})\rvert_{\gl(U)} \cong \bigoplus_{\mu \in  \mathcal{I}^{+}_{\lambda^{(i)}}} S^{\mu} U  &\forall i \geq 2
\end{align*}
}
 \end{lemma}
\begin{proof}
 Consider $V^{\underline{\otimes}  \nu}$ as an object in $\underline{Rep}(S_{\nu})$. As such, it is a direct sum $ \bigoplus_{\mu} X_{\mu} \otimes V_{\mu}$, where $V_{\mu}$ is the multiplicity space of $X_{\mu}$. In fact, $V_{\mu}$ has the structure of $\bZ_+$-graded $\gl(U)$-module, each grade being a polynomial $\gl(U)$-module. 

 We now consider the full subcategory $\mathcal{D}$ of $\mathcal{B}_{\lambda}$ whose objects are those which do not have $\mathbf{L}_0$ among their composition factors. Recall that $\mathbf{L}_i \in \mathcal{D}$ for $i>0$, and $\mathbf{M}_i, \InnaB{\mathbf{M}^*_i}, \mathbf{P}_i \in \mathcal{D}$ whenever $i \geq 2$.

 We will denote by $\mathcal{F}$ the following functor from $\mathcal{D}$ to the category \InnaA{$Ind-Mod_{\mathcal{U}(\gl(U)), poly}$:}
 %
 \begin{align*}
  \mathcal{F} := SW_{\nu}(\cdot)\rvert_{\gl(U)} = \Hom_{Ind-\underline{Rep}^{ab}(S_{\nu})}(\cdot,  V^{\underline{\otimes}  \nu})
 \end{align*}

Next, for any $X \in \mathcal{D}$, we have the following isomorphism of $\gl(U)$-modules:
 $$ \mathcal{F}(X) = \Hom_{Ind-\underline{Rep}^{ab}(S_{\nu})}(X,  V^{\underline{\otimes}  \nu}) \cong \Hom_{Ind-\underline{Rep}^{ab}(S_{\nu})}(X,  \bigoplus_{i>0} X_{\lambda^{(i)}} \otimes V_{\lambda^{(i)}})$$
Since we know that $X_{\lambda^{(i)}} = \mathbf{P}_{i-1}$ is injective for $i>0$ (see Proposition \ref{prop:obj_ab_env}), we immediately conclude that $\mathcal{F}$ is exact.

Now, one easily sees from the $\gl(U)$-decomposition \InnaA{$V^{\underline{\otimes}  \nu} \rvert_{\gl(U)} = \bigoplus_{k \geq 0} (U^{\otimes k} \otimes \Delta_k)^{S_k} $}, together with Lemma \ref{lem:hom_X_tau_Delta_k}, that for any $i \geq 0$,
\begin{align*}                                                                                                                                                 & SW_{\nu}(\mathbf{P}_{i} = X_{\lambda^{(i+1)}})\rvert_{\gl(U)} \cong \Hom_{Ind-\underline{Rep}^{ab}(S_{\nu})}(X_{\lambda^{(i+1)}}, \bigoplus_{k \geq 0} (U^{\otimes k} \otimes \Delta_k)^{S_k}) \cong \\
&\cong \bigoplus_{k \geq 0} \bigoplus_{\mu: \abs{\mu}=k} \Hom_{\underline{Rep}(S_{\nu}) \boxtimes Rep(S_k)}(X_{\lambda^{(i+1)}} \otimes \mu,  \Delta_k) \otimes S^{\mu} U \cong
\bigoplus_{\mu \in \mathcal{I}^{+}_{\lambda^{(i)}}} S^{\mu} U \oplus \bigoplus_{\mu \in \mathcal{I}^{+}_{\lambda^{(i+1)}}} S^{\mu} U
\end{align*}
(the $k$-th grade of $SW_{\nu}(\mathbf{P}_{i})$ is the direct sum of $S^{\mu} U$ such that $\abs{\mu}=k$).

Fix $i \geq 1$.
We can now apply $\mathcal{F}$ to the following long exact sequences in $\mathcal{D}$ (these exact sequences exist due to Proposition \ref{prop:obj_ab_env}):
$$ ... \rightarrow \mathbf{P}_{i+3} \rightarrow \mathbf{P}_{i+2} \rightarrow \mathbf{P}_{i+1} \rightarrow \mathbf{M}_{i+1} \rightarrow 0$$
$$ 0 \rightarrow  \InnaB{\mathbf{M}^*_{i+1}} \rightarrow \mathbf{P}_{i+1} \rightarrow \mathbf{P}_{i+2} \rightarrow \mathbf{P}_{i+3} \rightarrow ...$$
$$ 0 \rightarrow \mathbf{L}_{i} \rightarrow \mathbf{M}_{i+1} \rightarrow \mathbf{M}_{i+2} \rightarrow \mathbf{M}_{i+3} \rightarrow ...$$

Using $$\mathcal{F}(\mathbf{P}_{i}) \cong \bigoplus_{\mu \in \mathcal{I}^{+}_{\lambda^{(i)}}\cup \mathcal{I}^{+}_{\lambda^{(i+1)}}} S^{\mu} U$$ and the fact that $\mathcal{F}$ is exact, we conclude that
$$\mathcal{F}(\mathbf{M}_{i+1}) \cong \bigoplus_{\mu \in  \mathcal{I}^{+}_{\lambda^{(i+1)}}} S^{\mu} U \cong \mathcal{F}(\InnaB{\mathbf{M}^*_{i+1}}) $$ and
$$ \mathcal{F}(\mathbf{L}_{i}) \cong V_{\lambda^{(i+1)}} \cong \bigoplus_{\mu \in  \mathcal{I}^{+}_{\lambda^{(i+1)}} \cap \mathcal{I}^{+}_{\lambda^{(i)}}} S^{\mu} U$$

It remains to check the $\gl(U)$-structure of $SW_{\nu}(\mathbf{L}_0 \cong X_{\lambda^{(0)}})\rvert_{\gl(U)}$, $SW_{\nu}(\mathbf{M}_1)\rvert_{\gl(U)}$.

Similarly to the decomposition of $\mathcal{F}(\mathbf{P}_i)$, we use the $\gl(U)$-decomposition $V^{\underline{\otimes}  \nu} = \bigoplus_{k \geq 0} (U^{\otimes k} \otimes \Delta_k)^{S_k} $, together with Lemma \ref{lem:hom_X_tau_Delta_k}, to get the following isomorphisms of $\gl(U)$-modules:
\begin{align*}                                                                                                                                                 &SW_{\nu}(X_{\lambda^{(0)}})\rvert_{\gl(U)} \cong \Hom_{Ind-\underline{Rep}^{ab}(S_{\nu})}(X_{\lambda^{(0)}}, \bigoplus_{k \geq 0} (U^{\otimes k} \otimes \Delta_k)^{S_k}) \cong \\
&\cong \bigoplus_{k \geq 0} \bigoplus_{\mu: \abs{\mu}=k} \Hom_{\underline{Rep}(S_{\nu}) \boxtimes Rep(S_k)}(X_{\lambda^{(0)}} \otimes \mu,  \Delta_k) \otimes S^{\mu} U \cong \bigoplus_{\mu \in \mathcal{I}^{+}_{\lambda^{(0)}}} S^{\mu} U
\end{align*}
\InnaA{In particular, we have: 
$$V_{\lambda^{(0)}} \oplus V_{\lambda^{(1)}} \cong SW_{\nu}(X_{\lambda^{(0)}})\rvert_{\gl(U)} \cong \bigoplus_{\mu \in \mathcal{I}^{+}_{\lambda^{(0)}}} S^{\mu} U$$ 
Recall that $$V_{\lambda^{(0)}} \cong (S^{\tilde{\lambda^{(0)}}(\nu)} V) \rvert_{\gl(U)} \cong \bigoplus_{\mu : \, \tilde{\lambda^{(0)}}(\nu) \in \mathcal{I}^{+}_{\mu}} S^{\mu} U$$ (c.f. Proposition \ref{prop:comp_tens_power_F_n}).
Now, for any Young diagram $\mu$, we have (c.f. the proof of Lemma \ref{lem:hom_X_tau_Delta_k}):   
$$ \tilde{\lambda^{(0)}}(\nu) \in \mathcal{I}^{+}_{\mu} \, \Leftrightarrow \, \left[ \mu \in \mathcal{I}^{+}_{\lambda^{(0)}} \text{ and } \mu_1 +\abs{\lambda^{(0)}} \leq \nu \right]$$ 
On the other hand, the description of non-trivial $\stackrel{\nu}{\sim}$-classes (c.f. Lemma \ref{lem:nu_classes_struct}) tells us that $\lambda^{(0)} \subset \lambda^{(1)}$, and $\lambda^{(1)} \setminus \lambda^{(0)}$ is a strip in row $1$ of length $\nu - \abs{\lambda^{(0)}} - \lambda^{(0)}_1 +1$. Thus 
$$\{\mu : \tilde{\lambda^{(0)}}(\nu) \in \mathcal{I}^{+}_{\mu}\} = \mathcal{I}^{+}_{\lambda^{(0)}} \setminus \mathcal{I}^{+}_{\lambda^{(1)}}$$
and so $V_{\lambda^{(1)}} \cong \bigoplus_{\mu \in \mathcal{I}^{+}_{\lambda^{(0)}} \cap  \mathcal{I}^{+}_{\lambda^{(1)}} } S^{\mu} U$. \InnaB{We have already seen that $V_{\lambda^{(2)}} \cong \bigoplus_{\mu \in  \mathcal{I}^{+}_{\lambda^{(2)}} \cap \mathcal{I}^{+}_{\lambda^{(1)}}} S^{\mu} U$, and we conclude that }
\begin{align*}                                                                                                                                                 &SW_{\nu}(\mathbf{M}_1)\rvert_{\gl(U)} \cong \Hom_{Ind-\underline{Rep}^{ab}(S_{\nu})}(\mathbf{M}_1, V^{\underline{\otimes}  \nu}) = \Hom_{Ind-\underline{Rep}^{ab}(S_{\nu})}(\mathbf{M}_1,  \bigoplus_{i \geq 0} X_{\lambda^{(i)}} \otimes V_{\lambda^{(i)}}) \cong \\
&\cong V_{\lambda^{(1)}} \oplus V_{\lambda^{(2)}} \cong 
 \bigoplus_{\mu \in \mathcal{I}^{+}_{\lambda^{(0)}} \cap  \mathcal{I}^{+}_{\lambda^{(1)}} } S^{\mu} U \oplus \bigoplus_{\mu \in  \mathcal{I}^{+}_{\lambda^{(2)}} \cap \mathcal{I}^{+}_{\lambda^{(1)}}} S^{\mu} U \cong  \bigoplus_{\mu \in \mathcal{I}^{+}_{\lambda^{(1)}}} S^{\mu} U
\end{align*} 
(the last isomorphism can be inferred from Lemma \ref{lem:nu_classes_struct}).}

 \end{proof}

\begin{proposition}\label{prop:image_SW_simples}
 For a simple object $\mathbf{L}_i$ in $\mathcal{B}_{\lambda}$, $i>0$, we have: $SW_{\nu}(\mathbf{L}_i) \cong L_{i+1}$ (recall that the latter is defined to be zero if $i \geq k_{\lambda} -1$).
\end{proposition}

\begin{proof}
 Fix $i>0$.
By definition, $SW_{\nu}(\mathbf{L}_i):= \Hom_{Ind-\underline{Rep}^{ab}(S_{\nu})}(\mathbf{L}_i, V^{\underline{\otimes}  \nu})$ with $\gl(V)$-action on this space induced from the action of $\gl(V)$ on $V^{\underline{\otimes}  \nu}$.

By Lemma \ref{lem:gl_u_decomp_images_SW}, we have the following decomposition of $SW_{\nu}(\mathbf{L}_i)$ \InnaA{as a $\gl(U)$-module:}
$$SW_{\nu}(\mathbf{L}_i) \rvert_{\gl(U)} \cong \bigoplus_{\mu \in \mathcal{I}^{+}_{\lambda^{(i)}}\cap \mathcal{I}^{+}_{\lambda^{(i+1)}}} S^{\mu} U$$

If $i \geq k_{\lambda} -1$, then $\ell(\mu) > \dim V -1$ for any $\mu \in \mathcal{I}^{+}_{\lambda^{(i+1)}} \cap \mathcal{I}^{+}_{\lambda^{(i)}}$, so we get that $SW_{\nu}(\mathbf{L}_i) = \Hom_{Ind-\underline{Rep}^{ab}(S_{\nu})}(\mathbf{L}_i, V^{\underline{\otimes}  \nu}) = 0 = L_{i+1}$ and we are done.

Otherwise, notice that $\Hom_{Ind-\underline{Rep}^{ab}(S_{\nu})}(\mathbf{L}_i, V^{\underline{\otimes}  \nu})$ is a \InnaA{$\bZ_+$-}graded $\gl(U)$-module, with the grading inherited from $V^{\underline{\otimes}  \nu}$: $S^{\mu} U$ lies in grade $\abs{\mu}$. The minimal grade is thus $\abs{\lambda^{(i+1)}}$, and it consists of the $\gl(U)$-module $S^{\lambda^{(i+1)}} U$.

Recall that $\mathfrak{u}_{\mathfrak{p}}^+$ acts on the graded space $V^{\underline{\otimes}  \nu} \cong \bigoplus_{k \geq 0} (U^{\otimes k} \otimes \Delta_k)^{S_k}$ by operators of degree $-1$, therefore it acts by zero on the subspace $S^{\lambda^{(i+1)}} U$ of $\Hom_{Ind-\underline{Rep}^{ab}(S_{\nu})}(\mathbf{L}_i, V^{\underline{\otimes}  \nu})$.

So there must be a non-zero map $M_{i+1} \rightarrow SW_{\nu}(\mathbf{L}_i)$, and its image can be either $M_{i+1}$ itself or $L_{i+1}$. From the decomposition of $SW_{\nu}(\mathbf{L}_i)$ as a $\gl(U)$-module, we see that the image is $L_{i+1}$, and the induced map $L_{i+1} \rightarrow SW_{\nu}(\mathbf{L}_i)$ is an isomorphism.
\end{proof}

The following lemma will be useful to us later:

\begin{lemma}\label{lem:maps_L_0}
 $$\Hom_{Ind-\underline{Rep}^{ab}(S_{\nu})}(\mathbf{L}_0, SW^*_{\nu}(L_0)) = 0$$
\end{lemma}

\begin{proof}
Recall that by definition of the functor $ SW^*_{\nu}$,
 $$\Hom_{Ind-\underline{Rep}^{ab}(S_{\nu})}(\mathbf{L}_0, SW^*_{\nu}(L_0)) = \Hom_{Ind-(\underline{Rep}^{ab}(S_{\nu}) \boxtimes \co^{\mathfrak{p}}_{\nu, V})}(\mathbf{L}_0 \otimes L_0, V^{\underline{\otimes}  \nu})$$
 Recall also that the space $\mathfrak{u}_{\mathfrak{p}}^{-}$ acts on $\mathbf{L}_0 \otimes L_0$ by nilpotent operators, since $ L_0$ is finite-dimensional and $\bZ_+$-graded, and each non-zero element of $\mathfrak{u}_{\mathfrak{p}}^{-}$ acts by operator of degree $1$.
 Now let $$\phi \in \Hom_{Ind-(\underline{Rep}^{ab}(S_{\nu}) \boxtimes \co^{\mathfrak{p}}_{\nu, V})}(\mathbf{L}_0 \otimes L_0, V^{\underline{\otimes}  \nu})$$
 The map $\phi$ is zero iff $\phi\rvert_{\mathbf{L}_0 \otimes S^{\lambda^{(0)}} U}=0$. 
 
 \InnaA{Fix $k \in \bZ_+$ so that we have an inclusion of $\underline{Rep}^{ab}(S_{\nu})$-objects: $$\phi(\mathbf{L}_0 \otimes v) \subset (\Del_k \otimes U^{\otimes k} )^{S_k}$$ where $v$ is the highest weight vector in $S^{\lambda^{(0)}} U$. Then we automatically get an inclusion of $\underline{Rep}^{ab}(S_{\nu})$-objects} $$\phi(\mathbf{L}_0 \otimes S^{\lambda^{(0)}} U) \subset (\Del_k \otimes U^{\otimes k} )^{S_k}$$

 Let $l:=\abs{{\lambda^{(0)}}}$.
 Since $\mathbf{L}_0 = X_{\lambda^{(0)}}$, it is a summand of $\Delta_i$ iff $i \geq l$, so we immediately see that $k \geq l$.

 Now, $\mathbf{L}_0 \otimes S^{\lambda^{(0)}} U$ is a direct summand of $ \Del_l \otimes U^{\otimes l}$, so one can easily find $\psi:  \Del_l \otimes U^{\otimes l} \rightarrow \Del_k \otimes U^{\otimes k} $ such that $\phi(\mathbf{L}_0 \otimes S^{\lambda^{(0)}} U) = \Im(\psi)$.

 As we said before, \InnaA{$\mathfrak{u}_{\mathfrak{p}}^{-}$ acts on $\mathbf{L}_0 \otimes L_0$ by nilpotent operators, so} for any $u \in U \cong \mathfrak{u}_{\mathfrak{p}}^{-}$, $(F_u)^N \circ \phi = 0$ for $N>>0$.

 On the other hand, we know that $(F_u)^N \circ \psi \neq 0$ if $\psi \neq 0$ (by applying Lemma \ref{lem:tens_power_weak_torsion_free} iteratively to $\psi, F_u\circ \psi, ...,(F_u)^{N-1} \circ \psi$).

We conclude that
$$\Hom_{Ind-(\underline{Rep}^{ab}(S_{\nu}) \boxtimes \co^{\mathfrak{p}}_{\nu, V})}(\mathbf{L}_0 \otimes L_0, V^{\underline{\otimes}  \nu}) = 0$$ as needed.
\end{proof}

We now use Lemma \ref{lem:maps_L_0} to compute the images of the ``exceptional'' objects in our blocks $\mathcal{B}_{\lambda}$, $\mathfrak{B}_{\lambda}$ under the functors $SW_{\nu}, SW^*_{\nu}$, respectively.

\begin{lemma}\label{lem:images_exceptionals}
 \mbox{}
\begin{enumerate}[label=(\alph*)]
 \item $SW_{\nu}(\mathbf{L}_{0}=\mathbf{M}_{0} = X_{\lambda^{(0)}}) \cong M_0$.
\item $SW^*_{\nu}(L_0) = 0$.

\end{enumerate}
\end{lemma}
\begin{proof}
 \begin{enumerate}[leftmargin=*, label=(\alph*)]
  \item \InnaA{We will use an agrument similar to the one in the proof of Proposition \ref{prop:image_SW_simples}.} Recall that from Lemma \ref{lem:gl_u_decomp_images_SW}, we have the following isomorphism of \InnaA{$\bZ_+$-graded} $\gl(U)$-modules: $$SW_{\nu}(\mathbf{L}_{0}=\mathbf{M}_{0}=X_{\lambda^{(0)}}) \cong \bigoplus_{\mu \in \mathcal{I}^{+}_{\lambda^{(0)}}} S^{\mu} U$$ \InnaA{($ S^{\mu} U$ lies in degree $\abs{\mu}$).
 Recall also that $\mathfrak{u}_{\mathfrak{p}}^{+}$ acts on the right hand side by operators of degree $-1$.}
 
  This implies that there is a non-zero map of $\gl(V)$-modules $$\phi: M_0=M_{\mathfrak{p}}(\nu-\abs{\lambda^{(0)}}, \lambda^{(0)}) \rightarrow SW_{\nu}(\mathbf{L}_{0})$$ \InnaA{From the $\gl(U)$-decomposition of $M_0$ (c.f. Lemma \ref{lem:gl_u_struct_o_cat}), if this map $\phi$ is injective, then it is bijective as well.}

  So we only need to check that $\phi$ is injective. Indeed, assume $\phi$ is not injective. Since $\phi$ is not zero, $\phi$ must factor through $L_0$, so we have:
  $$\dim \Hom_{\co^{\mathfrak{p}}_{\nu}}(L_0, SW_{\nu}(\mathbf{L}_{0})) \geq 1$$
  But from adjointness of $SW_{\nu}, SW^*_{\nu}$, together with Lemma \ref{lem:maps_L_0}, we have:
  $$ \Hom_{\co^{\mathfrak{p}}_{\nu}}(L_0, SW_{\nu}(\mathbf{L}_{0})) \cong \Hom_{Ind-\underline{Rep}^{ab}(S_{\nu})}(\mathbf{L}_0, SW^*_{\nu}(L_0)) = 0$$

  We obtained a contradiction, which means that $\phi$ is injective.

  \item Recall that since $SW_{\nu}, SW^*_{\nu}$ are adjoint, for any Young diagram $\mu$ we have
  $$\Hom_{Ind-\underline{Rep}^{ab}(S_{\nu})}(\mathbf{L}({\mu}), SW^*_{\nu}(L_0)) \cong \Hom_{\co^{\mathfrak{p}}_{\nu}}(L_0, SW_{\nu}(\mathbf{L}({\mu})))$$
  The latter is zero by Propositions \ref{prop:image_SW_ss_case}, \ref{prop:image_SW_simples}, and Lemma \ref{lem:maps_L_0}. Hence $SW^*_{\nu}(L_0)=0$.
%
%

  \end{enumerate}

\end{proof}

We can now prove the following proposition:

\begin{proposition}\label{prop:images_SW}
 \mbox{}
\begin{enumerate}[label=(\alph*)]
 \item $SW_{\nu}(\mathbf{M}_i) \cong M_i$ whenever $i \geq 0$.
 \item $SW_{\nu}(\InnaB{\mathbf{M}^*_i}) \cong M^{\vee}_i$ whenever $i \geq 2$.
 \item $SW_{\nu}(\mathbf{P}_{i}) \cong P_{i+1}$ whenever $i \geq 0$, and $i < k_{\lambda} -1$ or $i \geq k_{\lambda}$ (recall that in the latter case $P_{i+1}=0$); $SW_{\nu}(\mathbf{P}_{k_{\lambda}-1}) \cong L_{k_{\lambda}-1}$.
 \item $SW_{\nu}(\InnaB{\mathbf{M}^*_1}) \cong Ker(P_1 \twoheadrightarrow L_1)$.
\end{enumerate}
\end{proposition}
\begin{proof}
\begin{enumerate}[leftmargin=*, label=(\alph*)]
 \item For $i=0$, we have already proved (in Lemma \ref{lem:images_exceptionals}) that $SW_{\nu}(\mathbf{M}_0) \cong M_0$. We now fix $i \in \bZ_{>0}$.

 Recall from Lemma \ref{lem:gl_u_decomp_images_SW} that for $i>0$, we have an isomorphism of $\gl(U)$-modules:
 $$SW_{\nu}(\mathbf{M}_{i})\rvert_{\gl(U)} \cong \bigoplus_{\mu \in  \mathcal{I}^{+}_{\lambda^{(i)}}} S^{\mu} U$$

 Similarly to the argument in the proofs of Propositions \ref{prop:image_SW_ss_case} and \ref{prop:image_SW_simples}, we have a $\bZ_+$-grading on the space $SW_{\nu}(\mathbf{M}_{i})$, which is inherited from the grading on $V^{\underline{\otimes}  \nu}$. Grade $j$ of $SW_{\nu}(\mathbf{M}_{i})$ is then $\bigoplus_{\mu \in  \mathcal{I}^{+}_{\lambda^{(i)}}, \abs{\mu}=j} S^{\mu} U$.
 The minimal grade is thus $\abs{\lambda^{(i)}}$, and it consists of the $\gl(U)$-module $S^{\lambda^{(i)}} U$.

 Recall that by definition of $V^{\underline{\otimes}  \nu}$, $\mathfrak{u}^{\mathfrak{p}}_+$ acts on the graded space $V^{\underline{\otimes}  \nu} \cong \bigoplus_{k \geq 0} (U^{\otimes k} \otimes \Delta_k)^{S_k}$ by operators of degree $-1$, therefore, it acts by zero on the subspace $S^{\lambda^{(i)}} U$ of $SW_{\nu}(\mathbf{M}_{i})=\Hom_{\underline{Rep}^{ab}(S_{\nu})}(\mathbf{M}_i, V^{\underline{\otimes}  \nu})$.

If $i \geq k_{\lambda} $, then $SW_{\nu}(\mathbf{M}_{i}) =0 =M_i$, and we are done.

Otherwise, $M_i \neq 0$, and there must be a non-zero map $M_{i} \rightarrow SW_{\nu}(\mathbf{M}_i)$; its image can be either $M_{i}$ itself or $L_{i}$. In the first case, we get an isomorphism $M_i \cong SW_{\nu}(\mathbf{M}_i)$ (from the $\gl(U)$-decomposition of both), and again, we are done.

We will now assume that we are in the second case, and there is a non-zero map $L_i \rightarrow SW_{\nu}(\mathbf{M}_i)$. Then the $\gl(U)$-decomposition of the quotient $\quotient{SW_{\nu}(\mathbf{M}_i)}{L_i}$ means that this module is congruent to $L_{i+1}$. Notice that at this point we can assume that $i < k_{\lambda}-1 $ (otherwise $L_{i+1} =0$, so $M_i =L_i\cong SW_{\nu}(\mathbf{M}_i)$).
We will prove that under this assumption, we arrive to a contradiction.

Now, consider the short exact sequence $$0 \rightarrow \mathbf{L}_{i-1} \stackrel{\phi}{\longrightarrow} \mathbf{M}_i  \stackrel{\psi}{\longrightarrow} \mathbf{L}_i  \rightarrow 0$$
Since the contravariant functor $SW_{\nu}$ is left-exact, we have (using Proposition \ref{prop:image_SW_simples} and Lemma \ref{lem:images_exceptionals}, part (a))
$$0 \rightarrow SW_{\nu}(\mathbf{L}_{i})\cong L_{i+1} \stackrel{SW_{\nu}(\psi)}{\longrightarrow} SW_{\nu}(\mathbf{M}_i)  \stackrel{SW_{\nu}(\phi)}{\longrightarrow} SW_{\nu}(\mathbf{L}_{i-1})\cong L_i $$

 So $SW_{\nu}(\psi)$ is an insertion of a direct summand, $SW_{\nu}(\phi)$ is a projection onto a direct summand, and we get: $$SW_{\nu}(\mathbf{M}_i) \cong L_i \oplus L_{i+1}$$

 We now use the unit natural transformation $\epsilon$ described in Notation \ref{notn:eta_eps_units}. We have a commutative diagram:
$$\begin{CD}
 SW^*_{\nu}(SW_{\nu}(\mathbf{L}_{i-1})) @>SW^*_{\nu}(SW_{\nu}(\phi))>> SW^*_{\nu}(SW_{\nu}(\mathbf{M}_{i})) @>SW^*_{\nu}(SW_{\nu}(\psi))>> SW^*_{\nu}(SW_{\nu}(\mathbf{L}_{i})) \\
@A\epsilon_{\mathbf{L}_{i-1}}AA @A\epsilon_{\mathbf{M}_{i}}AA @A\epsilon_{\mathbf{L}_{i}}AA\\
 \mathbf{L}_{i-1} @>\phi\phantom{\text{.........................}}>> \mathbf{M}_i  @>\psi\phantom{\text{.........................}}>> \mathbf{L}_i                                                                                                                                                                                                          
\end{CD}$$
 which can be rewritten as
 $$\begin{CD}
 SW^*_{\nu}(L_{i}) @>SW^*_{\nu}(SW_{\nu}(\phi))>>SW^*_{\nu}(L_{i}) \oplus SW^*_{\nu}(L_{i+1}) @>SW^*_{\nu}(SW_{\nu}(\psi))>>  SW^*_{\nu}(L_{i+1}) \\
@A\epsilon_{\mathbf{L}_{i-1}}AA @A\epsilon_{\mathbf{M}_{i}}AA @A\epsilon_{\mathbf{L}_{i}}AA\\
 \mathbf{L}_{i-1} @>\phi\phantom{\text{.........................}}>> \mathbf{M}_i  @>\psi\phantom{\text{.........................}}>> \mathbf{L}_i                                                                                                                                                                                                          
\end{CD}$$
Since the contravariant functor $SW^*_{\nu}$ is additive, $SW^*_{\nu}(SW_{\nu}(\phi))$ is an insertion of a direct summand, and $SW^*_{\nu}(SW_{\nu}(\psi))$ is a projection onto a direct summand.

 Now, the relations in Lemma \ref{lem:cond_on_units} imply that $\epsilon_{\mathbf{L}_{i-1}},\epsilon_{\mathbf{M}_{i}}, \epsilon_{\mathbf{L}_{i}}$ are all non-zero as long as $SW_{\nu}(\mathbf{L}_{i-1}), SW_{\nu}(\mathbf{M}_{i}), SW_{\nu}(\mathbf{L}_{i}) $ are non-zero, which is guaranteed by the assumption $i < k_{\lambda}-1 $.

 This means that the image of $\mathbf{L}_{i-1}$ under $SW^*_{\nu}(SW_{\nu}(\phi)) \circ \epsilon_{\mathbf{L}_{i-1}}$ is $\mathbf{L}_{i-1}$, and it lies inside the direct summand $SW^*_{\nu}(L_{i})$ of $SW^*_{\nu}(SW_{\nu}(\mathbf{M}_{i}))$.
We then deduce that $\epsilon_{\mathbf{M}_{i}}$ is injective (since it is not zero on the socle \InnaA{$\mathbf{L}_{i-1}$} of $\mathbf{M}_{i}$), and that its image lies entirely inside the direct summand $SW^*_{\nu}(L_{i})$ of $SW^*_{\nu}(SW_{\nu}(\mathbf{M}_{i}))$ (since $\mathbf{M}_{i}$ is indecomposable).

But this clearly contradicts the right half of the above commutative diagram, since it means that $$SW^*_{\nu}(SW_{\nu}(\psi)) \circ \epsilon_{\mathbf{M}_{i}} =0$$ while we have already established that $\epsilon_{\mathbf{L}_{i}}\circ \psi \neq 0$.

\item The proof for $SW_{\nu}(\InnaB{\mathbf{M}^*_k})$ is very similar to the one given for $SW_{\nu}(\mathbf{M}_k)$.
 Now fix $i \geq 2$.

 Consider the short exact sequence $$0 \rightarrow \mathbf{L}_{i} \stackrel{\phi}{\longrightarrow} \InnaB{\mathbf{M}^*_i}  \stackrel{\psi}{\longrightarrow} \mathbf{L}_{i-1}  \rightarrow 0$$
Since the contravariant functor $SW_{\nu}$ is left-exact, we have (using Proposition \ref{prop:image_SW_simples})
$$0 \rightarrow SW_{\nu}(\mathbf{L}_{i-1})\cong L_{i} \stackrel{SW_{\nu}(\psi)}{\longrightarrow} SW_{\nu}(\InnaB{\mathbf{M}^*_i})  \stackrel{SW_{\nu}(\phi)}{\longrightarrow} SW_{\nu}(\mathbf{L}_{i})\cong L_{i+1}$$

Furthermore, Lemma \ref{lem:gl_u_decomp_images_SW} tells us that for $i>0$, we have an isomorphism of $\gl(U)$-modules:
 $$SW_{\nu}(\InnaB{\mathbf{M}^*_{i}})\rvert_{\gl(U)} \cong \bigoplus_{\mu \in  \mathcal{I}^{+}_{\lambda^{(i)}}} S^{\mu} U$$
This decomposition, together with the $\gl(U)$-decomposition of $L_i, L_{i+1}$, tell us that the above exact sequence can be completed to a short exact sequence
$$0 \rightarrow SW_{\nu}(\mathbf{L}_{i-1})\cong L_{i} \stackrel{SW_{\nu}(\psi)}{\longrightarrow} SW_{\nu}(\InnaB{\mathbf{M}^*_i})  \stackrel{SW_{\nu}(\phi)}{\longrightarrow} SW_{\nu}(\mathbf{L}_{i})\cong L_{i+1} \rightarrow 0$$

Applying the (exact) functor $(\cdot)^{\vee}: \co^{\mathfrak{p}}_{\nu} \rightarrow \left(\co^{\mathfrak{p}}_{\nu} \right)^{op}$ to the above exact sequence, we conclude that $SW_{\nu}(\InnaB{\mathbf{M}^*_i})^{\vee}$ is isomorphic to either $M_i$ or $L_i \oplus L_{i+1}$. This implies that $SW_{\nu}(\InnaB{\mathbf{M}^*_i})$ is isomorphic to either $M^{\vee}_i$ (which is what we want to show) or $L_i \oplus L_{i+1}$.

We will now assume that we are in the case $SW_{\nu}(\InnaB{\mathbf{M}^*_i}) \cong L_i \oplus L_{i+1}$. Furthermore, we will assume that $i < k_{\lambda}-1 $ (otherwise $L_{i+1} =0$, so $M^{\vee}_i =L_i \cong SW_{\nu}(\InnaB{\mathbf{M}^*_i})$).
We will prove that under this assumption, we arrive to a contradiction. Since we assumed that $i < k_{\lambda}-1 $, we have: $L_i, L_{i+1} \neq 0$, which means that $SW_{\nu}(\psi), SW_{\nu}(\phi) \neq 0$ are insertion of and projection onto direct summands, respectively.

We now construct the commutative diagram
$$\begin{CD}
 SW^*_{\nu}(SW_{\nu}(\mathbf{L}_{i})) @>SW^*_{\nu}(SW_{\nu}(\phi))>> SW^*_{\nu}(SW_{\nu}(\InnaB{\mathbf{M}^*_{i}})) @>SW^*_{\nu}(SW_{\nu}(\psi))>> SW^*_{\nu}(SW_{\nu}(\mathbf{L}_{i-1})) \\
@A\epsilon_{\mathbf{L}_{i}}AA @A\epsilon_{\InnaB{\mathbf{M}^*_{i}}}AA @A\epsilon_{\mathbf{L}_{i-1}}AA\\
 \mathbf{L}_i @>\phi\phantom{\text{.........................}}>> \InnaB{\mathbf{M}^*_i}  @>\psi\phantom{\text{.........................}}>> \mathbf{L}_{i-1}                                                                                                                                                                                                          
\end{CD}$$

which can be rewritten as
$$\begin{CD}
SW^*_{\nu}(L_{i+1}) @>SW^*_{\nu}(SW_{\nu}(\phi))>> SW^*_{\nu}(L_{i+1})\oplus SW^*_{\nu}(L_{i}) @>SW^*_{\nu}(SW_{\nu}(\psi))>> SW^*_{\nu}(L_{i}) \\
@A\epsilon_{\mathbf{L}_{i}}AA @A\epsilon_{\InnaB{\mathbf{M}^*_{i}}}AA @A\epsilon_{\mathbf{L}_{i-1}}AA\\
 \mathbf{L}_i @>\phi\phantom{\text{.........................}}>> \InnaB{\mathbf{M}^*_i}  @>\psi\phantom{\text{.........................}}>> \mathbf{L}_{i-1}                                                                                                                                                                                                          
\end{CD}$$
 Exactly the same arguments as in part (a) now apply (we use the fact that $\InnaB{\mathbf{M}^*_{i}}$ is indecomposable), and we get a contradiction.

 \item Let $i\geq 0$. Consider the exact sequence $$0 \rightarrow \mathbf{M}_{i+1} \stackrel{\phi}{\longrightarrow} \mathbf{P}_i  \stackrel{\psi}{\longrightarrow} \mathbf{M}_{i}  \rightarrow 0$$

 Since the contravariant functor $SW_{\nu}$ is left-exact, we get an exact sequence
 $$ 0 \rightarrow SW_{\nu}(\mathbf{M}_i)  \stackrel{SW_{\nu}(\psi)}{\longrightarrow} SW_{\nu}(\mathbf{P}_i) \stackrel{SW_{\nu}(\phi)}{\longrightarrow} SW_{\nu}(\mathbf{M}_{i+1})$$ 
 and in particular (see part (a)): $M_i \cong SW_{\nu}(\mathbf{M}_i) \hookrightarrow SW_{\nu}(\mathbf{P}_i)$. 
 
 If $i \geq k_{\lambda} -1$, then \InnaA{part (a) tells us that $SW_{\nu}(\mathbf{M}_{i+1}) =M_{i+1} =0$}. We conclude that $M_i \cong SW_{\nu}(\mathbf{M}_i)  \cong SW_{\nu}(\mathbf{P}_i)$. In particular, $SW_{\nu}(\mathbf{P}_i) = 0 $ if $i \geq k_{\lambda} $, and  $$SW_{\nu}(\mathbf{M}_{k_{\lambda}-1})  \cong SW_{\nu}(\mathbf{P}_{k_{\lambda}-1})\cong M_{k_{\lambda}-1}  \cong L_{k_{\lambda}-1} $$

 From now on, we will assume that $i < k_{\lambda} -1$, and thus $M_i, M_{i+1} \neq 0$.
 
 Now, $P_{i+1}$ is the injective hull of $M_i$, so there is a map $f:SW_{\nu}(\mathbf{P}_i) \rightarrow P_{i+1}$ such that the following diagram is commutative:

 $$\begin{CD}
    0 @>>> SW_{\nu}(\mathbf{M}_i) @>SW_{\nu}(\psi)>> SW_{\nu}(\mathbf{P}_i) \\
    @. @VVV @VfVV\\
    0 @>>> M_i @>>> P_{i+1}
   \end{CD}$$

From the $\gl(U)$-decomposition of $ SW_{\nu}(\mathbf{M}_i), SW_{\nu}(\mathbf{P}_i),SW_{\nu}(\mathbf{M}_{i+1})$ (see Lemma \ref{lem:gl_u_decomp_images_SW}), we see that the map \InnaA{$SW_{\nu}(\phi)$} is surjective. This means that there is a non-zero map $\bar{f}: SW_{\nu}(\mathbf{M}_{i+1})\rightarrow M_{i+1}$ so the diagram below is commutative:
$$\begin{CD}
    0 @>>> SW_{\nu}(\mathbf{M}_i) @>SW_{\nu}(\psi)>> SW_{\nu}(\mathbf{P}_i) @>SW_{\nu}(\phi)>> SW_{\nu}(\mathbf{M}_{i+1}) @>>>0\\
    @. @VVV @VfVV @V{\bar{f}}VV @.\\
    0 @>>> M_i @>>> P_{i+1} @>>> M_{i+1} @>>>0
   \end{CD}$$

 Since $SW_{\nu}(\mathbf{M}_{i+1}) \cong M_{i+1}$, we see that $\bar{f}$ is either an isomorphism, or zero. In the former case, $f$ is an isomorphism as well, and we are done. 

 So it remains to prove that $\bar{f} \neq 0$.

 Assume $\bar{f}=0$. This means that the image of $f$ is $M_i \subset P_{i+1}$, and thus $SW_{\nu}(\mathbf{P}_i) = M_i \oplus M_{i+1}$, with the maps $SW_{\nu}(\psi), SW_{\nu}(\phi)$ being an insertion of a direct summand and a projection onto a direct summand, respectively.

 We now construct the commutative diagram
 $$\begin{CD}
 SW^*_{\nu}(SW_{\nu}(\mathbf{M}_{i+1})) @>SW^*_{\nu}(SW_{\nu}(\phi))>> SW^*_{\nu}(SW_{\nu}(\mathbf{P}_{i})) @>SW^*_{\nu}(SW_{\nu}(\psi))>> SW^*_{\nu}(SW_{\nu}(\mathbf{M}_{i})) \\
@A\epsilon_{\mathbf{M}_{i+1}}AA @A\epsilon_{\mathbf{P}_{i}}AA @A\epsilon_{\mathbf{M}_{i}}AA\\
 \mathbf{M}_{i+1} @>\phi\phantom{\text{.........................}}>> \mathbf{P}_i  @>\psi\phantom{\text{.........................}}>> \mathbf{M}_{i}                                                                                                                                                                                                          
\end{CD}$$
which can be rewritten as
 $$\begin{CD}
SW^*_{\nu}(M_{i+1}) @>SW^*_{\nu}(SW_{\nu}(\phi))>> SW^*_{\nu}(M_{i+1})\oplus SW^*_{\nu}(M_{i}) @>SW^*_{\nu}(SW_{\nu}(\psi))>> SW^*_{\nu}(M_{i}) \\
@A\epsilon_{\mathbf{M}_{i+1}}AA @A\epsilon_{\mathbf{P}_{i}}AA @A\epsilon_{\mathbf{M}_{i}}AA\\
 \mathbf{M}_{i+1} @>\phi\phantom{\text{.........................}}>> \mathbf{P}_i  @>\psi\phantom{\text{.........................}}>> \mathbf{M}_{i}                                                                                                                                                                                                          
\end{CD}$$
The same type of argument as in part (a) now applies (we use the fact that $\mathbf{P}_{i}$ is indecomposable), and we get a contradiction.

\item Consider the short exact sequences
$$0 \rightarrow \mathbf{L}_{1} {\longrightarrow} \InnaB{\mathbf{M}^*_1}  {\longrightarrow} \mathbf{L}_{0}  \rightarrow 0$$
$$0 \rightarrow \mathbf{L}_{0} {\longrightarrow} \mathbf{P}_0  {\longrightarrow} \InnaB{\mathbf{M}^*_{1}}  \rightarrow 0$$

The contravariant functor $SW_{\nu}$ is left-exact, so we have (using part (a), Lemma \ref{lem:images_exceptionals}(a) and Proposition \ref{prop:image_SW_simples})
$$0 \rightarrow SW_{\nu}(\mathbf{L}_{0}) \cong M_0 {\longrightarrow} SW_{\nu}(\InnaB{\mathbf{M}^*_1}) {\longrightarrow} SW_{\nu}(\mathbf{L}_{1})\cong L_2$$
$$0 \rightarrow SW_{\nu}(\InnaB{\mathbf{M}^*_{1}}) {\longrightarrow} SW_{\nu}(\mathbf{P}_0) \cong P_1 {\longrightarrow} SW_{\nu}(\mathbf{L}_{0})\cong M_0$$

The first of these two exact sequences implies that $[SW_{\nu}(\InnaB{\mathbf{M}^*_{1}}): L_1] =1$, hence the map $SW_{\nu}(\InnaB{\mathbf{M}^*_{1}}) \rightarrow P_1$ in the second sequence is not an isomorphism. The second one then means that $SW_{\nu}(\InnaB{\mathbf{M}^*_{1}})$ is the kernel of the unique non-zero map $P_1 \rightarrow M_0$, which factors through the canonical map $P_1 \rightarrow L_1$. Thus $SW_{\nu}(\InnaB{\mathbf{M}^*_1}) \cong Ker(P_1 \twoheadrightarrow L_1)$.

\end{enumerate}
\end{proof}

\begin{proposition}\label{prop:SW_ess_surj}
The contravariant functor $\widehat{SW}_{\nu}: \mathcal{B}_{\lambda} \rightarrow \hat{\pi}(\mathfrak{B}_{\lambda})$ is essentially surjective.

\end{proposition}
\begin{proof}
We first prove a Sublemma:
\begin{sublemma}\label{sublem:SW_full_ess_surj_proj}
\mbox{}
\begin{enumerate}[label=(\alph*)]
 \item Let $I$ be an injective object in $\hat{\pi}(\mathfrak{B}_{\lambda})$. Then there exists a projective object $P$ in $\mathcal{B}_{\lambda}$ such that $\widehat{SW}_{\nu}(P) = I$.
 \item Consider the restriction of $\widehat{SW}_{\nu}$ to the full subcategory of $\mathcal{B}_{\lambda}$ consisting of projective objects. This restriction is a full contravariant functor to $\hat{\pi}(\mathfrak{B}_{\lambda})$.
\end{enumerate}

\end{sublemma}
\begin{proof}

Recall from Corollary \ref{cor:hat_O_p_enough_inj_proj} that the set of isomorphism classes of indecomposable injective objects in $\hat{\pi}(\mathfrak{B}_{\lambda})$ is $\{\hat{\pi}(P_i)\}_{\InnaA{0 < i <k_{\lambda}}}$. The set of isomorphism classes of indecomposable projective objects in $\mathcal{B}_{\lambda}$ is $\{\mathbf{P}_{i}\}_{i \geq 0}$ (c.f. Section \ref{ssec:S_nu_abelian_env}).

We know from Proposition \ref{prop:images_SW} that $\hat{\pi}(P_i) \cong \widehat{SW}_{\nu}(\mathbf{P}_{i-1})$ for any \InnaA{$0 < i < k_{\lambda}$}. This immediately implies the first part of the sublemma.


We now consider the restriction of $\widehat{SW}_{\nu}$ to the full subcategory of $\mathcal{B}_{\lambda}$ consisting of projective objects.

To see that this restriction is full, we need to check that for any $i, j \geq 0$, the map 
\begin{align}\label{eq:hat_SW_surj_on_proj}
\widehat{SW}_{\nu, \mathbf{P}_j, \mathbf{P}_i}: \Hom_{\Dab}(\mathbf{P}_j, \mathbf{P}_i) \rightarrow \Hom_{\widehat{\co}^{\mathfrak{p}}_{\nu}}(\widehat{SW}_{\nu}(\mathbf{P}_i), \widehat{SW}_{\nu}(\mathbf{P}_j)) 
\end{align}
is surjective.

We use the following observation (which follows from the definition of the Serre quotient):
\begin{obsr}
 Let $E, E' \in \co^{\mathfrak{p}}_{\nu}$. Assume $E $ has no finite-dimensional quotients and $E'$ has no finite-dimensional submodules. Then $$\Hom_{\widehat{\co}^{\mathfrak{p}}_{\nu}}(\hat{\pi}(E), \hat{\pi}(E')) = \Hom_{\co^{\mathfrak{p}}_{\nu}}(E, E')$$ In particular, this is true for $E, E'$ being $P_i, M_i, M^{\vee}_i, L_i$ ($ i\geq 1$).
\end{obsr}

Recall from Theorem \ref{thrm:blocks_S_nu}, Propositions \ref{prop:obj_ab_env} and \ref{prop:images_SW} that if $\abs{i -j} >1$, or if $i \geq k_{\lambda}$, or if $j \geq k_{\lambda}$, then the right hand side $\Hom$-space in \eqref{eq:hat_SW_surj_on_proj} is zero and there is nothing to prove.

If either $i = k_{\lambda} -1$ or $j =k_{\lambda} -1$, we only need to check the cases 

$(i, j) = (k_{\lambda} -1, k_{\lambda} -2), (k_{\lambda} -2, k_{\lambda} -1), (k_{\lambda} -1, k_{\lambda} -1)$.

In all three cases $\Hom_{\widehat{\co}^{\mathfrak{p}}_{\nu}}(\widehat{SW}_{\nu}(\mathbf{P}_i), \widehat{SW}_{\nu}(\mathbf{P}_j))$ is one-dimensional, so we only need to check that the above map $\widehat{SW}_{\nu, \mathbf{P}_j, \mathbf{P}_i}$ is not zero. 
The case $i = j = k_{\lambda} -1$ is obvious. Since $\widehat{SW}_{\nu}$ is contravariant and exact, the exact sequences $$ 0 \rightarrow \InnaB{\mathbf{M}^*_{ k_{\lambda} -2}} \rightarrow \mathbf{P}_{ k_{\lambda} -2} \rightarrow \mathbf{P}_{ k_{\lambda} -1}$$
and $$ \mathbf{P}_{ k_{\lambda} -1} \rightarrow \mathbf{P}_{ k_{\lambda} -2} \rightarrow \mathbf{M}_{ k_{\lambda} -2} \rightarrow 0$$
become $$ \hat{\pi}(L_{ k_{\lambda} -1}) \rightarrow \hat{\pi}(P_{ k_{\lambda} -1}) \rightarrow \hat{\pi}(M^{\vee}_{ k_{\lambda} -2}) \rightarrow 0$$
and $$0\rightarrow \hat{\pi}(M_{ k_{\lambda} -2}) \rightarrow \hat{\pi}(P_{ k_{\lambda} -1}) \rightarrow \hat{\pi}(L_{ k_{\lambda} -1})$$
which proves \InnaA{that $\widehat{SW}_{\nu, \mathbf{P}_j, \mathbf{P}_i}$ is not zero if $(i, j) = (k_{\lambda} -1, k_{\lambda} -2)$ or $(i,j)= (k_{\lambda} -2, k_{\lambda} -1)$}.
$$ $$

We can now assume that  $i, j < k_{\lambda} -1$, and thus $\widehat{SW}_{\nu}(\mathbf{P}_{i}) \cong \hat{\pi}(P_{i+1}), \widehat{SW}_{\nu}(\mathbf{P}_{j}) \cong \hat{\pi}(P_{j+1})$.

If $\abs{i-j} =1$, then both $\Hom$-spaces are at most one-dimensional and we only need to check that the above map $\widehat{SW}_{\nu, \mathbf{P}_j, \mathbf{P}_i}$ is not zero.
Assume $j = i+1$ (the case $j=i-1$ is proved in a similar way).
Let $\InnaB{\beta_{i+1}}: \mathbf{P}_i \rightarrow \mathbf{P}_{i+1}$ be a non-zero morphism.

Then the kernel of $\InnaB{\beta_{i+1}}$ is $\InnaB{\mathbf{M}^*_{i}}$, and since $\widehat{SW}_{\nu}$ is contravariant and exact, we get: $$Coker(\widehat{SW}_{\nu}(\InnaB{\beta_{i+1}})) \cong \widehat{SW}_{\nu}(\InnaB{\mathbf{M}^*_{i}}) \cong \hat{\pi}(M^{\vee}_{i}) \not\cong \hat{\pi}(P_{i+1})$$ which means that
$\widehat{SW}_{\nu}(\InnaB{\beta_{i+1}}) \neq 0$. Similarly, given a non-zero morphism $\InnaB{\alpha_{i+1}}: \mathbf{P}_{i+1} \rightarrow \mathbf{P}_{i}$, we have: $$Ker(\widehat{SW}_{\nu}(\InnaB{\alpha_{i+1}})) \cong \widehat{SW}_{\nu}(Coker(\InnaB{\alpha_{i+1}})) \cong \widehat{SW}_{\nu}(\mathbf{M}_{i}) \cong \hat{\pi}(M_{i}) \not\cong \hat{\pi}(P_{i+1})$$ which means that
$\widehat{SW}_{\nu}(\InnaB{\alpha_{i+1}}) \neq 0$.

Finally, if $i=j$, then the space $\End_{\Dab}(\mathbf{P}_i)$ is spanned by endomorphisms \InnaA{$\id_{\mathbf{P}_i}, \InnaB{\gamma_i}$} of $\mathbf{P}_i$, where $Im(\InnaB{\gamma_i}) \cong \mathbf{L}_i$ ($\InnaB{\gamma_i} := \InnaB{\alpha_{i+1}} \circ \InnaB{\beta_{i+1}}$ \InnaA{in the above notation}). 

\InnaA{Since $\widehat{SW}_{\nu}$ is contravariant and exact, and (by assumption) $i< k_{\lambda} -1$, we see that $\widehat{SW}_{\nu}(\InnaB{\gamma_i})$ will be a non-zero endomorphism of $\hat{\pi}(P_{i+1})$ factoring though $\hat{\pi}(L_{i+1})$. This means that $\widehat{SW}_{\nu}(\id_{\mathbf{P}_i}), \widehat{SW}_{\nu}(\InnaB{\gamma_i})$ span $\End_{\widehat{\co}^{\mathfrak{p}}_{\nu}}(\hat{\pi}(P_{i+1}))$.

This proves that for any $i, j \geq 0$, the map in \eqref{eq:hat_SW_surj_on_proj} is surjective, and we are done.}
\end{proof}

We now show that $\widehat{SW}_{\nu}$ is essentially surjective. Indeed, let $E \in \hat{\pi}(\mathfrak{B}_{\lambda})$. Then $E$ has an injective resolution $$0 \rightarrow E \rightarrow I^0 \stackrel{f}{\rightarrow} I^1$$ ($I^0, I^1$ are injective objects in $\hat{\pi}(\mathfrak{B}_{\lambda})$).
From the Sublemma \ref{sublem:SW_full_ess_surj_proj} above, we know that there exist projective objects $P^0, P^1 \in
\mathcal{B}_{\lambda}$ and a morphism $g: P^1 \rightarrow P^0$ such that $$\widehat{SW}_{\nu}(P^0) = I^0, \widehat{SW}_{\nu}(P^1) =I^1, \widehat{SW}_{\nu}(g)=f$$ Then $E \cong Ker(f) \cong \widehat{SW}_{\nu}(Coker (g))$ (since $\widehat{SW}_{\nu}$ is exact). Thus $\widehat{SW}_{\nu}$ is essentially surjective.
\end{proof}
\begin{remark}
 The functor $\widehat{SW}_{\nu}: \mathcal{B}_{\lambda} \rightarrow \hat{\pi}(\mathfrak{B}_{\lambda})$ is not full. For example, consider $$ \Hom_{\Dab}(\mathbf{P}_{k_{\lambda}-1}, \mathbf{L}_{k_{\lambda}-2}) \stackrel{\widehat{SW}_{\nu}}{\longrightarrow}  \Hom_{\widehat{\co}^{\mathfrak{p}}_{\nu}}(\widehat{SW}_{\nu}(\mathbf{L}_{k_{\lambda}-2}),\widehat{SW}_{\nu}( \mathbf{P}_{k_{\lambda}-1})) = \End_{\widehat{\co}^{\mathfrak{p}}_{\nu}}(\hat{\pi}(L_{k_{\lambda}-1}))$$
 The $\Hom$-space in the left hand side is clearly zero, while the $\Hom$-space in the right hand side is one-dimensional.
\end{remark}

We now consider the Serre subcategory $Ker(\widehat{SW}_{\nu} \rvert_{{\mathcal{B}}_{\lambda}})$ of ${\mathcal{B}}_{\lambda}$ (this is a Serre subcategory since $\widehat{SW}_{\nu} $ is exact). This subcategory is the Serre subcategory of ${\mathcal{B}}_{\lambda}$ generated by the simple objects $\mathbf{L}_i$, $i \geq k_{\lambda} -1$.

We define the quotient of ${\mathcal{B}}_{\lambda}$ by $Ker(\widehat{SW}_{\nu} \rvert_{{\mathcal{B}}_{\lambda}})$:
$$\overline{\pi}: {\mathcal{B}}_{\lambda} \longrightarrow \overline{\pi}(\mathcal{B}_{\lambda})$$ 

By definition of $Ker(\widehat{SW}_{\nu} \rvert_{{\mathcal{B}}_{\lambda}})$, the functor $\widehat{SW}_{\nu}$ factors through $\overline{\pi}$ and we get an exact contravariant functor 
\InnaA{
$$ \widehat{\overline{SW}}_{\nu}: \overline{\pi}(\mathcal{B}_{\lambda}) \longrightarrow  \hat{\pi}(\mathfrak{B}_{\lambda})$$
such that 
$$ \xymatrix{  &{\mathcal{B}}_{\lambda}^{op} \ar[rr]^{SW_{\nu}}  \ar[rrdd]^{\widehat{SW}_{\nu}} \ar[dd]_{\overline{\pi}^{op}} &{} &\mathfrak{B}_{\lambda} \ar[dd]^{\widehat{\pi}} \\
&{} &{}\\
    &\overline{\pi}(\mathcal{B}_{\lambda})^{op} \ar[rr]^{ \widehat{\overline{SW}}_{\nu}} &{} &\hat{\pi}(\mathfrak{B}_{\lambda})
   }
$$}
Notice that all the functors in this commutative diagram except ${SW_{\nu}}$ are exact.

We now prove some properties of the functor $\overline{\pi}$ and the category $\overline{\pi}(\mathcal{B}_{\lambda})$.

\begin{lemma}\label{lem:overline_Del_enough_inj_proj}
\mbox{}
\begin{enumerate}[label=(\alph*)]
 \item The objects $\overline{\pi}(\mathbf{P}_i)$ are indecomposable injective (and projective) objects in $\overline{\pi}(\mathcal{B}_{\lambda})$ for any $i \leq k_{\lambda} -2$. 
 \item The category $\overline{\pi}(\mathcal{B}_{\lambda})$ has enough injectives and enough projectives.
 \item Moreover, $\{\overline{\pi}(\mathbf{P}_i) \}_{0 \leq i \leq k_{\lambda} -2}$ is the full set of representatives of isomorphism classes of indecomposable injective (respectively, projective) objects in $\overline{\pi}(\mathcal{B}_{\lambda})$.
\end{enumerate}

\end{lemma}
\begin{proof}
 To prove the first statement, we use Lemma \ref{lem:proj_in_serre_quotient} and the information on the structure of $\mathbf{P}_i$ given in Proposition \ref{prop:obj_ab_env}. 
 
 The proof of the last two statements mimics the proof of Corollary \ref{cor:hat_O_p_enough_inj_proj}. 
 
 All we need to show is that the object $\overline{\pi}(\mathbf{P}_{k_{\lambda} -1})$ is neither injective nor projective in $\overline{\pi}(\mathcal{B}_{\lambda})$, but has a projective cover and an injective hull in $\overline{\pi}(\mathcal{B}_{\lambda})$, both being direct sums of objects $\overline{\pi}(\mathbf{P}_i), i \leq k_{\lambda} -2$.
 $$ $$
 
 But $\overline{\pi}(\mathbf{P}_{k_{\lambda} -1})\cong \overline{\pi}(\mathbf{L}_{k_{\lambda} -2})$ (c.f. Proposition \ref{prop:obj_ab_env}), and we have a surjective map $\overline{\pi}(\mathbf{P}_{k_{\lambda} -2})\twoheadrightarrow \overline{\pi}(\mathbf{L}_{k_{\lambda} -2})$ and an injective map $ \overline{\pi}(\mathbf{L}_{k_{\lambda} -2})\hookrightarrow \overline{\pi}(\mathbf{P}_{k_{\lambda} -2})$. Since $\overline{\pi}(\mathbf{P}_{k_{\lambda} -2})$ is an indecomposable injective and projective object in $\overline{\pi}(\mathcal{B}_{\lambda})$, we conclude that $\overline{\pi}(\mathbf{P}_{k_{\lambda} -1})$ is neither injective nor projective in $\overline{\pi}(\mathcal{B}_{\lambda})$, but $\overline{\pi}(\mathbf{P}_{k_{\lambda} -2})$ is the projective cover and the injective hull of $\overline{\pi}(\mathbf{P}_{k_{\lambda} -1})$ in $\overline{\pi}(\mathcal{B}_{\lambda})$.
 \end{proof}

\begin{theorem}\label{thrm:hat_overline_SW_equiv}
 The functor $\widehat{\overline{SW}}_{\nu}: \overline{\pi}(\mathcal{B}_{\lambda}) \rightarrow \hat{\pi}(\mathfrak{B}_{\lambda})$ is an anti-equivalence of abelian categories. That is, $\widehat{\overline{SW}}_{\nu}: \overline{\pi}(\mathcal{B}_{\lambda}) \rightarrow \hat{\pi}(\mathfrak{B}_{\lambda})$ is an essentially surjective, fully faithful, exact contravariant functor.
\end{theorem}
\begin{proof}
\mbox{}
\begin{itemize}[leftmargin=*]
 \item Proof that $\widehat{\overline{SW}}_{\nu}$ is faithful: by definition, if $\widehat{\overline{SW}}_{\nu}(X)=0$ for some $X \in \overline{\pi}(\mathcal{B}_{\lambda})$, then $X =0$. Now, let $f: X \rightarrow Y$ in $\overline{\pi}(\mathcal{B}_{\lambda})$, and assume $\widehat{\overline{SW}}_{\nu}(f)=0$. Then $\widehat{\overline{SW}}_{\nu}(Im(f))=0$, i.e. $Im(f)=0$, and thus $f=0$.
 \item The fact that $\widehat{\overline{SW}}_{\nu}$ is essentially surjective follows directly from the fact that $\widehat{SW}_{\nu}$ is essentially surjective, c.f. Proposition \ref{prop:SW_ess_surj}.
 \item Proof that $\widehat{\overline{SW}}_{\nu}$ is full:
 
 We start with the following sublemma:
 \begin{sublemma}\label{sublem:overline_SW_equiv_proj}
  Let $Proj_{\overline{\pi}(\mathcal{B}_{\lambda})}$ be the full subcategory of projective objects in $\overline{\pi}(\mathcal{B}_{\lambda})$, and $Inj_{\hat{\pi}(\mathfrak{B}_{\lambda})}$ be the full subcategory of injective objects in $\hat{\pi}(\mathfrak{B}_{\lambda})$. Then $\widehat{\overline{SW}}_{\nu}$ induces an anti-equivalence of additive categories $Proj_{\overline{\pi}(\mathcal{B}_{\lambda})} \rightarrow Inj_{\hat{\pi}(\mathfrak{B}_{\lambda})}$.
 \end{sublemma}
\begin{proof}
The first thing we need to check is that given a projective object in $\overline{\pi}(\mathcal{B}_{\lambda})$, $\widehat{\overline{SW}}_{\nu}$ takes it to an injective object in $\hat{\pi}(\mathfrak{B}_{\lambda})$. By Lemma \ref{lem:overline_Del_enough_inj_proj}, it is enough to check this for $\overline{\pi}(\mathbf{P}_i)$ for $i \leq k_{\lambda}-2$, in which case this follows straight from the definition of $\widehat{\overline{SW}}_{\nu}$ together with Proposition \ref{prop:images_SW} and Corollary \ref{cor:hat_O_p_enough_inj_proj}. 

Now, $$\Hom_{\overline{\pi}(\mathcal{B}_{\lambda})}(\overline{\pi}(\mathbf{P}_i), \overline{\pi}(\mathbf{P}_j)) = \Hom_{{\mathcal{B}}_{\lambda}}(\mathbf{P}_i, \mathbf{P}_j), \, i,j \leq k_{\lambda}-2$$
(since $\mathbf{P}_i, \mathbf{P}_j$ have no non-trivial subobjects nor quotients lying in $Ker(\widehat{SW}_{\nu})$).
The proof of Sublemma \ref{sublem:SW_full_ess_surj_proj} then implies that the contravariant functor $$\widehat{\overline{SW}}_{\nu}:Proj_{\overline{\pi}(\mathcal{B}_{\lambda})} \rightarrow Inj_{\hat{\pi}(\mathfrak{B}_{\lambda})}$$ is full and essentially surjective. We have already established that $\widehat{\overline{SW}}_{\nu}$ is faithful, which concludes the proof of the sublemma.
\end{proof} 

Let $X \in \overline{\pi}(\mathcal{B}_{\lambda})$. Since $\overline{\pi}(\mathcal{B}_{\lambda})$ has enough projectives (which are also injectives), there exists an exact sequence $$ 0 \rightarrow X \rightarrow I^0_X \rightarrow I^1_X$$ in $\overline{\pi}(\mathcal{B}_{\lambda})$, where $I^0_X, I^1_X$ are injective (and thus projective as well).

Now let $P \in \overline{\pi}(\mathcal{B}_{\lambda})$ be a projective object. Sublemma \ref{sublem:overline_SW_equiv_proj} then tells us that $ \widehat{\overline{SW}}_{\nu}(P)$ is an injective object in $\hat{\pi}(\mathfrak{B}_{\lambda})$. \InnaA{Together with the fact that $ \widehat{\overline{SW}}_{\nu}$ is exact, this gives us the following commutative diagram, whose rows are short exact sequences}:
$$ \minCDarrowwidth10pt\begin{CD}
     \Hom_{\overline{\pi}(\mathcal{B}_{\lambda})}(P, X) @>>> \Hom_{\overline{\pi}(\mathcal{B}_{\lambda})}(P, I^0_X) @>>> \Hom_{\overline{\pi}(\mathcal{B}_{\lambda})}(P, I^1_X) \\
     @V{\widehat{\overline{SW}}_{\nu}}VV @V{\widehat{\overline{SW}}_{\nu}}VV @V{\widehat{\overline{SW}}_{\nu}}VV \\
     \Hom_{\hat{\pi}(\mathfrak{B}_{\lambda})}(\widehat{\overline{SW}}_{\nu}(X), \widehat{\overline{SW}}_{\nu}(P)) @>>> \Hom_{\hat{\pi}(\mathfrak{B}_{\lambda})}(\widehat{\overline{SW}}_{\nu}(I^0_X), \widehat{\overline{SW}}_{\nu}(P)) @>>> \Hom_{\hat{\pi}(\mathfrak{B}_{\lambda})}(\widehat{\overline{SW}}_{\nu}(I^1_X), \widehat{\overline{SW}}_{\nu}(P)) \\
   \end{CD}
$$
By Sublemma \ref{sublem:overline_SW_equiv_proj}, the two rightmost vertical arrows are isomorphisms, which means that the arrow $$\widehat{\overline{SW}}_{\nu}: \Hom_{\overline{\pi}(\mathcal{B}_{\lambda})}(P, X) \rightarrow \Hom_{\hat{\pi}(\mathfrak{B}_{\lambda})}(\widehat{\overline{SW}}_{\nu}(X), \widehat{\overline{SW}}_{\nu}(P))$$ is an isomorphism as well.

Now, let $Y \in \overline{\pi}(\mathcal{B}_{\lambda})$. There exists an exact sequence $$ P^1_Y \rightarrow P^0_Y \rightarrow Y \rightarrow 0$$ in $\overline{\pi}(\mathcal{B}_{\lambda})$, where $P^0_Y, P^1_Y$ are projective.
We then get the following commutative diagram, \InnaA{whose rows are short exact sequences}:
$$ \minCDarrowwidth10pt\begin{CD}
    \Hom_{\overline{\pi}(\mathcal{B}_{\lambda})}(Y, X) @>>> \Hom_{\overline{\pi}(\mathcal{B}_{\lambda})}(P^0_Y, X) @>>> \Hom_{\overline{\pi}(\mathcal{B}_{\lambda})}(P^1_Y, X) \\
     @V{\widehat{\overline{SW}}_{\nu}}VV @V{\widehat{\overline{SW}}_{\nu}}VV @V{\widehat{\overline{SW}}_{\nu}}VV \\
     \Hom_{\hat{\pi}(\mathfrak{B}_{\lambda})}(\widehat{\overline{SW}}_{\nu}(X), \widehat{\overline{SW}}_{\nu}(Y)) @>>> \Hom_{\hat{\pi}(\mathfrak{B}_{\lambda})}(\widehat{\overline{SW}}_{\nu}(X), \widehat{\overline{SW}}_{\nu}(P^0_Y))@>>> \Hom_{\hat{\pi}(\mathfrak{B}_{\lambda})}(\widehat{\overline{SW}}_{\nu}(X), \widehat{\overline{SW}}_{\nu}(P^0_Y)) \\
   \end{CD}
$$
We have already established that the two rightmost vertical arrows are isomorphisms, which means that the left vertical arrow 
$$\widehat{\overline{SW}}_{\nu}: \Hom_{\overline{\pi}(\mathcal{B}_{\lambda})}(Y, X) \rightarrow \Hom_{\hat{\pi}(\mathfrak{B}_{\lambda})}(\widehat{\overline{SW}}_{\nu}(X), \widehat{\overline{SW}}_{\nu}(Y))$$ is an isomorphism as well.
Thus $\widehat{\overline{SW}}_{\nu}$ is fully faithful.
\end{itemize}

\end{proof}

\subsection{\InnaB{Schur-Weyl functor and dualities in categories $\Dab$, $\co^{\mathfrak{p}}_{\nu, V}$}}\label{ssec:dualities_SW_commute}
\InnaB{
As a corollary, we obtain the following relation between dualities in categories $\Dab$, $\co^{\mathfrak{p}}_{\nu, V}$. 

Consider the contravariant functors $$( \cdot )^*: \Dab \rightarrow \Dab \; \; \; \text{        and        }  \; \; \; ( \cdot )^{\vee}: \co^{\mathfrak{p}}_{\nu, V} \rightarrow \co^{\mathfrak{p}}_{\nu, V}$$ where $( \cdot )^*$ is the duality functor on $\Dab$ (with respect to the tensor structure of $\Dab$), and $( \cdot )^{\vee}$ is the usual duality in the category $\co$ (c.f. Section \ref{sec:par_cat_o}, or \cite[Section 3.2]{H}). The Schur-Weyl functor $SW_{\nu, V}$ relates these two duality notions:

\begin{proposition}\label{prop:dualities_SW_commute}
 For any $\nu \in \bC$, there is an isomorphism of (covariant) functors $$ \widehat{SW}_{\nu, V}(( \cdot )^*) \longrightarrow \hat{\pi}(SW_{\nu, V}( \cdot )^{\vee})$$
\end{proposition}

\begin{proof}
 First of all, notice that both sides are exact functors. Indeed, the duality functor on any abelian rigid monoidal category is exact, and $\widehat{SW}_{\nu}$ is a (contravariant) exact functor by Lemma \ref{lem:overline_SW_exact}, which implies that $\widehat{SW}_{\nu}(( \cdot )^*)$ is exact. 
 
 On the other hand, $( \cdot )^{\vee}$ is exact (c.f. \cite[Section 3.2]{H}), so an argument similar to the proof of Lemma \ref{lem:overline_SW_exact} shows that $\hat{\pi}(SW_{\nu}( \cdot )^{\vee})$ is exact as well. 
 
 Since any object in $\Dab$ has a projective resolution, it remains to establish a natural isomorphism between the two functors when we restrict ourselves to the full subcategory of projective objects in $\Dab$.
 

 We now use the fact that all projective objects in $\Dab$ are self-dual, since they lie in $\underline{Rep}(S_{\nu})$ (c.f. Proposition \ref{prop:proj_in_ab_envelope}). This allows us to construct the isomorphism between the two functors block-by-block. 
 
 Fix a block $\mathcal{B}_{\lambda}$ of $\Dab$. If this block is semisimple, then by Proposition \ref{prop:image_SW_ss_case}, there is nothing to prove.
 
 So we assume that the block $\mathcal{B}_{\lambda}$ is not semisimple, and use the same notation as in Subsection \ref{ssec:SW_duality_non_ss_block} for simple, standard, co-standard and projective objects in both $\mathcal{B}_{\lambda}$ and the corresponding block of $\co^{\mathfrak{p}}_{\nu, V}$. We will also denote by $Proj_{\lambda}$ the full subcategory of projective objects in $\mathcal{B}_{\lambda}$.

 For each $i \geq 1$, fix a non-zero morphism $\InnaB{\beta}_{i}: \mathbf{P}_{i-1} \rightarrow \mathbf{P}_{i}$; Proposition \ref{prop:obj_ab_env} tells us that we have an exact sequence 
  $$0 \rightarrow \mathbf{M}^*_{i-1}  \rightarrow \mathbf{P}_{i-1} \stackrel{\InnaB{\beta}_{i}}{\longrightarrow} \mathbf{P}_{i} $$
 
 Recall from Theorem \ref{thrm:blocks_S_nu} that such a morphism $\InnaB{\beta}_{i}$ is unique up to a non-zero scalar, and that the morphisms $\{ \InnaB{\beta}_{i}, \InnaB{\beta}_{i}^*\}_{i \geq 1}$ generate all the morphisms in $Proj_{\lambda}$.
 
 We construct isomorphisms $$\theta_{i}: SW_{\nu}(\mathbf{P}_i^*) \longrightarrow SW_{\nu}(\mathbf{P}_i)^{\vee}$$ iteratively (recall that such isomorphisms exist by Theorem \ref{thrm:images_SW}).
 
 We start by choosing any isomorphism $\theta_{0}: SW_{\nu}(\mathbf{P}_0^*) \rightarrow SW_{\nu}(\mathbf{P}_0)^{\vee}$; at the $i$-th step, we have already constucted $\theta_0, ..., \theta_{i-1}$, and we choose an isomorphism $\theta_{i}$ so that the diagram below is commutative: 
 $$\begin{CD}
 SW_{\nu}(\mathbf{P}_i^*) @>{\theta_{i}}>> SW_{\nu}(\mathbf{P}_i)^{\vee} \\
 @A{SW_{\nu}(\InnaB{\beta}_{i}^*)}AA @A{SW_{\nu}(\InnaB{\beta}_{i})^{\vee}}AA \\
  SW_{\nu}(\mathbf{P}_{i-1}^*) @>{\theta_{i-1}}>> SW_{\nu}(\mathbf{P}_{i-1})^{\vee}
   \end{CD}$$
 We now explain why it is possible to make such a choice of $\theta_{i}$. 
 
 Applying the left-exact (covariant) functors $SW_{\nu}(\cdot)^{\vee}, SW_{\nu}((\cdot)^*)$ to the exact sequence $$0 \rightarrow \mathbf{M}^*_{i-1}  \rightarrow \mathbf{P}_{i-1} \stackrel{\InnaB{\beta}_{i}}{\longrightarrow} \mathbf{P}_{i} $$
 and using Theorem \ref{thrm:images_SW}, we see that the maps $SW_{\nu}(\InnaB{\beta}_{i}^*), SW_{\nu}(\InnaB{\beta}_{i})$ are either simultaneously zero or simultaneously not zero. Since the space $$\Hom_{\co^{\mathfrak{p}}_{\nu, V}} (SW_{\nu}(\mathbf{P}_{i-1}^*), SW_{\nu}(\mathbf{P}_{i})^{\vee})$$ is at most one-dimensional (c.f. Theorem \ref{thrm:images_SW} and Proposition \ref{prop:proj_parab_cat_O_struct}), we can take $\theta_i$ to be any isomorphism $SW_{\nu}(\mathbf{P}_i^*) \rightarrow SW_{\nu}(\mathbf{P}_i)^{\vee}$, and then multiply it by a non-zero scalar to make the above diagram commutative.
 
 We now claim that the isomorphisms $\theta_{i}$ define a natural transformation. Since the morphisms $\{ \InnaB{\beta}_{i}, \InnaB{\beta}_{i}^*\}_{i \geq 1}$ generate all the morphisms in $Proj_{\lambda}$, we only need to check that for any $i \geq 1$, the following diagram is commutative:
  $$\begin{CD}
 SW_{\nu}(\mathbf{P}_i^*) @>{\theta_{i}}>> SW_{\nu}(\mathbf{P}_i)^{\vee} \\
 @V{SW_{\nu}(\InnaB{\beta}_{i})}VV @V{SW_{\nu}(\InnaB{\beta}_{i}^*)^{\vee}}VV \\
  SW_{\nu}(\mathbf{P}_{i-1}^*) @>{\theta_{i-1}}>> SW_{\nu}(\mathbf{P}_{i-1})^{\vee}
   \end{CD}$$
 The latter follows easily from the construction of $\theta_i$, together with the fact that $\mathbf{P}_i = \mathbf{P}_i^*$ (for any $i\geq 0$) and $\theta_i = \theta_i^{\vee}$.
 \end{proof}

}

\newpage
\appendix
\section{Proofs of Lemmas \ref{lem:complex_ten_power_action_well_def}, \ref{lem:tens_power_weak_torsion_free}, \ref{lem:gl_V_action_usual_tens_power}}\label{app:complex_ten_power_action_well_def}
\subsection{Action of $\gl(V)$ in Definition \ref{def:complex_ten_power_splitting}}\label{appssec:complex_ten_power_action_well_def}
\begin{lemma}
 The action of $\gl(V)$ described in Definition \ref{def:complex_ten_power_splitting} is well-defined.
\end{lemma}

\begin{proof}
Let $ u, u_1, u_2 \in U \cong \mathfrak{u}_{\mathfrak{p}}^{-}, f, f_1, f_2 \in U^* \cong \mathfrak{u}_{\mathfrak{p}}^{+}, A \in \gl(U)$. We have to check that the morphisms in $\underline{Rep}(S_{\nu})$ by which $\InnaA{u, f, A}$ act are well-defined and satisfy the same commutation relations as do $\InnaA{u, f, A} \in \gl(V)$.

The first claim is obvious for the actions of $f$ and $A$ and one only needs to check that the image of $(U^{\otimes k} \otimes \Del_k)^{S_k}$ under $\frac{1}{k+1}\sum_{1 \leq l \leq k+1} u^{(l)} \otimes res_{l}^*$ is $S_{k+1}$-invariant. For this, we will prove
\InnaA{\begin{lemma}\label{applem:res_commutant_S_k_projection}
 Let $\sigma \in S_{k+1}, l \in \{1,...,k+1\}$. Then there exists $\rho_l(\sigma) \in S_k$ such that $$ \sigma \circ res_{l}^*  = res_{\sigma(l)}^*\circ \rho_l(\sigma), \sigma \circ u^{(l)}  = u^{(\sigma(l))}\circ \rho_l(\sigma)$$
\end{lemma}
\begin{proof}
We define the permutation $\rho_l(\sigma)$ to be the diagram in \InnaA{$\bar{P}_{k, k}$ constructed as follows.

Consider the diagram $\sigma \in \bar{P}_{k+1, k+1}$. Remove vertex $l$ in its top row, vertex $\sigma(l)$ in the bottom row, as well as the edge connecting these vertices. The obtained diagram will lie in $\bar{P}_{k, k}$ and will have no solitary vertices; thus it represents a permutation in $S_k$.}

The diagram obtained is the same we would get by considering the diagram for $\sigma \circ res_{l}^* \in \bar{P}_{k, k+1}$, and removing the unique solitary vertex $\sigma(l)$ from the bottom row of $\sigma \circ res_{l}^*$. From this construction we immediately get: $ \sigma \circ res_{l}^*  = res_{\sigma(l)}^*\circ \rho_l(\sigma)$.
One then easily sees that $\sigma \circ u^{(l)}  = u^{(\sigma(l))}\circ \rho_l(\sigma)$ holds as well.
\end{proof}

We now see that for any $\sigma \in S_{k+1}$, $$\frac{1}{k+1}\sum_{1 \leq l \leq k+1} (\sigma \circ u^{(l)}) \otimes (\sigma \circ res_{l}^*) = \frac{1}{k+1}\sum_{1 \leq l \leq k+1} (u^{(\sigma(l))}\circ \rho_l(\sigma)) \otimes (res_{\sigma(l)}^*\circ \rho_l(\sigma))  $$
Restricted to $(U^{\otimes k} \otimes \Del_k)^{S_k}$, the latter morphism equals $$\frac{1}{k+1}\sum_{1 \leq l \leq k+1} u^{(\sigma(l))} \otimes res_{\sigma(l)}^* = \frac{1}{k+1}\sum_{1 \leq l \leq k+1} u^{(l)} \otimes res_{l}^*$$ as wanted.}

Moving on to the commutation relations, one only needs to check that the following commutation relations between operators on $(U^{\otimes k} \otimes \Del_k)^{S_k}$ hold (the rest are obvious):
\begin{enumerate}[leftmargin=*, label=(\alph*)]
 \item $$\frac{1}{(k+1)(k+2)}\sum_{\substack{1 \leq l_1 \leq k+1 \\ 1 \leq l_2 \leq k+2}} (u_2^{(l_2)} \circ u_1^{(l_1)}) \otimes (res_{l_2}^* \circ res_{l_1}^*) \stackrel{?}{=} \frac{1}{(k+1)(k+2)}\sum_{\substack{1 \leq l_2 \leq k+1 \\ 1 \leq l_1 \leq k+2}} (u_1^{(l_1)} \circ u_2^{(l_2)}) \otimes (res_{l_1}^* \circ res_{l_2}^*)
$$
\item $$
\sum_{\substack{1 \leq l_1 \leq k \\ 1 \leq l_2 \leq k-1}} (f_2^{(l_2)} \circ f_1^{(l_1)}) \otimes (res_{l_2} \circ res_{l_1}) \stackrel{?}{=} \sum_{\substack{1 \leq l_2 \leq k \\ 1 \leq l_1 \leq k-1}} (f_1^{(l_1)} \circ f_2^{(l_2)}) \otimes (res_{l_1} \circ res_{l_2})
$$
\item 
\begin{align*}
&\frac{1}{(k+1)}\sum_{1 \leq l_1, l_2 \leq k+1} (f^{(l_2)} \circ u^{(l_1)}) \otimes (res_{l_2} \circ res_{l_1}^*) \stackrel{?}{=} \frac{1}{k}\sum_{1 \leq l_1,l_2 \leq k} (u^{(l_1)} \circ f^{(l_2)}) \otimes (res_{l_1}^* \circ res_{l_2}) +\\
&+(\nu -k) f(u) \InnaA{\id_{(U^{\otimes k} \otimes \Del_k)^{S_k}} - T_{f,u}\rvert_{(U^{\otimes k} \otimes \Del_k)^{S_k}}}
\end{align*}
\end{enumerate}


These identities are proved below.

\begin{enumerate}[leftmargin=*, label=(\alph*)]
  \item The claim follows immediately from the following easy computations (consequences of Lemma \ref{lem:Delta_k_homs}):
\begin{itemize}[leftmargin=*]

\item For $1 \leq l_1 < l_2 \leq k+2$,
$$ res_{l_2}^* \circ res_{l_1}^* = res_{l_1}^* \circ res_{l_2-1}^* $$ as operators on $\Del_k$.
We also have $ u_2^{(l_2)} \circ u_1^{(l_1)} = u_1^{(l_1)} \circ u_2^{(l_2-1)}$.

\item For $k+1 \geq l_1 \geq l_2 \geq 1$,
$$ res_{l_2}^* \circ res_{l_1}^* = res_{l_1+1}^* \circ res_{l_2}^*$$ as operators on $\Del_k$.
We also have $u_2^{(l_2)} \circ u_1^{(l_1)} = u_1^{(l_1+1)} \circ u_2^{(l_2)}$.
\end{itemize}
\item The claim follows immediately from the following easy computations (consequences of Lemma \ref{lem:Delta_k_homs}):
\begin{itemize}[leftmargin=*]
\item For $1 \leq l_1 \leq l_2 \leq k-1$,
$$ res_{l_2} \circ res_{l_1} = res_{l_1} \circ res_{l_2+1} $$ as operators on $\Del_k$.
We also have $ f_2^{(l_2)} \circ f_1^{(l_1)} = f_1^{(l_1)} \circ f_2^{(l_2+1)}$.

\item For $k \geq l_1 > l_2 \geq 1$,
$$ res_{l_2} \circ res_{l_1} = res_{l_1-1} \circ res_{l_2}$$ as operators on $\Del_k$.
We also have $f_2^{(l_2)} \circ f_1^{(l_1)} = f_1^{(l_1-1)} \circ f_2^{(l_2)}$.
\end{itemize}

\item
  We have:
\begin{itemize}[leftmargin=*]
 \item For any $1\leq l \leq k+1$, $ res_{l} \circ res_{l}^* = (\nu -k) \id_{\Del_k}$ and thus $$ (f^{(l)} \circ u^{(l)}) \otimes (res_{l} \circ res_{l}^* ) = (\nu -k) f(u) \id_{(U^{\otimes k} \otimes \Del_k)^{S_k}}$$

\item For $1 \leq l_1 < l_2 \leq k+1$,
$$ res_{l_2} \circ res_{l_1}^* = res_{l_1}^* \circ res_{l_2-1} - C_{(l_1, ..., l_2-1)}$$ as operators on $\Del_k$, where $C_{(l_1, ..., l_2-1)}: \Delta_k \rightarrow \Del_k$ is the action of the cycle $C_{(l_1, ..., l_2-1)} \in S_k$ on $\Del_k$.
We also have $$ f^{(l_2)} \circ u^{(l_1)} = u^{(l_1)} \circ f^{(l_2-1)} = T_{f,u}^{(l_1)} \circ C_{(l_1, ..., l_2-1)}$$
Thus
\begin{align*}
& (f^{(l_2)} \circ u^{(l_1)}) \otimes (res_{l_2} \circ res_{l_1}^*) = (u^{(l_1)} \circ f^{(l_2-1)}) \otimes (res_{l_1}^* \circ res_{l_2-1}) -\\
&- (T_{f,u}^{(l_1)} \circ C_{(l_1, ..., l_2-1)} ) \otimes C_{(l_1, ..., l_2-1)}
\end{align*}
as operators on $U^{\otimes k} \otimes \Del_k$.

Finally, note that $res_{l_1}^* \circ res_{l_2-1} \circ C_{(l_1, ..., l_2-1)}^{-1} = res_{l_1}^* \circ res_{l_1}$.

\item For $k+1 \geq l_1 > l_2 \geq 1$, $ res_{l_2} \circ res_{l_1}^* = res_{l_1-1}^* \circ res_{l_2} - C_{(l_2, ..., l_1-1)}^{-1}$ as operators on $\Del_k$, where $C_{(l_2, ..., l_1-1)}: \Delta_k \rightarrow \Del_k$ is the action of the cycle $C_{(l_2, ..., l_1-1)} \in S_k$ on $\Del_k$.
We also have $$ f^{(l_2)} \circ u^{(l_1)} = u^{(l_1-1)} \circ f^{(l_2)} = T_{f,u} \circ C_{(l_2, ..., l_1-1)}^{-1}$$
Thus
\begin{align*}
& (f^{(l_2)} \circ u^{(l_1)}) \otimes (res_{l_2} \circ res_{l_1}^*) = (u^{(l_1-1)} \circ f^{(l_2)}) \otimes (res_{l_1-1}^* \circ res_{l_2}) -\\
&- (T_{f,u}^{(l_1-1)} \circ C_{(l_2, ..., l_1-1)}^{-1} ) \otimes C_{(l_2, ..., l_1-1)}^{-1}
\end{align*}
as operators on $U^{\otimes k} \otimes \Del_k$.

Finally, note that $res_{l_1-1}^* \circ res_{l_2} \circ C_{(l_1, ..., l_2-1)} = res_{l_1-1}^* \circ res_{l_1-1}$.
\end{itemize}

Together these imply the following identities of operators on $(U^{\otimes k} \otimes \Del_k)^{S_k}$:
\begin{align*}
&\frac{1}{(k+1)}\sum_{1 \leq l_1, l_2 \leq k+1} (f^{(l_2)} \circ u^{(l_1)}) \otimes (res_{l_2} \circ res_{l_1}^*) = (\nu -k) f(u) \id_{(U^{\otimes k} \otimes \Del_k)^{S_k}}+\\
&+ \frac{1}{(k+1)}\sum_{1 \leq l_1 < l_2 \leq k+1}  (u^{(l_1)} \circ f^{(l_2-1)}) \otimes (res_{l_1}^* \circ res_{l_2-1})
- (T_{f,u}^{(l_1)} \circ C_{(l_1, ..., l_2-1)} ) \otimes C_{(l_1, ..., l_2-1)} +\\
&+ \frac{1}{(k+1)}\sum_{1 \leq l_2 < l_1 \leq k+1} (u^{(l_1-1)} \circ f^{(l_2)}) \otimes (res_{l_1-1}^* \circ res_{l_2})
- (T_{f,u}^{(l_1-1)} \circ C_{(l_2, ..., l_1-1)}^{-1} ) \otimes C_{(l_2, ..., l_1-1)}^{-1} =\\
&= (\nu -k) f(u) \id_{(U^{\otimes k} \otimes \Del_k)^{S_k}} - T_{f,u} \rvert_{(U^{\otimes k} \otimes \Del_k)^{S_k}} +\\
&+ \frac{1}{(k+1)}\sum_{1 \leq l_1, l_2 \leq k}  (u^{(l_1)} \circ f^{(l_2)}) \otimes (res_{l_1}^* \circ res_{l_2})
 + \frac{1}{(k+1)}\sum_{1 \leq l_1 \leq k+1} (u^{(l_1)} \circ f^{(l_1)}) \otimes (res_{l_1}^* \circ res_{l_1})=\\
&=(\nu -k) f(u) \id_{(U^{\otimes k} \otimes \Del_k)^{S_k}}- T_{f,u} \rvert_{(U^{\otimes k} \otimes \Del_k)^{S_k}} +\\
&+ \frac{1}{(k+1)}\sum_{1 \leq l_1, l_2 \leq k}  (u^{(l_1)} \circ f^{(l_2)}) \otimes (res_{l_1}^* \circ res_{l_2})
 + \frac{1}{k(k+1)}\sum_{1 \leq l_1, l_2 \leq k}  (u^{(l_1)} \circ f^{(l_2)}) \otimes (res_{l_1}^* \circ res_{l_2})=\\
&=\frac{1}{k}\sum_{1 \leq l_1,l_2 \leq k} (u^{(l_1)} \circ f^{(l_2)}) \otimes (res_{l_1}^* \circ res_{l_2})  +(\nu -k) f(u) \id_{(U^{\otimes k} \otimes \Del_k)^{S_k}} - T_{f,u}\rvert_{(U^{\otimes k} \otimes \Del_k)^{S_k}}
\end{align*}

\end{enumerate}
\end{proof}

\subsection{Proof of Lemma \ref{lem:tens_power_weak_torsion_free}}\label{appssec:tens_power_weak_torsion_free}
\begin{lemma}\label{applem:tens_power_weak_torsion_free}
Let \InnaA{$l \leq k$, and consider a non-zero morphism in $\underline{Rep}(S_{\nu})$} $$\phi: U^{\otimes l} \otimes \Del_l \longrightarrow U^{\otimes k} \otimes \Del_k $$ Let $u \in U \cong \mathfrak{u}_{\mathfrak{p}}^{-}, u \neq 0$.
Then $F_u \circ \phi \neq 0$, where $F_u \circ \phi := \frac{1}{k+1}\sum_{1 \leq l \leq k+1} (u^{(l)} \otimes res_l^*) \circ \phi$.
\end{lemma}
\begin{proof}
  Recall from Lemma \ref{lem:Delta_k_homs} that $$\Hom_{\underline{Rep}(S_{\nu})} (\Del_l,\Del_k) := \bC \bar{P}_{l,k}$$ where $ \bar{P}_{l,k}$ is the set of partitions of $\{1,..,l, 1',...,k'\}$ into disjoint subsets such that $i, j$ do not lie in the same subset, and neither do $i', j'$, for any $i \neq j, i' \neq j'$.

%
  So $$\Hom_{\underline{Rep}(S_{\nu})} (U^{\otimes l} \otimes \Del_l, U^{\otimes k} \otimes \Del_k) \cong \bC \bar{P}_{l,k} \otimes U^{\otimes k} \otimes {U^*}^{\otimes l}$$

  We now \InnaA{study} the map $$F_u \circ (\cdot): \bC \bar{P}_{l,k} \otimes U^{\otimes k} \otimes {U^*}^{\otimes l} \rightarrow \bC \bar{P}_{l,k+1} \otimes U^{\otimes k+1} \otimes {U^*}^{\otimes l}$$
  By definition of $F_u$, we know that $$F_u \circ (x \otimes u_1 \otimes ...\otimes u_k \otimes f_1 \otimes ...\otimes f_l) = \frac{1}{k+1}\sum_{1 \leq s \leq k+1} res_s^*(x) \otimes u_1 \otimes ...\otimes u_{s-1} \otimes u \otimes u_s \otimes ... \otimes u_k \otimes f_1 \otimes ...\otimes f_l$$ where $x \in \bar{P}_{l,k}, u_1, ..., u_k \in U, f_1 , ..., f_l \in U^*$.

  As we said before, we can consider $\phi$ as an element of $\bC \bar{P}_{l,k} \otimes U^{\otimes k} \otimes {U^*}^{\otimes l}$.

  Let $N:=\dim U$, and choose a basis $u_1, ..., u_N$ of $ U$ such that $u_1 =u$. Then we can write $$\phi = \sum_{\substack{x \in \bar{P}_{l,k},\\ i_1,...,i_k \in \{1,...,N\}}} \alpha_{x, I} x \otimes u_{i_1} \otimes ...\otimes u_{i_k}$$ where $I$ denotes the sequence $(i_1,...,i_k)$ and $ \alpha_{x, I} \in {U^*}^{\otimes l}$.

  Now assume $F_u \circ \phi=0$, i.e. $$\sum_{1 \leq s \leq k+1} \sum_{ \substack{x \in \bar{P}_{l,k},\\ i_1,...,i_k \in \{1,...,N\}}} \alpha_{x, I} res_s^*(x) \otimes u_{i_1} \otimes ...\otimes u_{i_{s-1}} \otimes u \otimes u_{i_s} \otimes ...\otimes u_{i_k} =0$$

  So for any $y \in \bar{P}_{l,k+1}$, and any sequence $J = (j_1,...,j_{k+1})$ (here $j_1,..,j_{k+1} \in \{1,...,N\}$), we have $$\sum_{\InnaA{\substack{\text{triples } (x, I, s): \\ 1 \leq s \leq k+1, \: res_s^*(x)=y, \: j_s =1, \: I=(j_1,...,j_{s-1}, j_{s+1},...,j_{k+1}) }}}\alpha_{x, I} = 0$$

  We will now show that this implies that $\alpha_{x, I} =0$ for any $x \in \bar{P}_{l,k}, I = (i_1,...,i_k), i_1,...,i_k \in \{1,...,N\}$, which will mean that $\phi=0$ and thus lead to a contradiction.

  For our convenience, let us denote by $Ins_s(I)$ the sequence $(i_1,..., i_{s-1}, 1, i_s, ...,i_k)$ ($1$ inserted in the $s$-th place). We will also use the following notation:
  \begin{itemize}
   \item For $x \in \bar{P}_{l,k}$, consider the longest sequence of consecutive solitary vertices in the bottom row of the diagram of $x$ (\InnaA{if there are several such sequences of maximal length, choose the first one}).

   Denote the length of this sequence by $m(x)$. Let $j_x$ be such that $j_x+1$ is the first element of this sequence (if this sequence is empty, then put $j_x:=1$).

   So this sequence of solitary vertices in $x$ is $\{j_x+1, j_x+2, ..., j_x+m(x)\}$.

   \item Let $x \in \bar{P}_{l,k}, I = (i_1,...,i_k), i_1,...,i_k \in \{1,...,N\}$. Consider the sequence $(i_{j_x+1}, i_{j_x+2}, ..., i_{j_x+m(x)})$ and inside it the longest segment of consecutive occurrences of $1$ (\InnaA{if there are several such segments of maximal length, choose the first one}).

   Denote the length of this segment by $M(I,x)$. Let $j_{I,x}$ be such that $j_{I, x}+1$ is the position of the first element of this segment (if this segment is empty, i.e. $M(I,x)=0$, then put $j_{I,x}:=j_x$).

  \end{itemize}

  \InnaA{
  We now rewrite the equality we obtained above: for any triple $x, I, s$ where $x \in \bar{P}_{l,k}$, $I$ is a sequence of length $k$ with entries in $\{1,...,N\}$, and $1 \leq s \leq k+1$, we have:  
  $$\sum_{\substack{\text{triples } (x', I', s'): \\ 1 \leq s' \leq k+1, \: res_{s'}^*(x')=res_s^*(x), \: Ins_{s'}(I') =Ins_s(I)}}\alpha_{x', I'} = 0$$
We will now use two-fold descending induction on the values $m(x), M(I, x)$ to prove that $\alpha_{x, I} =0$ for any $x \in \bar{P}_{l,k}$, and any sequence $I$ of length $k$ with entries in $\{1,...,N\}$.}

 Base: Let $x, I$ such that $m(x)=k, M(I, x)=k$. Then the bottom row of $x$ consists of solitary vertices, and $I$ consists only of $1$'s. Now choose any $s \in \{1,...,k+1\}$. Then, by definition, the bottom row of $res_s^*(x)$ consists of solitary vertices, and $Ins_s(I)$ consists only of $1$'s.

 \InnaA{Then for any triple} $(x', I', s')$ which satisfies $res_{s'}^*(x')=res_s^*(x), Ins_{s'}(I') =Ins_s(I)$, we will have $x'=x$, $I'=I$. The above equality then implies that $\alpha_{x, I} =0$.

 Step: Let $0\leq M, m \leq k$, and $M+m <2k$. Assume $\alpha_{x, I} =0$ for any $x, I$ such that either $m(x)> m$, or $m(x)=m,  M(I, x) >M$.

 Let $x, I$ be such that $m(x)=m,  M(I, x)=M$. Set $s:=j_{I,x}+1$.

 All we have to do is prove the following Sublemma, and we are done.
 \begin{sublemma}
   Let $(x', I', s')$ be a triple which satisfies $res_{s'}^*(x')=res_s^*(x), Ins_{s'}(I') =Ins_s(I)$. Then one of the following statements holds:
  \begin{itemize}
   \item $m(x') > m(x)$,
   \item $x' =x, M(I', x) > M(I, x)$,
   \item $x'=x, I'=I$.
  \end{itemize}
 \end{sublemma}
\begin{proof}

By definition, $res_s^*(x)$ has a sequence of $m(x)+1$ consecutive solitary vertices. We assumed that $res_{s'}^*(x')=res_s^*(x)$, so $x'$ is obtained by removal of the $s'$-th vertex from the bottom row of $res_s^*(x)$. So either $x'$ has a sequence of $m(x)+1$ consecutive solitary vertices, i.e. $m(x')=m(x)+1$, or we are removing one of the $m(x)+1$ consecutive solitary vertices of $res_s^*(x)$, which means that $x'=x$.

Now, assume $x'=x$, and use a similar argument for $I, I'$. By definition, the sequence $Ins_s(I)$ has a segment of $ M(I, x)+1$ consecutive occurrences of $1$. Again, we assumed that $Ins_{s'}(I') =Ins_s(I)$, so $I'$ is obtained by removal of the $s'$-th element of the sequence $Ins_s(I)$. So either $I'$ has a segment of $ M(I, x)+1$ consecutive occurrences of $1$, i.e. $M(I', x)= M(I, x)+1$, or $I'=I$.
\end{proof}

\end{proof}

\subsection{Proof of Lemma \ref{lem:gl_V_action_usual_tens_power}}\label{appssec:action_on_usual_tensor_power_lemmas}
\InnaA{
Let $V \cong \InnaB{\bC \triv} \oplus U$ be a split unital finite-dimensional vector space. We will use the same notations as in Section \ref{ssec:comp_tens_power_compatible_classical}.

The following two technical lemmas will be used to prove Lemma \ref{applem:gl_V_action_usual_tens_power}.
 \begin{lemma}\label{applem:action_on_usual_tensor_power1}
   Let $k \in \{0, ..., n-1\}, \{j_1 < j_2 <... <j_k\} \subset \{1, ..., n\}$, $u \in U$, $v_{j_1}, v_{j_2}, ..., v_{j_k} \in U$, and let $f_{j_1 < j_2 < ...<j_k}$ be the map $\{1,...,k\} \rightarrow \{1, ..., n\}$ taking $i$ to $j_i$. Then 
   \begin{align*}
  &e_{S_{k+1}}\left(\sum_{ 1\leq l \leq k+1} \sum_{\substack {g \in Inj(\{1,...,k+1\} , \{1,...,n\}):\\ g \circ \iota_l =f_{j_1 < j_2 < ...<j_k}\\ g \text{ monotone increasing} }}  u^{(l)}.(v_{j_1} \otimes v_{j_2} \otimes ...\otimes v_{j_k}) \otimes g \right) =\\
  &=\frac{1}{(k+1)!}\sum_{ 1\leq l \leq k+1} \sum_{\sigma \in S_k}  (u^{(l)}\circ \sigma)(v_{j_1} \otimes v_{j_2} \otimes ...\otimes v_{j_k}) \otimes ( \mathtt{res}_l^*\circ \sigma)(f_{j_1 < j_2 < ...<j_k})
   \end{align*}
 \end{lemma}
\begin{proof}
We rewrite both sides of the identity we want to prove: the left hand side becomes
   \begin{align*}
  &\frac{1}{(k+1)!}\sum_{\substack {\sigma \in S_{k+1}, \\ 1\leq l \leq k+1}}   (\sigma \circ u^{(l)}).(v_{j_1} \otimes v_{j_2} \otimes ...\otimes v_{j_k}) \otimes \left(\sum_{\substack {g \in Inj(\{1,...,k+1\} , \{1,...,n\}):\\ g \circ \iota_l =f_{j_1 < j_2 < ...<j_k}\\ g \text{ monotone increasing} }} g \circ \sigma^{-1}\right) 
   \end{align*}
   and the right hand side becomes
      \begin{align*}
  &\frac{1}{(k+1)!}\sum_{ 1\leq l' \leq k+1} \sum_{\sigma' \in S_k}  (u^{(l')}\circ \sigma').(v_{j_1} \otimes v_{j_2} \otimes ...\otimes v_{j_k}) \otimes \left( \sum_{\substack{g' \in Inj(\{1,...,k+1\} , \{1,...,n\}):\\ g' \circ \iota_{l'} =f_{j_1 < j_2 < ...<j_k} \circ \sigma'^{-1} } } g' \right)
   \end{align*}
   
We now define the following map:
\begin{align*}
 \mu: S_{k+1} \times \{1,...,k+1\} &\longrightarrow S_{k} \times \{1,...,k+1\} \\
 (\sigma, l) &\mapsto (\rho_l(\sigma), \sigma(l))
\end{align*}
where $\rho_l(\sigma)$ is defined in Lemma \ref{applem:res_commutant_S_k_projection}.

Then it is enough to check that for every $(\sigma', l') \in S_{k} \times \{1,...,k+1\}$, the following identity holds:
\begin{align*}
&\sum_{(\sigma, l) \in \mu^{-1}(\sigma', l')}  (\sigma \circ u^{(l)}).(v_{j_1} \otimes v_{j_2} \otimes ...\otimes v_{j_k})\otimes \left( \sum_{\substack {g \in Inj(\{1,...,k+1\} , \{1,...,n\}):\\ g \circ \iota_l =f_{j_1 < j_2 < ...<j_k}\\ g \text{ monotone increasing} }}  g \circ \sigma^{-1} \right)= \\
&= (u^{(l')}\circ \sigma').(v_{j_1} \otimes v_{j_2} \otimes ...\otimes v_{j_k}) \otimes \left( \sum_{\substack{g' \in Inj(\{1,...,k+1\} , \{1,...,n\}):\\ g' \circ \iota_{l'} =f_{j_1 < j_2 < ...<j_k} \circ \sigma'^{-1} } } g' \right)
\end{align*}

From Lemma \ref{applem:res_commutant_S_k_projection}, we know that $\sigma \circ u^{(l)} = u^{(l')}\circ \sigma'$ for any $(\sigma, l) \in \mu^{-1}(\sigma', l')$.

So we need to check that
\begin{align*}
&\sum_{(\sigma, l) \in \mu^{-1}(\sigma', l')} \sum_{\substack {g \in Inj(\{1,...,k+1\} , \{1,...,n\}):\\ g \circ \iota_l =f_{j_1 < j_2 < ...<j_k}\\ g \text{ monotone increasing} }} (g \circ \sigma^{-1}) =  \sum_{\substack{g' \in Inj(\{1,...,k+1\} , \{1,...,n\}):\\ g' \circ \iota_{l'} =f_{j_1 < j_2 < ...<j_k} \circ \sigma'^{-1} } } g' 
\end{align*}
Notice that by definition of $\mu$, $\sigma \circ \iota_l = \iota_{\sigma(l)} \circ \rho_l(\sigma)$, i.e. $$\sigma \circ \iota_l = \iota_{l'} \circ \sigma'$$ Thus for any monotone increasing map $g :\{1,...,k+1\} \rightarrow \{1,...,n\}$ such that $g \circ \iota_l =f_{j_1 < j_2 < ...<j_k}$, the map $g':= g \circ \sigma^{-1}$ is an injective map $\{1,...,k+1\} \rightarrow \{1,...,n\}$ satisfying: $g' \circ \iota_{l'} =f_{j_1 < j_2 < ...<j_k} \circ \sigma'^{-1}$.

It remains to check that the summands in the left hand side are pairwise different and that both sides have the same number of summands. 

The first of these statements is proved as follows: let $g_1 \circ \sigma^{-1}$, $g_2 \circ \tau^{-1}$ be two summands in the left hand side. Assume they are equal. This means that $g_1, g_2$ have the same image, and since they both are monotone increasing, we conclude that $g_1 =g_2$, which of course means that $\sigma =\tau$, and we are done.

The second statement is proved as follows: the number of summands in the right hand side is obviously $n-k$. To show that this is also the number of summands in the left hand side, we only need to check that the projection $S_{k+1} \times \{1,...,k+1\} \rightarrow \{1,...,k+1\}$ maps $\mu^{-1}(\sigma', l')$ bijectively to $\{1,...,k+1\}$. 

By the definition of $\mu$ and the construction described in Lemma \ref{applem:res_commutant_S_k_projection}, we see that for every $l \in \{1,...,k+1\}$, we can (uniquely) reconstruct $\sigma$ from the data $(\sigma', l', l)$ so that $\mu(\sigma, l)= (\sigma', l')$: we consider the diagram of $\sigma' \in \bar{P}_{k, k}$, insert a vertex in the $l$-th position in the top row, a vertex in the $l'$-th position in the bottom row and an edge connecting the two. The obtained diagram will be $\sigma$.
This completes the proof of the lemma.
\end{proof}

  \begin{lemma}\label{applem:action_on_usual_tensor_power2}
  Let $k \in \{0, ..., n-1\}, \{j_1 < j_2 <... <j_{k+1}\} \subset \{1, ..., n\}$, $\lambda \in U^*$, $v_{j_1}, v_{j_2}, ..., v_{j_{k+1}} \in U$, and let $f_{j_1 < j_2 < ...<j_{k+1}}$ be the map $\{1,...,{k+1}\} \rightarrow \{1, ..., n\}$ taking $i$ to $j_i$. Then 
   \begin{align*}
  & e_{S_{k}} \left(\sum_{ 1\leq l \leq k+1} \lambda^{(l)}.(v_{j_1} \otimes v_{j_2} \otimes ...\otimes v_{j_{k+1}}) \otimes \mathtt{res}_l(f_{j_1 < j_2 < ...<j_{k+1}}) \right)=\\
  &=\frac{1}{(k+1)!}\sum_{ 1\leq l \leq k+1} \sum_{\sigma \in S_{k+1}}  (\lambda^{(l)}\circ \sigma)(v_{j_1} \otimes v_{j_2} \otimes ...\otimes v_{j_{k+1}}) \otimes ( \mathtt{res}_l\circ \sigma)(f_{j_1 < j_2 < ...<j_{k+1}})
   \end{align*}
 \end{lemma}
\begin{proof}
 We rewrite the left hand of the identity we want to prove, and it becomes
   \begin{align*}
  &\frac{1}{k!}\sum_{\substack {\sigma' \in S_{k}, \\ 1\leq l' \leq k+1}}   (\sigma' \circ \lambda^{(l')}).(v_{j_1} \otimes v_{j_2} \otimes ...\otimes v_{j_{k+1}}) \otimes \left(\sigma' \circ \mathtt{res}_{l'}\right)(f_{j_1 < j_2 < ...<j_{k+1}})
   \end{align*}
   
 We use the definition of the map $\mu$ from the proof of Lemma \ref{applem:action_on_usual_tensor_power1}, and define the map 
 \begin{align*}
 \tilde{\mu}: S_{k+1} \times \{1,...,k+1\} &\longrightarrow S_{k} \times \{1,...,k+1\} \\
 (\sigma, l) &\mapsto \mu(\sigma^{-1}, l)
\end{align*}
Then it is enough to check that for every $(\sigma', l') \in S_{k} \times \{1,...,k+1\}$,
\begin{align*}
&(\sigma' \circ \lambda^{(l')}).(v_{j_1} \otimes v_{j_2} \otimes ...\otimes v_{j_{k+1}}) \otimes \left(\sigma' \circ \mathtt{res}_{l'}\right)(f_{j_1 < j_2 < ...<j_{k+1}}) =\\
&= \frac{1}{k+1}\sum_{(\sigma, l) \in \tilde{\mu}^{-1}(\sigma'^{-1}, l')}(\lambda^{(l)}\circ \sigma)(v_{j_1} \otimes v_{j_2} \otimes ...\otimes v_{j_{k+1}}) \otimes ( \mathtt{res}_l\circ \sigma)(f_{j_1 < j_2 < ...<j_{k+1}})
\end{align*}

By definition of $\tilde{\mu}$, for every $(\sigma, l) \in \tilde{\mu}^{-1}(\sigma'^{-1}, l')$, we have: $\mu(\sigma^{-1}, l)  = (\sigma'^{-1}, l')$, which means that $\sigma^{-1} \circ \iota_l = \iota_{l'} \circ \sigma'^{-1}$ (see the proof of Lemma \ref{applem:action_on_usual_tensor_power1}), and so 
$$\sigma' \circ \mathtt{res}_{l'} = \mathtt{res}_{l} \circ\sigma$$ and similarly
$$\sigma' \circ \lambda^{(l')} = \lambda^{(l)}\circ \sigma$$
Thus it only remains to check that the right hand side has $k+1$ summands, i.e. that $\tilde{\mu}^{-1}(\sigma'^{-1}, l')$ has $k+1$ elements. The latter can be easily deduced from the arguments in the proof of Lemma \ref{applem:action_on_usual_tensor_power1}.
\end{proof}

\begin{lemma}\label{applem:gl_V_action_usual_tens_power}
There is an isomorphism of $\mathfrak{gl}(V)$-modules \InnaA{$$\Phi: V^{\otimes n} \stackrel{\sim}{\longrightarrow} \bigoplus_{k=0,...,n} (U^{\otimes k} \otimes \bC Inj(\{1,...,k\} , \{1,...,n\}))^{S_k}$$
where $\Phi (\triv \otimes \triv \otimes ... \otimes \triv) =1$ \InnaB{(lies in degree zero of the right hand side)}.

Moreover, this isomorphism is an isomorphism of $\bC[S_n] \otimes_{\bC} \mathcal{U}(\gl(V))$-modules.}
\end{lemma}

\begin{proof}
\InnaB{Fix a dual basis vector $\triv^* \in (\bC \triv)^*$ such that $\triv^* ( \triv )=1$.}

Given a subset $J = \{j_1 < j_2 < ...<j_k\} \subset \{1,...,n\}$, let $$U^{\otimes J} = \InnaB{\bC \triv}^{\otimes j_1-1} \otimes U \otimes \InnaB{\bC \triv}^{\otimes j_2-j_1-1} \otimes U \otimes... \otimes \InnaB{\bC \triv}^{\otimes j_k-j_{k-1}-1} \otimes U \otimes \InnaB{\bC \triv}^{\otimes n-j_k}$$ (that is, the factors $j_1, j_2$, etc. are $U$, and the rest are $\InnaB{\bC \triv}$).
Then $V^{\otimes n} = \bigoplus_{J \subset \{1,..,n\}} U^{\otimes J}$.

Let
\begin{align*}
\Phi_{J}:U^{\otimes \{j_1 < j_2 < ...<j_k\}} &\longrightarrow (U^{\otimes k} \otimes \bC Inj(\{1,...,k\} , \{1,...,n\}))^{S_k}\\
v_{1} \otimes ...\otimes v_{n} &\mapsto  e_{S_k}(v_{j_1} \otimes v_{j_2} \otimes ...\otimes v_{j_k} \otimes f_{j_1 < j_2 < ...<j_k}) \cdot \prod_{i \notin J} \InnaB{\triv^*(v_i)}
\end{align*}
 Here $f:=f_{j_1 < j_2 < ...<j_k} \in Inj(\{1,...,k\} , \{1,...,n\})$ is given by $f(s):=j_s$ and $e_{S_k}$ is the projection $$U^{\otimes k} \otimes \bC Inj(\{1,...,k\} , \{1,...,n\}) \rightarrow (U^{\otimes k} \otimes \bC Inj(\{1,...,k\} , \{1,...,n\}))^{S_k}$$

Finally, set the map $$\Phi: V^{\otimes n} = \bigoplus_{J \subset \{1,..,n\}} U^J \longrightarrow \bigoplus_{k=0,...,n} (U^{\otimes k} \otimes \bC Inj(\{1,...,k\} , \{1,...,n\}))^{S_k}$$ to be $\sum_{J \subset \{1,..,n\}} \Phi_{J}$.

Notice that $\Phi(\triv \otimes \triv \otimes ... \otimes \triv) = \Phi_{\emptyset}(\triv \otimes \triv \otimes ... \otimes \triv) = 1$.

We claim that $\Phi$ is a map of $\gl(V)$-modules. Again, we consider the decomposition $$\gl(V) \cong \bC \id_V \oplus \mathfrak{u}_{\mathfrak{p}}^{-} \oplus \mathfrak{u}_{\mathfrak{p}}^{+} \oplus \gl(U)$$
\begin{itemize}[leftmargin=*]
 \item $\id_V$ acts by the scalar $n$ on both sides.
 \item Let $u \in  U \cong \mathfrak{u}_{\mathfrak{p}}^{-}$, and let $v_{1} \otimes ...\otimes v_{n} \in U^{\otimes \{j_1 < j_2 < ...<j_k\}}$ 
 
 Then $u$ acts on $ V^{\otimes n}$ by operator $F_u$ which satisfies:
 $$F_u.(v_{1} \otimes ...\otimes v_{n}) = \sum_{i \notin J} v_1 \otimes ... \otimes \InnaB{\triv^*(v_i)} u \otimes ...\otimes v_n$$
and thus 
$$\Phi(F_u.(v_{1} \otimes ...\otimes v_{n})) = e_{S_{k+1}}\left(\sum_{ 1\leq l \leq k+1} \sum_{\substack {g \in Inj(\{1,...,k+1\} , \{1,...,n\}):\\ g \circ \iota_l =f_{j_1 < j_2 < ...<j_k}\\ f(l-1) < g(l) < f(l) }}  u^{(l)}.(v_{j_1} \otimes v_{j_2} \otimes ...\otimes v_{j_k}) \otimes g \right)\cdot \prod_{i \notin J} \InnaB{\triv^*(v_i)}
$$
Here $\iota_l$ is the injection 
$$ \{1,...,k\}\hookrightarrow \{1,...,k+1\}, \, i \mapsto \begin{cases} i &\text{ if } i < l \\
i+1 &\text{ if } i \geq l                                                                                                  \end{cases}$$
 Now,
 \begin{align*}
  &F_u.\Phi(v_{1} \otimes ...\otimes v_{n}) = F_u.\left(e_{S_k}(v_{j_1} \otimes v_{j_2} \otimes ...\otimes v_{j_k} \otimes f_{j_1 < j_2 < ...<j_k}) \cdot \prod_{i \notin J} \InnaB{\triv^*(v_i)} \right) =\\
  &= \frac{1}{(k+1)!}  \prod_{i \notin J} \InnaB{\triv^*(v_i)} \left(\sum_{ 1\leq l \leq k+1} \sum_{\sigma \in S_k}  (u^{(l)}\circ \sigma)(v_{j_1} \otimes v_{j_2} \otimes ...\otimes v_{j_k}) \otimes ( \mathtt{res}_l^*\circ \sigma)(f_{j_1 < j_2 < ...<j_k})\right)
 \end{align*}
 We now use Lemma \ref{applem:action_on_usual_tensor_power1} to conclude that $\Phi$ is a map of $\mathfrak{u}_{\mathfrak{p}}^{-}$-modules. 

\item Let $\lambda \in  U^* \cong \mathfrak{u}_{\mathfrak{p}}^{+}$ , $k \geq 1$, and $v_{1} \otimes ...\otimes v_{n} \in U^{\otimes \{j_1 < j_2 < ...<j_k\}}$.
 
 Then $\lambda$ acts on $ V^{\otimes n}$ by operator $E_{\lambda}$ which satisfies:
 $$E_{\lambda}.(v_{1} \otimes ...\otimes v_{n}) = \sum_{j \in J} v_1 \otimes ... \otimes \lambda(v_j) \otimes ...\otimes v_n$$
and thus 
$$\Phi(E_{\lambda}.(v_{1} \otimes ...\otimes v_{n})) = e_{S_{k-1}} \left(\sum_{ 1\leq l \leq k} \lambda^{(l)}.(v_{j_1} \otimes v_{j_2} \otimes ...\otimes v_{j_k}) \otimes \mathtt{res}_l(f_{j_1 < j_2 < ...<j_k}) \right)\cdot \prod_{i \notin J} \InnaB{\triv^*(v_i)}$$
 Now,
 \begin{align*}
  &E_{\lambda}.\Phi(v_{1} \otimes ...\otimes v_{n}) = E_{\lambda}.\left(e_{S_k}(v_{j_1} \otimes v_{j_2} \otimes ...\otimes v_{j_k} \otimes f_{j_1 < j_2 < ...<j_k}) \cdot \prod_{i \notin J} \InnaB{\triv^*(v_i)} \right) =\\
  &= \frac{1}{k!} \prod_{i \notin J} \InnaB{\triv^*(v_i)} \left(\sum_{ 1\leq l \leq k} \sum_{\sigma \in S_k}  (\lambda^{(l)}\circ \sigma)(v_{j_1} \otimes v_{j_2} \otimes ...\otimes v_{j_k}) \otimes ( \mathtt{res}_l\circ \sigma)(f_{j_1 < j_2 < ...<j_k})\right)
 \end{align*}

 We now use Lemma \ref{applem:action_on_usual_tensor_power2} to conclude that $\Phi$ is a map of $\mathfrak{u}_{\mathfrak{p}}^{+}$-modules (note that the action of $\lambda$ on $(U^{\otimes 0} \otimes \bC)^{S_0} \cong \bC$ is zero, as is on $\InnaB{\bC \triv} \cong U^{\otimes \emptyset}$). This Lemma is proved at the end of this section.

 \item $ \gl(U)$ acts naturally on each summand $U^{\otimes \{j_1 < j_2 < ...<j_k\}}$ on the left and on each summand $(U^{\otimes k} \otimes \bC Inj(\{1,...,k\} , \{1,...,n\}))^{S_k}$ on the right, and this action gives us isomorphisms of $\gl(U)$-modules: 
 $$\Phi_{J}:U^{\otimes \{j_1 < j_2 < ...<j_k\}} \rightarrow U^{\otimes k} \otimes \bC f_{j_1 < j_2 < ...<j_k}$$ and
 $$\bigoplus_{\{j_1 < j_2 < ...<j_k\} \subset \{1,...,n\}} U^{\otimes k} \otimes \bC f_{j_1 < j_2 < ...<j_k} \cong (U^{\otimes k} \otimes \bC Inj(\{1,...,k\} , \{1,...,n\}))^{S_k}  $$
\end{itemize}

Note that the last argument also shows that $\Phi$ is an isomorphism.

\InnaA{It remains to check that $\Phi$ is also a morphism of $S_n$-modules. 
Fix $k \in \{0, ..., n\}$.
It is enough to check that 
\begin{align*}
\bigoplus_{\substack{J \subset \{1, ..., n\},\\ \abs{J}=k}} \Phi_{J}: \bigoplus_{\substack{J \subset \{1, ..., n\},\\ \abs{J}=k}} U^{\otimes \{j_1 < j_2 < ...<j_k\}} &\longrightarrow (U^{\otimes k} \otimes \bC Inj(\{1,...,k\} , \{1,...,n\}))^{S_k}\\
v_{1} \otimes ...\otimes v_{n} &\mapsto  e_{S_k}(v_{j_1} \otimes v_{j_2} \otimes ...\otimes v_{j_k} \otimes f_{j_1 < j_2 < ...<j_k}) \cdot \prod_{i \notin J} \InnaB{\triv^*(v_i)}
\end{align*}
is a morphism of $S_n$-modules (here $J = \{j_1 < j_2 < ...<j_k\}$).

Fix $\sigma \in S_n$, and fix $J = \{j_1 < j_2 < ...<j_k\} \subset \{1, ..., n\}$. Let $\tau \in S_k$ be such that 
$\sigma(j_{\tau^{-1}(1)}), \sigma(j_{\tau^{-1}(2)}), ..., \sigma(j_{\tau^{-1}(k)})$ is a monotone increasing sequence. We will denote this sequence by $\sigma(J)$.

We have:
\begin{align*}
\sigma(\Phi_{J}( v_{1} \otimes ...\otimes v_{n})) = e_{S_k}(v_{j_1} \otimes v_{j_2} \otimes ...\otimes v_{j_k} \otimes (\sigma \circ f_{j_1 < j_2 < ...<j_k})) \cdot \prod_{i \notin J} \InnaB{\triv^*(v_i)}
\end{align*}
On the other hand, 
\begin{align*}
&\Phi_{\sigma(J)}(\sigma( v_{1} \otimes ...\otimes v_{n})) = \Phi_{\sigma(J)}( v_{\sigma^{-1}(1)} \otimes ...\otimes v_{\sigma^{-1}(n)}) =\\
&=e_{S_k}\left(v_{j_{\tau^{-1}(1)}} \otimes v_{j_{\tau^{-1}(2)}} \otimes ...\otimes v_{j_{\tau^{-1}(k)}} \otimes f_{j_{\sigma(\tau^{-1}(1)} )< \sigma(j_{\tau^{-1}(2)}) < ...<\sigma(j_{\tau^{-1}(k)})}\right) \cdot \prod_{i \notin J} \InnaB{\triv^*(v_i)}=\\
&=e_{S_k}\left(\tau.(v_{j_1} \otimes v_{j_2} \otimes ...\otimes v_{j_k}) \otimes (\sigma \circ f_{j_1 < j_2 < ...<j_k} \circ \tau^{-1}) \right) \cdot \prod_{i \notin J} \InnaB{\triv^*(v_i)} = \\
&= e_{S_k}(v_{j_1} \otimes v_{j_2} \otimes ...\otimes v_{j_k} \otimes (\sigma \circ f_{j_1 < j_2 < ...<j_k})) \cdot \prod_{i \notin J} \InnaB{\triv^*(v_i)}
\end{align*}

Thus $$\sigma \circ \left(\bigoplus_{J \subset \{1, ..., n\}, \abs{J}=k} \Phi_{J} \right) = \left( \bigoplus_{J \subset \{1, ..., n\}, \abs{J}=k} \Phi_{J} \right) \circ \sigma$$ and we are done.
}
\end{proof}

}

\newpage
\section{Definition of a complex tensor power: $\nu \notin \bZ_+$}\label{app:complex_ten_power_generic_nu}

 Let $\nu \in \bC \setminus \bZ_+$, and let $(V, \InnaB{\triv})$ be a finite-dimensional unital vector space.
 \InnaB{In this section we discuss an alternative definition of the $\nu$-th tensor power of $(V, \InnaB{\triv})$, which does not use a splitting of $V$, but is applicable only when $\nu \notin \bZ_+$.}

 As before, we use Notation \ref{notn:par_subalg} for the parabolic subalgebra $\mathfrak{p}$, its ``mirabolic'' subalgebra $\bar{\mathfrak{p}}_{\InnaB{\bC \triv}}$ and the nilpotent subalgebra $\mathfrak{u}_{\mathfrak{p}}^{+}$. 
 
 We will also denote by $\mathfrak{p}_{0}$ the subalgebra of $\mathfrak{p}$ consisting of all the endomorphisms $\phi: V \rightarrow V$ for which $\Im \phi \subset \InnaB{\bC \triv}$ (notice that $\mathfrak{u}_{\mathfrak{p}}^{+}= \mathfrak{p}_{0} \cap \bar{\mathfrak{p}}_{\InnaB{\bC \triv}}$).
 $$ $$
The quotient $\mathfrak{l}:=\quotient{\mathfrak{p}}{\mathfrak{u}_{\mathfrak{p}}^{+}}$ is then a reductive Lie algebra which can be decomposed as a direct sum of reductive Lie algebras $\mathfrak{l} \cong \mathfrak{l}_1 \oplus \mathfrak{l}_2$, where $\mathfrak{l}_1:=\quotient{\mathfrak{p}_{0}}{\mathfrak{u}_{\mathfrak{p}}^{+}}$ is a one-dimensional Lie algebra, and $\mathfrak{l}_2:=\quotient{\bar{\mathfrak{p}}_{\InnaB{\bC \triv}}}{\mathfrak{u}_{\mathfrak{p}}^{+}}$.

 Notice that the Lie algebra $\mathfrak{p}$ has a one-dimensional center ($Z(\mathfrak{p}) = \bC \id_V$), and so does $\mathfrak{l}$. In fact, we have a canonical splitting $\mathfrak{p}\cong \bC \id_V \oplus \bar{\mathfrak{p}}_{\InnaB{\bC \triv}}$, and thus a canonical splitting $\mathfrak{l} \cong Z(\mathfrak{l}) \oplus \mathfrak{l}_2$.

 Consider the quotient space $U := \quotient{V}{\InnaB{\bC \triv}}$. Since both $V$ and $\InnaB{\bC \triv}$ are $\mathfrak{p}$-modules, $U$ also has the structure of a $\mathfrak{p}$-module. Moreover, $\mathfrak{p}_{0}$ obviously acts trivially on $U$, so the action of $\mathfrak{p}$ on $U$ factors through an action of $\mathfrak{l}_2$ on $U$.

  We now give the following definition of a parabolic Verma module for $(\gl(V), \mathfrak{p})$ (this is actually the parabolic Verma module of highest weight $(\nu-\abs{\lambda}, \lambda)$):
 \begin{definition}\label{def:par_Verma_mod_indep}
 Let $\lambda$ be a Young diagram.

 If $\ell(\lambda) > \dim V  -1$, we define the parabolic Verma module $M_{\mathfrak{p}}(\nu-\abs{\lambda}, \lambda)$ to be zero.

 Otherwise, consider the $\mathfrak{l}_2$-module $ S^{\lambda} U$ (i.e. the Schur functor $S^{\lambda}$ applied to the $\mathfrak{l}_2$-module $U$). We make $S^{\lambda} U$ a $\mathfrak{l}$-module by requiring that $Z(\mathfrak{l})$ act on $S^{\lambda} U$ by the scalar $\nu$, and then lift the action of $\mathfrak{l}$ on $S^{\lambda} U$ to an action of $\mathfrak{p}$ on $S^{\lambda} U$ by requiring that $\mathfrak{u}_{\mathfrak{p}}^{+}$ act trivially on $S^{\lambda} U$.

Finally, we define
  $$M_{\mathfrak{p}}(\nu-\abs{\lambda}, \lambda):=\mathcal{U}(\gl(V)) \otimes_{\mathcal{U}(\mathfrak{p})} S^{\lambda} U$$
 \end{definition}

 \begin{remark}
  Recall that the $\gl(V)$-module $M_{\mathfrak{p}}(\nu-\abs{\lambda}, \lambda)$ is irreducible, since $\nu \notin \bZ_+$ (c.f. Proposition \ref{prop:par_cat_O_Vermas}).
 \end{remark}

The following definition is proposed in \cite{E1} (we still assume that $\nu \notin \bZ_+$).

\begin{definition}\label{def:complex_ten_power_generic_nu}
Define the object $V^{\otimes^{Del'} \nu}$ of $Ind-(\underline{Rep}^{ab}(S_{\nu}) \boxtimes \co^{\mathfrak{p}}_{\nu, V})$ through the formula
 $$V^{\otimes^{Del'} \nu}:= \bigoplus_{\lambda \InnaA{\text{ is a Young diagram}}} X_{\lambda} \otimes M_{\mathfrak{p}}(\nu-\abs{\lambda}, \lambda)$$
\end{definition}

\begin{proposition}\label{prop:equiv_def_tensor_powers}
Fix a splitting $V \cong \InnaB{\bC \triv} \oplus U$.

The object $V^{\otimes^{Del'} \nu}$ defined in Definition \ref{def:complex_ten_power_generic_nu} is isomorphic to the object $V^{\underline{\otimes}  \nu}$ defined in Definition \ref{def:complex_ten_power_splitting}.
\end{proposition}

\begin{proof}

Since we assumed $\nu \notin \bZ_+$, the categories $\underline{Rep}(S_{\nu}), \co^{\mathfrak{p}}_{\nu, V}$ are semisimple abelian categories (see Sections \ref{sec:Del_cat_S_nu}, \ref{sec:par_cat_o}). In this case, any object $A$ of $Ind-(\underline{Rep}^{ab}(S_{\nu}) \boxtimes \co^{\mathfrak{p}}_{\nu, V})$ can be written as a direct sum with summands of the form $L(\nu-\abs{\lambda}, \lambda) \otimes X_{\mu} $ ($\lambda, \mu$ are Young diagrams).

In the case of the object $V^{\underline{\otimes}  \nu}$ of $Ind-(\underline{Rep}^{ab}(S_{\nu}) \boxtimes \co^{\mathfrak{p}}_{\nu, V})$, we get:
$$V^{\underline{\otimes}  \nu} \cong \bigoplus_{\lambda, \mu} L(\nu-\abs{\lambda}, \lambda) \otimes X_{\mu} \otimes Mult_{\lambda, \mu}$$
(here $Mult_{\lambda, \mu}$ is the multiplicity space of $L(\nu-\abs{\lambda}, \lambda) \otimes X_{\mu} $ in $V^{\underline{\otimes}  \nu}$, not necessarily finite dimensional). 

\InnaA{Recall from Section \ref{sec:par_cat_o} that for any Young diagram $\lambda$, we have an isomorphism of $\gl(V)$-modules $$L(\nu-\abs{\lambda}, \lambda) \cong M_{\mathfrak{p}}(\nu-\abs{\lambda}, \lambda)$$

We now need to prove that $\dim Mult_{\lambda, \mu} = \delta_{\lambda, \mu}$, and we are done.}

To do this, consider
$$V^{\underline{\otimes}  \nu}  \cong \bigoplus_{\lambda, \mu} L(\nu-\abs{\lambda}, \lambda) \otimes X_{\mu} \otimes Mult_{\lambda, \mu}$$ as an object of $Ind-\underline{Rep}(S_{\nu})$ with an action of $\gl(U)$. Using Lemmas \ref{lem:hom_X_tau_Delta_k}, \ref{lem:gl_u_struct_o_cat}, we get the following decompositions:
$$V^{\underline{\otimes}  \nu} \cong \bigoplus_{\mu} \bigoplus_{\rho \in \mathcal{I}^{+}_{\mu} } S^{\rho} U \otimes X_{\mu} $$
and
$$\bigoplus_{\lambda, \mu} L(\nu-\abs{\lambda}, \lambda) \otimes X_{\mu} \otimes Mult_{\lambda, \mu}\cong \bigoplus_{\lambda, \mu} \bigoplus_{\rho \in \mathcal{I}^{+}_{\lambda} } S^{\rho} U \otimes X_{\mu} \otimes Mult_{\lambda, \mu}$$

Thus for any Young diagram $\mu$, we have $$ \bigoplus_{\rho \in \mathcal{I}^{+}_{\mu} } S^{\rho} U \cong \bigoplus_{\lambda} \bigoplus_{\rho \in \mathcal{I}^{+}_{\lambda} } S^{\rho} U \otimes Mult_{\lambda, \mu}$$
and we immediately conclude that $\dim Mult_{\lambda, \mu} = \delta_{\lambda, \mu}$, proving the statement of the proposition.
\end{proof}

\end{document}